\documentclass[11pt]{article}
\usepackage[sort, numbers]{natbib}
\usepackage{xparse}
\NewDocumentCommand{\ceil}{s O{} m}{%
  \IfBooleanTF{#1} 
    {\left\lceil#3\right\rceil} 
    {#2\lceil#3#2\rceil} 
}

\usepackage{amsfonts,amsmath,graphicx,fullpage,bbm,amsthm}
\usepackage{amssymb,multirow,verbatim}
\usepackage{acronym,wrapfig,plain,mathtools,enumerate,relsize,color}
\usepackage{lipsum,latexsym,tikz,graphicx,float,colordvi,pgfplots}
\newcommand{\tabincell}[2]{\begin{tabular}{@{}#1@{}}#2\end{tabular}}
\usetikzlibrary{calc,shapes.geometric,positioning,arrows,intersections}
\newtheorem{theorem}{Theorem }
\newtheorem{assumption}{Assumption}[]

\newtheorem{lemma}{Lemma}[]
\newtheorem{definition}{Definition}[section]

\newtheorem{remark}{Remark}
\newtheorem{proposition}{Proposition}[]

\usepackage{algorithm}
\usepackage{amsmath}
\usepackage{subfig}
\usepackage{amssymb}
\usepackage{pifont}
\usepackage{algorithmic}
\usepackage{boxedminipage}
\usepackage[justification=centering]{caption}
\def\bko{{\rm 1\kern-.17em l}}
\usepackage{wrapfig}
\usepackage{array}
\newcolumntype{C}[1]{>{\centering\arraybackslash}p{#1}}

\newcommand{\an}[1]{{\color{black}#1}}
\newcommand{\fy}[1]{{\color{black}#1}}
\newcommand{\us}[1]{{\color{black}#1}}
\newcommand{\vs}[1]{{\color{black}#1}}
\newcommand{\vvs}[1]{{\color{black}#1}}

\newcommand{\ssc}[1]{{\color{black}#1}}
\newcommand{\ic}[1]{{\color{black}#1}}
\newcommand{\avs}[1]{{\color{black}#1}}

\newcommand\mtiny[1]{\mbox{\tiny\ensuremath{#1}}}

\def\Fscr{{\cal F}}
\def\Kscr{{\cal K}}

\def\be{\begin{enumerate}}
\def\ee{\end{enumerate}}

\def\st{\mbox{subject to}}

 \newcommand{\remove}[1]{}
\newcommand{\EXP}[1]{\mathsf{E}\!\left[#1\right] }

\newcommand{\pmat}[1]{\begin{pmatrix} #1 \end{pmatrix}}

\def\sF{\mathcal{F}}

\def\Real{\mathbb{R}}
\def\g{\gamma}

\def\e{\epsilon}

\title{On the analysis of variance-reduced and randomized projection variants of single projection schemes  for monotone stochastic variational inequality problems}
\author{Shisheng Cui and   
		Uday V.~Shanbhag\footnote{Cui and Shanbhag are with the Department of Indust. and
			Manuf. Engg., Penn. State Univ.,
						  University Park, PA 16802, USA. They are
							  reachable at 
\texttt{\small{suc256,udaybag@psu.edu}}. The authors have been partly funded by NSF Grants 1246887 and  1400217 as well as the Gary and Sheila Bello Chair funds. A preliminary conference version of this work appeared in~\cite{CuiS16}. }}

\begin{document}
\maketitle

\begin{abstract}
Classical extragradient schemes and their stochastic counterpart represent a
cornerstone for resolving monotone variational inequality problems.  Yet, such
schemes have a per-iteration complexity of two projections {onto} a convex set
and require two evaluations of the map, the former of which could be relatively
expensive if $X$ is a complicated set. We consider two related avenues where
the per-iteration complexity is significantly reduced: (i) A stochastic
projected reflected gradient ({\bf SPRG}) method requiring a single evaluation
of the map and a single projection; and (ii) A stochastic subgradient
extragradient ({\bf SSE}) method that requires two evaluations of the map, a
single projection onto $X$, and a significantly cheaper projection (onto a
halfspace) computable in closed form. Under a variance-reduced framework
reliant on a sample-average of the map based on an increasing batch-size, we
prove almost sure (a.s.) convergence of the iterates to a random point in the
solution set for both schemes. Additionally, both schemes display a
non-asymptotic rate of $\mathcal{O}(1/K)$ {in terms of the gap function} where
$K$ denotes the number of iterations; notably, both rates match those obtained
in deterministic regimes.  {To address feasibility sets given by the
intersection of a large number of convex constraints, we adapt both of the
aforementioned schemes to a random projection framework.} We then show that
\vvs{the} random projection analogs of both schemes also display a.s.
convergence under a weak-sharpness requirement; furthermore, without imposing
the weak-sharpness requirement, both schemes are characterized by a provable
rate of $\mathcal{O}(1/\sqrt{K})$  in terms of the gap function of the
projection of the averaged sequence onto $X$ as well as the infeasibility of
this sequence.  Preliminary numerics support theoretical findings and the
schemes outperform standard extragradient schemes in terms of the per-iteration
complexity. 
 \end{abstract}

\allowdisplaybreaks
\section{Introduction}
This paper considers the solution of stochastic variational inequality
problems, a stochastic generalization of the variational inequality
problem. Given a set $X \subseteq \Real^n$ and a map
$F:\Real^n \to \Real^n$, the variational inequality problem VI$(X,F)$
requires finding a point $x^*\in X$ such that
\begin{align} \tag{VI$(X,F)$}
\notag F(x^*)^T(x-x^*)\ge0,\quad\forall x \in X.
\end{align}
In the stochastic generalization, the components of the map $F$ are
expectation-valued; specifically $F_i(x)\triangleq
\mathbb{E}[F_i(x,\xi(\omega))]$, where $\xi:\Omega \to \Real^d$ is a
random variable, $F_i: \Real^n \times \Real^d \to \Real$ is a
single-valued function, and the $\mathbb{E}[\cdot]$ denotes the
expectation and the associated probability space being denoted by
$(\Omega, {\cal F}, \mathbb{P})$. In short, we are interested in a
vector $x^* \in X$ such that 
\begin{align} \tag{SVI$(X,F)$}
\mathbb{E}[F(x^*,\omega)]^T(x-x^*) \geq 0, \quad \forall x \in X, 
\end{align}
where $\mathbb{E}[F(x,\omega)] =
\pmat{\mathbb{E}[F_i(x,\omega)]}_{i=1}^K.$ {The variational inequality problem is an immensely relevant problem that
finds application in engineering, economics, and applied sciences
(cf.~\cite{facchinei2007finite,rockafellar2009variational,iusem2018incremental,chen2017accelerated}).
Increasingly, the stochastic generalization is of relevance and has found
application in the study of a broad class of equilibrium problems under
uncertainty. Of these, sample average approximation (SAA) scheme solves the
expected value of the stochastic mapping which is approximated via the average
over a large number of samples
(cf.~\cite{birge2011introduction,shapiro2014lectures,chen2012stochastic,xu2010sample}).
A counterpart to SAA schemes is the stochastic approximation (SA) methods where
at each iteration, a sample of the stochastic mapping is used
(cf.~\cite{kannan2013addressing,ravat2011characterization,juditsky2011solving}).
Amongst the simplest of SA schemes are analogs of the standard projection-based
schemes, which we review next.

\subsection{Projection-based schemes and their variants} 
Given an $x_0 \in X$, the projection-based scheme (PG) generates a sequence $\{x_k\}$, where 
\begin{align}\tag{PG}
\notag x_{k+1}\coloneqq\Pi_X(x_k-\gamma F(x_k)),
\end{align}
$\Pi_X(y)$ denotes the projection of $y$ onto $X$ and $\gamma$
denotes a suitably small steplength. This method generally requires a
strong monotonicity assumption on $F$ to ensure convergence. An
extension referred to as the extragradient scheme, suggested by Antipin
\cite{antipin1978method} and Korpelevich
\cite{korpelevich1976extragradient}, required that $F$
be merely monotone and Lipschitz continuous over the set $X$. \vvs{However,} this scheme requires two projection steps, as captured by (EG).
\begin{align}
\tag{EG}
\begin{aligned}
\notag x_{k+\frac{1}{2}}&\coloneqq\Pi_X(x_k-\gamma F(x_k)), \\
\notag x_{k+1}&\coloneqq\Pi_X(x_k-\gamma F(x_{k+\frac{1}{2}})).
\end{aligned}
\end{align}
Naturally, when the set $X$ is not necessarily a {\em simple} set, this
projection operation by no means cheap. There have been several schemes in
which merely monotone variational inequality problems can be addressed by
taking a single projection operation and we consider two instances.  In
recent work, a {\em projected reflected gradient} (PRG) method was
proposed by Malitsky~\cite{malitsky2015projected}, requiring a {\bf
single},
	rather than {\bf two}, projections:
\begin{align}
\notag x_{k+1}\coloneqq\Pi_X(x_k-\gamma_kF(2x_k-x_{k-1})). \tag{PRG}
\end{align}
Intuitively, this scheme has a similar structure to the projected
	gradient scheme taking a form with the following key distinction:
		Rather than evaluating the map at $x_k$ (as in (PG)), the map is
evaluated at the reflection of $x_{k-1}$ in $x_k$ which is $x_k -
(x_{k-1}-x_k) = 2x_k - x_{k-1}$. Remarkably, this simple modification
allows for proving convergence of this scheme for merely monotone
Lipschitz continuous maps~\cite{malitsky2015projected}.
Malitsky~\cite{malitsky2015projected} derived  the rate of convergence
of the sequence under a strong monotonicity assumption of the map.  An
alternate modification of the extragradient method was proposed by Censor,
Gibali and Reich and was referred to as the {\em subgradient extragradient
method} (SE) \cite{censor2011subgradient}:
\begin{align}
\begin{aligned}
x_{k+\frac{1}{2}}&\coloneqq\Pi_X(x_k-\gamma_kF(x_k)), \\
x_{k+1}&\coloneqq\Pi_{C_k}(x_k-\gamma_kF(x_{k+\frac{1}{2}})),
\end{aligned} \tag{SE}
\end{align}
where $C_k \triangleq \{\ssc{y} \in \mathbb{R}^n \mid (x_k-\gamma_kF(x_k)-x_{k+\frac{1}{2}})^T(\ssc{y}-x_{k+\frac{1}{2}})\leq0\}$.
In (SE), the two projections are replaced by a projection onto the set and a
second projection onto a halfspace, the latter of which is
computable in closed form. However, no rate of convergence has been
provided in their analysis. A third scheme that employs a single
projection to contend with merely monotone maps is the iterative Tikhonov
regularization (ITR) scheme, a regularized variant of (PG) in which $x_{k+1}$
is updated as per \begin{align}\tag{ITR}
\notag x_{k+1}\coloneqq\Pi_X(x_k-\gamma_k( F(x_k)+\epsilon_k x_k)),
\end{align}
where the steplength sequence $\{\gamma_k\}$ and the regularization sequence $\{\epsilon_k\}$ are suitably chosen positive diminishing sequences~\cite{konnov07equilibrium,yin11nash,kannan11distributed}.

\subsection{Stochastic variational inequality problems} There have been
schemes analogous to (PG) and (EG) in this regime with the key distinction
that an evaluation of the map, namely $F(x_k)$, is replaced by
$F(x_k,\omega_k)$, in the spirit of stochastic
approximation~\cite{robbins51sa}.  A simple stochastic extension
of the standard projection scheme for VI$(X,F)$ leads to a stochastic
approximation scheme~\cite{robbins51sa}:
\begin{align}
\notag x_{k+1} \coloneqq \Pi_X(x_k-\gamma_kF(x_k,\omega_k)). \tag{SPG}
\end{align}
Similarly, an extragradient counterpart to (EG) is (SEG) and is defined below:
\begin{align}
\tag{SEG}
\begin{aligned}
\notag x_{k+\frac{1}{2}}&\coloneqq\Pi_X(x_k-\gamma_kF(x_k,\omega_k)), \\ 
x_{k+1}&\coloneqq\Pi_X(x_k-\gamma_kF(x_{k+\frac{1}{2}},\omega_{k+\frac{1}{2}})). 
\end{aligned}
\end{align}
Jiang and Xu \cite{jiang08stochastic}
appear amongst the first who applied SA methods to solve stochastic
variational inequality problems. An extension of ITR to address merely monotone
stochastic VIs was presented by Koshal, Nedi\'c, and Shanbhag~\cite{koshal2013regularized}.
A regularized smoothing SA method to address nonsmooth stochastic Nash equilibrium problems that lead to stochastic VIs with possibly
non-Lipschitzian and merely monotone mappings was proposed in
\cite{yousefian2013regularized}. There has also been a development of prox-based generalization} of SA methods were developed
(cf.~\cite{nemirovski2009robust,yousefian17smoothing,yousefian2014optimal,yousefian2013self,nemirovski2004prox})
for solving smooth and nonsmooth stochastic convex optimization problems
and variational inequality problems. There has also been an effort to develop block-based schemes for Cartesian stochastic variational inequality problems~\cite{yousefian18stochastic}. 
Fig.~\ref{EG} illustrates {the} (SEG) scheme.
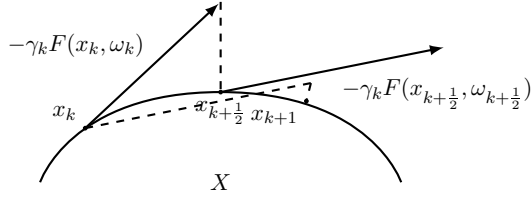
\begin{figure}[htbp]
\centering
\begin{tikzpicture}[thick,scale=.6,every node/.style={scale=.8}]
\node (o) at (0,0) {$X$};
\fill (-3,1.2)  circle[radius=1.5pt];
\node (d) at (-4,2) {};
\node (e) at (4,2) {};
\node (x) at (-4,0) {};
\node (y) at (0,2) {};
\node (z) at (4,0) {};
\fill (1.9,1.8)  circle[radius=1.5pt];
\fill (0,2)  circle[radius=1.5pt];
\draw[thick] plot[smooth,tension=1.5] coordinates {(x) (y) (z)};
\draw [-latex, black, thick, shorten >= 1] (-3,1.2) node[above left] {$x_k$} -- (0,4) node[midway,above left] {$-\gamma_kF(x_k,\omega_k)$};
\draw [dashed] (0,4) -- (0,2) node[below] {$x_{k+\frac{1}{2}}$};
\draw [dashed] (2,2.2) -- (1.9,1.8) node [below left] {$x_{k+1}$};
\draw [-latex, black, thick, shorten >= 1] (0,2) -- (5,3) node[midway,below right] {$-\gamma_kF(x_{k+\frac{1}{2}},\omega_{k+\frac{1}{2}})$};
\draw [dashed] (-3,1.2) -- (2,2.2) ;
\end{tikzpicture}
\caption{Stochastic extragradient {scheme} (SEG)}
\label{EG}
\end{figure}
Extragradient-based schemes (and their stochastic mirror-prox
counterparts) represent amongst the simplest of {the non-regularized} schemes for monotone
SVIs (cf.~\cite{dang2015convergence,juditsky2011solving}). However, each iteration requires
{\bf two} projection steps, rather than one (as in (SPG)). We summarize much of the prior results in Table \ref{sa}. Given that
this class of Monte-Carlo approximation schemes routinely requires 10s
or 100s of thousands of steps, our interest lies in ascertaining whether
projection-based schemes can be developed requiring a single projection
step per iteration, {\em reducing the per-iteration complexity by a factor of two}. We consider two such schemes given a random point
$x_0 \in X$: \\(i) {\em Stochastic
projected reflected gradient schemes} (SPRG). 
\begin{align}
\notag x_{k+1}\coloneqq\Pi_X(x_k-\gamma_kF(2x_k-x_{k-1},\omega_k));
	\tag{\bf SPRG}
\end{align}
and (ii) {\em Stochastic subgradient extragradient schemes} (SSE).
\begin{align}
\begin{aligned}
x_{k+\frac{1}{2}}&\coloneqq\Pi_X(x_k-\gamma_kF(x_k,\omega_k)), \\
x_{k+1}&\coloneqq\Pi_{C_k}(x_k-\gamma_kF(x_{k+\frac{1}{2}},\omega_{k+\frac{1}{2}})),
\end{aligned} \tag{\bf SSE}
\end{align}
where $C_k \triangleq \{y \in \mathbb{R}^n \mid (x_k-\gamma_kF(x_k,\omega_k)-x_{k+\frac{1}{2}})^T(y-x_{k+\frac{1}{2}})\leq0\}$. Clearly, the second projection is a simple optimization problem {solvable in closed form}. Solving for $x_{k+1}$, we could obtain an equivalent scheme which requires a single projection (the proof is in appendix).  Fig. \ref{SSE} illustrate the steps of these schemes.
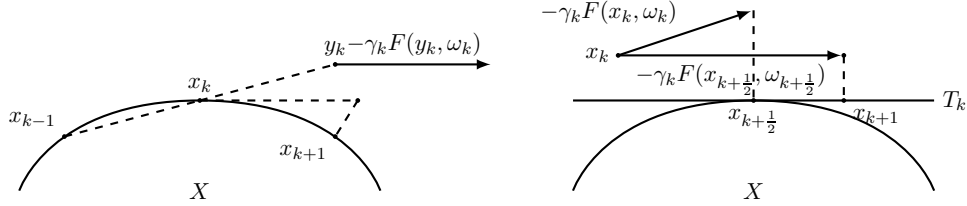
\begin{figure}[!h]
\centering
\begin{tikzpicture}[thick,scale=.6,every node/.style={scale=.8}]
\node (o) at (0,0) {$X$};
\fill (-3,1.2)  circle[radius=1.5pt];
\fill (3,2.8)  circle[radius=1.5pt];
\fill (0,2)  circle[radius=1.5pt];
\node (d) at (-4,2) {};
\node (e) at (4,2) {};
\fill (3.5,2)  circle[radius=1.5pt];
\fill (3,1.2)  circle[radius=1.5pt];
\node (g) at (2,2) {};
\node (x) at (-4,0) {};
\node (y) at (0,2) {};
\node (z) at (4,0) {};
\draw[thick] plot[smooth,tension=1.5] coordinates {(x) (y) (z)};
\draw [dashed] (-3,1.2) node[above left] {$x_{k-1}$} -- (0,2) node[above] {$x_k$};
\draw [dashed] (0,2) -- (3,2.8) node[above] {$y_k$};
\draw [dashed] (0,2) -- (3.5,2);
\draw [-latex, black, thick, shorten >= 1] (3,2.8) -- (6.5,2.8) node[midway,above] {$-\gamma_kF(y_k,\omega_k)$};
\draw [dashed] (3.5,2) -- (3,1.2) node [below left] {$x_{k+1}$};
\end{tikzpicture}
\quad
\begin{tikzpicture}[thick,scale=.6,every node/.style={scale=.8}]
\node (o) at (0,0) {$X$};
\fill (-3,3)  circle[radius=1.5pt];
\fill (0,2)  circle[radius=1.5pt];
\node (d) at (-4,2) {};
\node (e) at (4,2) {};
\fill (2,3)  circle[radius=1.5pt];
\fill (2,2)  circle[radius=1.5pt];
\node (x) at (-4,0) {};
\node (y) at (0,2) {};
\node (z) at (4,0) {};
\draw[thick] plot[smooth,tension=1.5] coordinates {(x) (y) (z)};
\draw [-latex, black, thick, shorten >= 1] (-3,3) node[left] {$x_k$} -- (0,4) node[midway, above left] {$-\gamma_kF(x_k,\omega_k)$};
\draw [dashed] (0,4) -- (0,2) node[below] {$x_{k+\frac{1}{2}}$};
\draw (-4,2) -- (4,2) node[right] {$T_k$};
\draw [-latex, black, thick, shorten >= 1] (-3,3) -- (2,3) node[midway,below] {$-\gamma_kF(x_{k+\frac{1}{2}},\omega_{k+\frac{1}{2}})$};
\draw [dashed] (2,3) -- (2,2) node[below right] {$x_{k+1}$};
\end{tikzpicture}
\caption{Left: (SPRG); Right: (SSE)}
\label{SSE}
\end{figure}

\begin{table}[!htb]
\scriptsize
\caption{A review of stochastic approximation schemes for SVIs}
\vspace{-0.3in}
\begin{center}
    \begin{tabular}[t]{ l | l | l | l | l | l | l | c}
    \hline
     Ref. & Applicability & Compact & Avg. & Metric & Rate & A.s. & \# proj. \\ \hline
    \cite{jiang08stochastic} & {Strongly monotone, Lipschitz} & N & N & Iterates & - & Y & 1 \\ 
    \cite{koshal2013regularized} & Monotone, Lipschitz & {N} & N & Iterates & - & Y  & {1}\\ 
    \cite{yousefian2013regularized} & Monotone, non-Lip. & {N} & N & Iterates & - & Y  & {1}\\ 
    \cite{juditsky2011solving} & Monotone, non-Lip. & Y & Y & Gap fn. & $\mathcal{O}(1/\sqrt{K})$ & N  & {1}\\ 
    \cite{hsieh2019convergence} & Strongly monotone, Lip. & N & N & Iterates & $\mathcal{O}(1/K)$ & N  & {1}\\
    \cite{yousefian2014optimal,yousefian17smoothing} & Monotone, non-Lip. & Y & Y & Gap fn. & $\mathcal{O}(1/\sqrt{K})$ & Y & {1} \\ 
    \cite{kannan19optimal} & Strongly pseudo/monotone+weak-sharp & Y & N & MSE & $\mathcal{O}(1/K)$ & Y & {2} \\
    \cite{wang2015incremental} & Strongly monotone, Lip., random proj. & N & N & Iterates & $\mathcal{O}(1/\sqrt{K})$ & Y & {1} \\ 
    \cite{iusem2017extragradient} & {Pseudo}monotone, Lip., var. reduction & N & N & Iterates & $\mathcal{O}(1/K)$ & Y  & {2}\\
     \cite{iusem2018incremental} & Monotone+weak-sharp, Lip., random proj. & N & Y & Dist. fn. & $\mathcal{O}(1/\sqrt{K})$ & Y & 2 \\
     \cite{iusem2018incremental} & Monotone, non-Lip., random proj. & Y & Y & Gap fn. & \tabincell{l}{$\mathcal{O}({K^\delta\ln K}/\sqrt{K})$\\ where $\delta > 0$}  & Y & 1 \\
    \hline
    \hline
    \textbf{v-SPRG} & Monotone+weak-sharp, Lip., var. reduction & Y & Y & Gap fn. & $\mathcal{O}(1/K)$ & Y & 1 \\ 
		& Noise: $\mathbb{E}[\|w(x)\|^2\mid x] \leq \nu_1^2 \|x\|^2 + \nu_2^2 $ a.s. & &&& & \\ \hline
    \textbf{v-SSE} & Monotone, Lip., var. reduction & Y & Y & Gap fn. & $\mathcal{O}(1/K)$ & Y & 1 \\ 
		& Noise: $\mathbb{E}[\|w(x)\|^2\mid x] \leq \nu_1^2 \|x\|^2 + \nu_2^2 $ a.s. & &&& &\\ \hline
    \textbf{r-SPRG} & Monotone+weak-sharp, Lip., random proj. & Y & Y & Gap fn. & $\mathcal{O}(1/\sqrt{K})$ & Y & 1 \\ 
		& Noise: $\mathbb{E}[\|w(x)\|^2\mid x] \leq \nu_1^2 \|x\|^2 + \nu_2^2 $ a.s. & &&& &\\ \hline
    \textbf{r-SSE} & Monotone+weak-sharp, Lip., random proj. & Y & Y & Gap fn. & $\mathcal{O}(1/\sqrt{K})$ & Y & 1 \\ 
		& Noise: $\mathbb{E}[\|w(x)\|^2\mid x] \leq \nu_1^2 \|x\|^2 + \nu_2^2 $ a.s. & &&& &\\ \hline
      \end{tabular}
\end{center}
\label{sa}
\vspace{-0.2in}
\end{table}

\subsection{Incorporating variance reduction and random projections.}  To
reduce the overall computational complexity, we define two variable sample-size
counterparts of ({\bf SPRG}) and ({\bf SEG}), where $N_k$ samples of the map are
utilized at iteration $k$ to approximate the expected map:  (i)
{\bf Variable sample-size (SPRG)}: 
\begin{align}
\begin{aligned}
x_{k+1}&\coloneqq\Pi_X\left(x_k-\gamma_k\tfrac{\sum_{j=1}^{N_k}F(2x_k-x_{k-1},\omega_{j,k})}{N_k}\right)
\end{aligned}, \tag{\bf v-SPRG}
\end{align}
and (ii) {\bf Variable sample-size (SSE)}.
\begin{align}
\begin{aligned}
\notag x_{k+\frac{1}{2}}&\coloneqq\Pi_{X}\left(x_k-\gamma_k\tfrac{\sum_{j=1}^{N_k}F(x_k,\omega_{j,k})}{N_k}\right), \\
\notag x_{k+1}&\coloneqq\Pi_{C_k}\left(x_k-\gamma_k\tfrac{\sum_{j=1}^{N_k}F(x_{k+\frac{1}{2}},\omega_{j,k+\frac{1}{2}})}{N_k}\right),
\end{aligned} \tag{\bf v-SSE}
\end{align}
where $C_k \triangleq \left\{y \in \mathbb{R}^n \mid \left(x_k-\gamma_k\tfrac{\sum_{j=1}^{N_k}F(x_k,\omega_{j,k})}{N_k}-x_{k+\frac{1}{2}}\right)^T(y-x_{k+\frac{1}{2}})\leq0\right\}$. \\

A difficulty arises when implementing such schemes on a complex set $X$
when $X$ is defined as the intersection of a large number of convex sets.
Inspired by~\cite{wang2015incremental}, we consider extending our work to
random projections when $X$ is defined as the intersection of a
finite number of sets: 
\setlength{\abovedisplayskip}{10pt}
\setlength{\belowdisplayskip}{10pt}
\begin{align*}
X=\bigcap_{i\in \mathcal{I}}X_i,
\end{align*}
where \us{$\mathcal{I}$ is a finite set} and $X_i \subseteq \mathbb{R}^n$
is closed and convex for all $i \in \mathcal{I}$. The key distinction is
that at each iteration, we project onto a random subset $X_{l_k}$ rather
than $X$, where $\{l_k\}$ is a sequence of random variables in the
appropriate steps of (SPRG) and (SSE). In prior work,
Nedi\'{c}~\cite{nedic2010random,nedic2011random} considered random
projection algorithms for convex optimization problems with similarly
defined sets and related schemes were subsequently considered for
nonsmooth convex regimes~
\cite{bertsekas2011incremental,wang2013incremental,wang2015random}.  Wang
and Bertsekas~\cite{wang2015incremental} applied this avenue to
strongly monotone stochastic variational inequality problems by extending
(SPG) to allow for projecting on a subset of constraints via random
projection technique while Iusem, Jofr\'{e}, and Thompson~\cite{iusem2018incremental} extended this framework by incorporating iterative regularization.  We consider analogous generalizations to {\bf (SPRG)} and {\bf
(SSE)}:\\ 

\noindent (i) {\bf Random projections SPRG schemes} {\bf (r-SPRG)}. 
\begin{align}
\notag x_{k+1}\coloneqq\Pi_{l_k}(x_k-\gamma_kF(2x_k-x_{k-1},\omega_k)),
	\tag{\bf r-SPRG}
\end{align}
{where $\Pi_{l_k}$ is defined as projection onto a random subset $X_{l_k}$ and}\\
(ii) {\bf Random projections   SSE schemes} {\bf (r-SSE)}.
\begin{align}
\notag x_{k+\frac{1}{2}}&\coloneqq\Pi_{l_k}(x_k-\gamma_kF(x_k,\omega_k)), \\
\notag x_{k+1}&\coloneqq\Pi_{C_k}(x_k-\gamma_kF(x_{k+\frac{1}{2}},\omega_{k+\frac{1}{2}})), \tag{\bf r-SSE}
\end{align}
where $C_k \triangleq \{y \in \mathbb{R}^n \mid (x_k-\gamma_kF(x_k,\omega_k)-x_{k+\frac{1}{2}})^T(y-x_{k+\frac{1}{2}})\leq0\}$. 

\avs{\subsection{Jutification and relation to other variance-reduced schemes} 

\noindent (i) {\em Terminology and applicability.} The term ``variance-reduced'' reflects the
usage of increasing accurate approximations of the expectation-valued map, as
opposed to noisy sampled variants that are used in single sample schemes. The
resulting schemes are often referred to as {\em mini-batch} SA schemes and
often achieve deterministic rates of convergence.  Schemes such as
SVRG~\cite{NIPS2013_ac1dd209} and SAGA~\cite{NIPS2014_ede7e2b6} also achieve
deterministic rates of convergence but are customized for finite sum problems
unlike mini-batch schemes that can process expectations over general
probability spaces.  Unlike in mini-batch schemes where increasing batch-sizes
are employed, in schemes such as SVRG, the entire set of samples is
periodically employed for computing a step. \\

\noindent (ii) {\em Weaker assumptions and stronger statements.} The proposed
variance-reduced framework has several crucial benefits that cannot be reaped
in the single-sample regime: (i) Under suitable assumptions, both ({\bf
v-SPRG}) and ({\bf v-SSE}) achieve optimal deterministic rates in terms of
major iterations (projection steps) while achieving near-optimal sample
complexity, i.e. $\mathcal{O}(1/\epsilon^{2+\delta})$. (ii) In addition, both
sets of schemes are equipped with a.s. convergence guarantees, statements which
are seldom obtained for single-sample extragradient schemes (to the best of our
knowledge).\\        

\noindent (iii) {\em Sampling requirements.} Naturally, variance-reduced
schemes can generally be employed only when sampling is relatively cheap compared to the main
computational step (such as computing a projection or a prox.) In terms of
overall sample-complexity, the proposed schemes are near optimal. As $k$ becomes large,
one might question how one might contend with $N_k$ tending to $+\infty$. This
issue does not arise since most schemes of this form are meant to provide
$\epsilon$-approximations. For instance, if $\epsilon = 1$e$-3$, then such a scheme
requires approximately $\mathcal{O}(1$e$3)$ steps. Since $N_k \approx \lceil
k^a \rceil$ and $a > 1$, we require approximately $(\mathcal{O}(1$e$3))^a$
samples. In a setting where multi-core architecture is ubiquitous, such
requirements are not terribly onerous particularly since computational costs
have been reduced from $\mathcal{O}(1$e$6)$ (single-sample) to
$\mathcal{O}(1$e$3)$. It is worth noting that competing schemes such as SVRG
would require taking the full batch-size intermittently and finite-sum problems
routingely have  $1$e$9$ or more samples.  }

\subsection{Contributions} We summarize the key aspects of our schemes in Tables \ref{cont} and elaborate on these next: 
\begin{table}[hpbt]
\scriptsize
\caption{(SRPG) and (SSE) schemes comparison,  $\mathbb{E}[\|w(x)\|^2\mid x] \leq \nu_1^2 \|x\|^2 + \nu_2^2 $}
\vspace{-0.2in}
\begin{center}
    \begin{tabular}[t]{ | c | C{3.4cm} | C{2.8cm} | C{3.4cm} | C{3.8cm} |}
    \hline
     & \multicolumn{2} {c} {Variance-reduced schemes} & \multicolumn{2} {|c|} {Random projection} \\ \cline{2-5}
     & Assump. & Result & Assump. & Result \\ \hline
    \multirow{2}{*}{{\bf (SPRG)}} & {\bf (v-SPRG)}: mono.+Lip., weak-sharpness & $\|x_k-x^*\| \xrightarrow[a.s.]{k \to \infty} 0$ &{\bf (r-SPRG)}:  mono.+Lip., weak-sharpness & $\|x_k-x^*\| \xrightarrow[a.s.]{k \to \infty} 0$  \\ \cline{2-5}
    & \tabincell{l}{{\bf (v-SPRG)}: mono.+Lip.\\ +compactness} & $\mathbb{E}[(G(\bar{x}_K))]\le\mathcal{O}\left(\tfrac{1}{K}\right)$ &\tabincell{l}{{\bf (r-SPRG)}:  mono.+Lip.\\ +compactness} & \tabincell{l}{$\mathbb{E}[(G(\vvs{\Pi_X(\bar{x}_K)}))]\le\mathcal{O}\left(\tfrac{1}{\sqrt{K}}\right)$ \\ $\mathbb{E}[\mbox{dist}(\bar{x}_{K},X)] \leq \mathcal{O}\left(\tfrac{1}{\sqrt{K}}\right)$}  \\ \hline
   \multirow{2}{*}{{\bf (SSE)}} & {\bf (v-SSE)}: mono.+Lip. & $\|x_k-x^*\| \xrightarrow[a.s.]{k \to \infty} 0$ &{\bf (r-SSE)}:  mono.+Lip., weak-sharpness & $\|x_k-x^*\| \xrightarrow[a.s.]{k \to \infty} 0$  \\ \cline{2-5}
    & \tabincell{l}{{\bf (v-SSE)}: mono.+Lip.\\ +compactness} & $\mathbb{E}[(G(\bar{x}_K))]\le\mathcal{O}\left(\tfrac{1}{K}\right)$ &\tabincell{l}{{\bf (r-SSE)}:  mono.+Lip.\\ +compactness} & \tabincell{l}{$\mathbb{E}[(G(\vvs{\Pi_X(\bar{x}_K)})]\le\mathcal{O}\left(\tfrac{1}{\sqrt{K}}\right)$ \\ $\mathbb{E}[\mbox{dist}(\bar{x}_{K},X)] \leq \mathcal{O}\left(\tfrac{1}{\sqrt{K}}\right)$}  \\ \hline
    \end{tabular}
\end{center}
\vspace{-0.2in}
\label{cont}
\end{table}
\\

\noindent (i) In Section 3, we prove that in settings where the maps are
monotone and Lipschitz continuous,  the iterates produced by both {variance
reduced variants of} {\bf (v-SPRG)} and {\bf v-(SSE)} converge almost surely
(a.s.) to a solution, where {\bf (v-SPRG)} requires an additional weak
sharpness requirement. However, without a weak-sharpness requirement, the
gap function for an averaged sequence for both schemes diminishes at the rate
of $\mathcal{O}(1/K)$. We emphasize that our findings for {\bf (v-SPRG)}
match {the best known} deterministic rate of convergence while we weaken the
assumption for the convergence of sequences generated by {\bf (v-SSE)} from
strong monotonicity to mere monotonicity. To the best of our knowledge, no
convergence rate for a deterministic version of ({\bf SSE}) is available in the
literature. \\ 

\noindent (ii) In Section 4, under a weak-sharpness requirement, the sequences produced by {random projection variants  {\bf (r-SPRG)} and {\bf (r-SSE)} are \vvs{shown to converge a.s. to the solution set of the original problem}. Additionally, without weak sharpness, the gap function of the projection of the  averaged sequence on $X$  as well as the infeasibility of the sequence with respect to the feasible set $X$ diminish at the rate of $\mathcal{O}(1/\sqrt{K})$.}\\ 

\noindent (iii) {In Section 5, preliminary numerics are observed support our expectations based on the theoretical findings.}

\section{Background and Assumptions}
We consider the schemes (SPRG) and (SSE) where $x_0 \in X$ is
a random initial point and $\{\gamma_k\}$ denotes the steplength
sequence.  We begin by imposing suitable Lipschitzian and monotonicity assumptions on the map $F$ which
will be valid through the remainder of this paper.
\begin{assumption}[{\bf Monotone and Lipschitz maps}] \label{lip-mon}
\em The mapping $F$ is $L$-Lipschitz continuous and monotone on $\mathbb{R}^n$, i.e. $\forall x, y \in \mathbb{R}^n$, 
$\|F(x)-F(y)\|  \le L\|x-y\|$
 and $(F(x)-F(y))^T(x-y)  \ge0$. \qed
\end{assumption}
{Since $F$ is a monotone map, VI$(X,F)$ may have multiple solutions. We assume that the set of solutions of VI$(X,F)$, {denoted by $X^*$}, is compact and nonempty.}  
\begin{assumption}[{\bf Compactness of $X^*$ and Boundedness of $F$}]
 \label{bd5} \em
{The set $X^*$ is compact and nonempty where $X^*$ denotes the set of solutions of  VI$(X,F)$, i.e.  $X^* \triangleq \{x^* \mid x^* \mbox{ solves } \mbox{VI}(X,F)\}$. There exists a constant $C>0$ such that $\|F(x^*)\|\le C$ for all $x^* \in X^*.$} \qed
\end{assumption}
{A sufficiency condition for the boundedness of $X^*$ (Assumption~\ref{bd5}) is a  suitable coercivity property of $F$ over the set $X$~\cite[Prop.~2.2.7]{facchinei2007finite}. This then allows for claiming the boundedness of $F$ over $X^*$.}
{For proving almost sure convergence of the iterates, we often impose} a
weak-sharpness requirement on VI$(X,F)$, which requires utilizing
the distance between a point $x$ and a set $X$, denoted by $\mbox{dist}(x,X)$
and defined as $\displaystyle \mbox{dist}(x,X) \triangleq \min_{y \in X}
\|x-y\|.$

\begin{assumption}[{\bf Weak sharpness}]~\label{weak-sharp}
\em The variational inequality problem {\em VI}$(X,F)$ satisfies the weak sharpness property implying
that there exists an $\alpha> 0$ such that for all $x \in X$, $(x-x^*)^TF(x^*) \geq \alpha \mbox{dist}\left(x,X^*\right).$ \qed
\end{assumption}
We assume the presence of a stochastic oracle that can provide a conditionally
unbiased estimator of $F(x)$, given by $F(x,\omega)$ such that
$\mathbb{E}[F(x,\omega) \vs{ \ \mid \ x}] = F(x)$. Define $w_k\triangleq
F(x_k,\omega_k)-F(x_k)$, $\bar{w}_k\triangleq
\tfrac{\sum_{j=1}^{N_k}F(x_k,\omega_{j,k})}{N_k}-F(x_k)$,
$w_{k+1/{2}}\triangleq
F(x_{k+{1}/{2}},\omega_{k+{1}/{2}})-F(x_{k+{1}/{2}})$ and
$\bar{w}_{k+{1}/{2}}\triangleq
\tfrac{\sum_{j=1}^{N_k}F(x_{k+{1}/{2}},\omega_{j,k})}{N_k}-F(x_{k+{1}/{2}})$, where $N_k$ denotes the batch-size of sampled maps $F(x,\omega_{j,k})$ at iteration $k$.
Furthermore, let $\mathcal{F}_k$ denote  the history up to iteration $k$, i.e.,
{\scriptsize
\begin{align*} 
\mathcal{F}_k \triangleq
\left\{x_0,\{F(x_0,\omega_{j,0})\}_{j=1}^{N_0},\{F(x_{1/2},\omega_{j,{1/2}})\}_{j=1}^{N_0},  \cdots  \{F(x_{k-1},\omega_{j,k-1})\}_{j=1}^{N_{k-1}}, \{F(x_{k-1/2},\omega_{j,k-1/2})\}_{j=1}^{N_{k-1}}\right\} \end{align*}}
and $\mathcal{F}_{k+\frac{1}{2}} \triangleq \mathcal{F}_k\cup\{F(x_k,\omega_{j,k})\}_{j=1}^{N_k}$.
In settings where the set $X$ may be unbounded, the assumption that the
conditional second moment $w_k$ is uniformly bounded a.s. is often a stringent
requirement. Instead, we impose a state-dependent assumption on $w_k$.

\begin{assumption}[{\bf State-dependent bound on noise}] \label{moment}
\em At iteration $k$, the following hold in an a.s. sense:
(i) The conditional means $\mathbb{E}[w_k\mid\mathcal{F}_k]$ and
$\mathbb{E}[w_{k+\frac{1}{2}}\mid\mathcal{F}_{k+\frac{1}{2}}]$ are zero
for all $k$ in an a.s. sense; (ii) The conditional second moments are bounded 
in an a.s. sense as follows. $\mathbb{E}[\|w_k\|^2\mid\mathcal{F}_k]\le\vvs{\nu_1^2\|x_k\|^2 + \nu_2^2}$ and
$\mathbb{E}[\|w_{k+\frac{1}{2}}\|^2\mid\mathcal{F}_{k+\frac{1}{2}}]\le\vvs{\nu_1^2\|x_{k+\frac{1}{2}}\|^2 + \nu_2^2}$
for all $k$ in an a.s. sense. \qed 
\end{assumption}
The following lemma is used in our analysis and may be found in~\cite{bauschke2011convex}.
\begin{lemma} \label{project} \em
Let $X$ be {a} nonempty closed convex set in $\mathbb{R}^n$. Then for all $y \in X$ and for any $x \in \Real^n$, we have that the following hold: 
(i) $(\Pi_X(x)-x)^T(y-\Pi_X(x))\ge0$; and (ii)  $\|\Pi_X(x)-y\|^2\leq\|x-y\|^2-\|x-\Pi_X(x)\|^2$. \qed
\end{lemma}
The {Robbins-Siegmund} (super-martingale convergence) lemma and its variant are also employed in our analysis (see~\cite{polyak1987introduction}).
\begin{lemma}\label{robbins}\em
Let $v_k$, $u_k$, $\delta_k$, $\psi_k$ be nonnegative random variables adapted to $\sigma$-algebra $\mathcal{F}_k$, and let the following relations hold almost surely.
\begin{align}
\notag\mathbb{E}[v_{k+1}\mid\mathcal{F}_k]\leq(1+u_k)v_k-\delta_k+\psi_k, \quad\forall k; \quad \sum_{k=0}^{\infty}u_k<\infty,\mbox{ and } \sum_{k=0}^{\infty}\psi_k<\infty.
\end{align}
Then a.s., we have that 
$\lim_{k\to\infty}v_k=v$ and $\sum_{k=0}^{\infty}\delta_k<\infty,$
where $v\ge0$ is a random variable. \qed
\end{lemma}
\begin{lemma}\label{robbins2}\em
Let $v_k$ be nonnegative random variables adapted to $\sigma$-algebra $\mathcal{F}_k$ where $\mathbb{E}[v_0] < \infty$. Suppose
\begin{align*}
\mathbb{E}[v_{k+1} \mid \Fscr_k] \leq (1-\alpha_k) v_k + \beta_k, \mbox{ a.s. } \mbox{ for all } k \geq 0.
\end{align*}
In addition, suppose $0 \leq \alpha_k \leq 1$ and $\beta_k \geq 0$ for $k \geq 0$, $\sum_{k=0}^{\infty} \alpha_k = \infty$, $\sum_{k=0}^{\infty} \beta_k < \infty$, and $\tfrac{\beta_k}{\alpha_k} \to 0$ as $k \to \infty$. Then $v_k \to 0$ as $k \to \infty$ in an a.s. sense. Furthermore, $\mathbb{E}[v_k] \to 0$ as $k \to \infty$. \qed 
\end{lemma}

We need the following Lemma to prove the a.s. convergence of (\textbf{v-SPRG}).
\begin{lemma} \label{mn}\em
 {Suppose the mapping $F$ is monotone on $\Real^n$ and the solution set of VI$(X,F)$ is given by $X^*$. Then for any $x^*, z^* \in X^*$, we have that 
\begin{align*}
	F(x^*)^T(z^* - x^*) = F(z^*)^T(x^*-z^*) = 0. 
\end{align*}}
\end{lemma}
\begin{proof}
Suppose {a} limit point of {a} subsequence of $\{x_k\}$ is {given by} $z^*\in
X^*$. By the definition of $X^*$, we have {for any $x^* \in X^*$ that}
\begin{align}
F(x^*)^T(z^*-x^*) &\ge 0, \label{eq_mon1} \\
F(z^*)^T(x^*-z^*) &\ge 0. \label{eq_mon2}
\end{align}
Combining these two inequalities, we obtain
\begin{align*}
(F(x^*)-F(z^*))^T(x^*-z^*) &\le 0.
\end{align*}
Since the mapping $F$ is monotone, we also have
\begin{align*}
(F(x^*)-F(z^*))^T(x^*-z^*) &\ge 0.
\end{align*}
{It follows that $F(z^*)^T(x^*-z^*)=F(x^*)^T(x^*-z^*)$, {which by invoking
\eqref{eq_mon2} implies that} $F(x^*)^T(x^*-z^*) \geq 0$. However, by recalling
\eqref{eq_mon1}, we have that $F(x^*)^T(z^*-x^*) = 0$. Consequently},
$F(z^*)^T(x^*-z^*)=F(x^*)^T(z^*-x^*)=0$. Thus, the conclusion follows.
\end{proof}

\section{Convergence analysis for (\textbf{v-SPRG}) and (v-SSE)}
{In this section, we analyze the convergence properties of {\bf (v-SPRG)} and {\bf (v-SSE)} in Sections~\ref{sec:3.1} and~\ref{sec:3.2}, respectively.}
\subsection{Stochastic Projected Reflected Gradient Schemes} \label{sec:3.1}
In this subsection, we prove the a.s. convergence of the iterates produced by (\textbf{v-SPRG}) when $F$ is a Lipschitz continuous and monotone map on $\mathbb{R}^n$ under a weak-sharpness requirement. We then relax the weak-sharpness assumption in deriving a rate statement in terms of the gap function for the averaged sequence.  We begin with a lemma that relates the error in consecutive iterates.
\begin{lemma} \label{egml1} \em
{Consider a sequence generated by {\bf (v-SPRG)}. Suppose Assumption
\ref{lip-mon} holds and 
$0<\gamma_k=\gamma\leq\tfrac{1}{8 \tilde L}$ for all $k$ where ${\tilde
L}^2 \triangleq (L^2 + \tfrac{10 \nu_1^2}{N_0}).$ Then for any $x_0 \in X$ and any $x^* \in X^*$,
the following holds for all $k \ge 0$.} 
\begin{align}
\notag&\|x_{k+1}-x^* \|^2+\tfrac{3}{4}\|x_{k+1}-y_k\|^2+2\gamma F(x^*)^T(x_k-x^*) \\
\notag&\le\|x_k-x^*
\|^2+\tfrac{3}{4}\|x_k-y_{k-1}\|^2+ 2\gamma F(x^*)^T(x_{k-1}-x^*) \\
\notag & +8\gamma^2\|w_k-w_{k-1}\|^2
-\left(1-16\gamma^2L^2\right)\|x_k-y_k\|^2-{2\gamma F(x^*)^T(x_k-x^*)} -2\gamma \vvs{\bar{w}_k}^T(y_k-x^*).
\end{align}
\end{lemma}

\begin{proof}
Define $y_k \triangleq 2x_k-x_{k-1}$ for all $k\ge1$ and $\bar{F}(y_k) \triangleq \frac{\sum_{j=1}^{N_k} F(y_k, \omega_{k,j})}{N_k}$. \ssc{We reuse the notation of $\bar{w}_k$ and define $\bar{w_k}=\bar{F}(y_k)-F(y_k)$ in this proof.} By Lemma \ref{project}(ii)
	and noting that $x_{k+1}=\Pi_X(x_k-\gamma_k\bar{F}(y_k))$ and
		$\bar{F}(y_k) = F(y_k) + \bar{w}_k$, 
	the following holds for $x_{k+1}$ and any solution $x^*$.
\begin{align}
\notag \|x_{k+1}-x^* \|^2 &\le \|x_k-\gamma_k\bar{F}(y_k)-x^*\|^2-\|x_k-\gamma_k\bar{F}(y_k)-x_{k+1}\|^2 \\ 
&= \| x_k-x^* \|^2-\| x_{k+1}-x_k \|^2-2\gamma_k(F(y_k)+\bar{w}_k)^T(x_{k+1}-x^*).\label{ineq1}
\end{align}
Since $F$ is monotone over $\Real^n$, by adding $2\gamma_k(F(y_k)-F(x^*))^T(y_k-x^*)$ to the {right hand side (rhs)} 
of \eqref{ineq1}, we obtain:
\begin{align}
\notag & \|x_{k+1}-x^* \|^2 \le \|x_k-x^*\|^2-\|x_{k+1}-x_k\|^2+2\gamma_k(F(y_k)-F(x^*))^T(y_k-x^*) \\ \notag&-2\gamma_k(F(y_k)+\bar{w}_k)^T(x_{k+1}-x^*) \\
\notag &= \|x_k-x^* \|^2-\| x_{k+1}-x_k \|^2+2\gamma_kF(y_k)^T(y_k-x_{k+1})+2\gamma_kF(y_k)^T(x_{k+1}-x^*) \\
\notag &-2\gamma_kF(x^*)^T(y_k-x^*)-2\gamma_kF(y_k)^T(x_{k+1}-x^*)+2\gamma_k\bar{w}_k^T(y_k-x_{k+1})-2\gamma_k\bar{w}_k^T(y_k-x^*) \\
\notag&=\|x_k-x^* \|^2-\| x_{k+1}-x_k \|^2+2\gamma_k(F(y_k)+\bar{w}_k)^T(y_k-x_{k+1})-2\gamma_k(F(x^*)+\bar{w}_k)^T(y_k-x^*) \\
\notag&=\|x_k-x^* \|^2-\| x_{k+1}-x_k \|^2+\underbrace{2\gamma_k(F(y_k)-F(y_{k-1}))^T(y_k-x_{k+1})}_{\vs{\tiny \mbox{Term 1}}} \\ 
&+\underbrace{2\gamma_k(F(y_{k-1})+\bar{w}_k)^T(y_k-x_{k+1})}_{\vs{\tiny \mbox{Term 2}}}-2\gamma_k(F(x^*)+\bar{w}_k)^T(y_k-x^*).\label{eq2}
\end{align}
Since $x_{k+1},x_{k-1}\in X$, by Lemma \ref{project}(i), we may conclude that
\begin{align}
\notag (x_k-x_{k-1}+\gamma_{k-1}(F(y_{k-1})+\bar{w}_{k-1}))^T(x_k-x_{k+1}) &\le 0 \mbox{ and }\\
\notag (x_k-x_{k-1}+\gamma_{k-1}(F(y_{k-1})+\bar{w}_{k-1}))^T(x_k-x_{k-1}) &\le 0.
\end{align}
Adding these two inequalities yields the following:
$$(x_k-x_{k-1}+\gamma_{k-1}(F(y_{k-1})+\bar{w}_{k-1}))^T(y_k-x_{k+1})\leq 0,$$
since $y_k = 2x_k - x_{k-1}$, leading to the following
inequality:
\begin{align}
\notag 2\gamma_{k-1}(F(y_{k-1})+\bar{w}_{k-1})^T(y_k-x_{k+1}) & \le2(x_k-x_{k-1})^T(x_{k+1}-y_k) \\
=2(y_k-x_k)^T(x_{k+1}-y_k) & =\|x_{k+1}-x_k \|^2-\| x_k-y_k \|^2-\| x_{k+1}-y_k \|^2, \label{eq3}
\end{align}
where the first equality follows from recalling that
$y_k=2x_k-x_{k-1}$.
Now, we may bound $2\gamma_k(F(y_{k-1})+\bar{w}_k)^T(y_k-x_{k+1})$ 
as follows:
\begin{align}
\notag& \vs{\mbox{Term 2}} = 2\gamma_k(F(y_{k-1})+\bar{w}_k)^T(y_k-x_{k+1})= 2\gamma_k(F(y_{k-1})+\bar{w}_k)^T(y_k-x_{k+1})\\\notag&-2\gamma_k(F(y_{k-1})+\bar{w}_{k-1})^T(y_k-x_{k+1})+2\gamma_k(F(y_{k-1})+\bar{w}_{k-1})^T(y_k-x_{k+1}) \\\notag
&=2\gamma_k(\bar{w}_k-\bar{w}_{k-1})^T(y_k-x_{k+1})+2\left(\tfrac{\gamma_k}{\gamma_{k-1}}\right)\gamma_{k-1}(F(y_{k-1})+\bar{w}_{k-1})^T(y_k-x_{k+1}) \\
\notag&\leq
8\gamma_k^2\|\bar{w}_k-\bar{w}_{k-1}\|^2+\tfrac{1}{8}\|x_{k+1}-y_k\|^2-\tfrac{\gamma_k}{\gamma_{k-1}}\|x_{k+1}-y_k\|^2+\tfrac{\gamma_k}{\gamma_{k-1}}\|x_{k+1}-x_k\|^2-\tfrac{\gamma_k}{\gamma_{k-1}}\|x_k-y_k\|^2\notag \\
&=8\gamma_k^2\|\bar{w}_k-\bar{w}_{k-1}\|^2+\left(\tfrac{1}{8}-\tfrac{\gamma_k}{\gamma_{k-1}}\right)\|x_{k+1}-y_k\|^2+\tfrac{\gamma_k}{\gamma_{k-1}}\|x_{k+1}-x_k\|^2-\tfrac{\gamma_k}{\gamma_{k-1}}\|x_k-y_k\|^2,\label{eq4}
\end{align}
where  
$2\gamma_k(w_k-w_{k-1})^T(y_k-x_{k+1})\le8\gamma_k^2\|w_k-w_{k-1}\|^2+\tfrac{1}{8}\|x_{k+1}-y_k\|^2$ and 
	inequality \eqref{eq3} allows for bounding $2\gamma_{k-1}(F(y_{k-1})+{\bar{w}_{k-1}})^T(y_k-x_{k+1}).$ Next we estimate
$(F(y_k)-F(y_{k-1})^T(y_k-x_{k+1})$. By  the Cauchy-Schwarz inequality
		and the Lipschitz continuity of the map (Ass.~\ref{lip-mon}),
		it follows that
\begin{align}
\notag& {\mbox{Term 1}} = 2\gamma_k(F(y_k)-F(y_{k-1}))^T(y_k-x_{k+1}) \le2\gamma_k\|F(y_k)-F(y_{k-1})\|\|y_k-x_{k+1}\|\\&\le2\gamma_kL\|y_k-y_{k-1}\|\|y_k-x_{k+1}\|\le8\gamma_k^2L^2\|y_k-y_{k-1}\|^2+\tfrac{1}{8}\|x_{k+1}-y_k\|^2
\label{eq5} \\ &\le16\gamma_k^2L^2\|x_k-y_{k-1}\|^2+16\gamma_k^2L^2\|x_k-y_k\|^2+\tfrac{1}{8}\|x_{k+1}-y_k\|^2,\label{eq6}
\end{align}
where \eqref{eq6} follows from $\|u+v\|^2\le2\|u\|^2+2\|v\|^2$.
Using \eqref{eq4} and \eqref{eq6}, we deduce from \eqref{eq2} that
\begin{align}
\notag \|x_{k+1}-x^*\|^2 
 & \le\| x_k-x^* \|^2-\left(1-\tfrac{\gamma_k}{\gamma_{k-1}}\right) \|x_{k+1}-x_k\|^2-\left(\tfrac{\gamma_k}{\gamma_{k-1}}-16\gamma_k^2L^2\right)\|x_k-y_k\|^2\\\notag& -\left(\tfrac{\gamma_k}{\gamma_{k-1}}-\tfrac{1}{4}\right)\|x_{k+1}-y_k\|^2+16\gamma_k^2L^2\|x_k-y_{k-1}\|^2+8\gamma_k^2\|\bar{w}_k-\bar{w}_{k-1}\|^2 \\
 &-2\gamma_k(F(x^*)+\bar{w}_k)^T(y_k-x^*).\label{eq7}
\end{align}
By assumption, {$0 \leq \tfrac{1}{8 \tilde L}$ for all $k$},
\begin{align}
\label{eq-bd}
16\gamma_k^2L^2\le {16 \gamma_k^2 \tilde {L}^2} \leq \tfrac{1}{4}\le\left(\tfrac{\gamma_{k-1}}{\gamma_{k-2}}-\tfrac{1}{4}\right).\end{align}
Consequently, from \eqref{eq7} and by invoking~\eqref{eq-bd}, we may conclude the following:
\begin{align}
\notag&\|x_{k+1}-x^* \|^2+\left(\tfrac{\gamma_k}{\gamma_{k-1}}-\tfrac{1}{4}\right)\|x_{k+1}-y_k\|^2\le\|x_k-x^*\|^2+\left(\tfrac{\gamma_{k-1}}{\gamma_{k-2}}-\tfrac{1}{4}\right)\|x_k-y_{k-1}\|^2 \\
&+8\gamma_k^2\|w_k-w_{k-1}\|^2-\left(\tfrac{\gamma_k}{\gamma_{k-1}}-16\gamma_k^2L^2\right)\|x_k-y_k\|^2-2\gamma_kF(x^*)^T(y_k-x^*)+\gamma_k\bar{w}_k^T(y_k-x^*).\label{eq9}
\end{align}
We may bound $2\gamma_kF(x^*)^T(y_k-x^*)$ as follows {when $\gamma_{k-1} \geq \gamma_k$}:
\begin{align} \notag
 & \quad -2\gamma_k F(x^*)^T(y_k - x^*)  = -2\gamma_k F(x^*)^T(x_k - x^*) - 2\gamma_k F(x^*)^T(x_k - x^*) + 2\gamma_k F(x^*)^T( x_{k-1}-x^*) \\
		& \leq -2\gamma_k F(x^*)^T(x_k - x^*) - 2\gamma_k F(x^*)^T(x_k - x^*) + 2\gamma_{k-1} F(x^*)^T( x_{k-1}-x^*).\label{bd-ws}
\end{align}
{By substituting $\gamma_k = \gamma$ for every $k$ in \eqref{bd-ws}, we obtain the required result.}
\begin{align}
\notag& \quad \|x_{k+1}-x^* \|^2+\tfrac{3}{4}\|x_{k+1}-y_k\|^2+2\gamma F(x^*)^T(x_k-x^*) \\
\notag&\le\|x_k-x^*
\|^2+\tfrac{3}{4}\|x_k-y_{k-1}\|^2+ 2\gamma F(x^*)^T(x_{k-1}-x^*) \\
& +8\gamma^2\|\bar{w}_k-\bar{w}_{k-1}\|^2
-\left(1-16\gamma^2L^2\right)\|x_k-y_k\|^2-{2\gamma F(x^*)^T(x_k-x^*)}  -2\gamma \bar{w}_k^T(y_k-x^*).\label{eq10}
\end{align}
\end{proof}

{By leveraging} this lemma, we prove a.s. convergence of {the sequence produced by} {\bf (v-SPRG)}.

\begin{theorem}[{\bf a.s. convergence of (v-SPRG)}] \em \label{egml}
{Consider a sequence generated by {\bf (v-SPRG)}. Let Assumptions
\ref{lip-mon} -- \ref{moment} hold. Suppose
$0<\gamma\leq\tfrac{1}{8 \tilde L}$ where ${\tilde
L}^2 \triangleq (L^2 + \tfrac{10 \nu_1^2}{N_0})$ and \vvs{$\{N_k\}$ is a
non-decreasing sequence} satisfying $\sum_{k=1}^\infty \tfrac{1}{N_k} < M$. Then for any $x_0 \in X$, the sequence generated by ({\bf v-SPRG}) converges to a point in $X^*$ in an a.s. sense. } 
\end{theorem}
 
\begin{proof}
Using \eqref{eq10}, taking expectations conditioned on $\mathcal{F}_k$, {and invoking Assumption~\ref{weak-sharp} and ~\ref{moment}, we obtain the following.} 
\begin{align}
\notag& \quad
\mathbb{E}\bigg[\|x_{k+1}-x^*\|^2+\tfrac{3}{4}\|x_{k+1}-y_k\|^2+2\gamma F(x^*)^T(x_k-x^*) 
\mid\mathcal{F}_k\bigg] \notag \\
\notag&\overset{\tiny (\mbox{Ass.}~\ref{weak-sharp})}{\le}\|x_k-x^* \|^2+\tfrac{3}{4}\|x_k-y_{k-1}\|^2+ 2\gamma F(x^*)^T(x_{k-1}-x^*) 
- 2\alpha \gamma \mbox{dist}\left(x_k,X^*\right) \notag \\
&\notag + 8\gamma^2\mathbb{E}[\|\bar{w}_k-\bar{w}_{k-1}\|^2\mid\mathcal{F}_k]-\left(1-16\gamma^2L^2\right)\|x_k-y_k\|^2  \\
&\overset{\tiny (\mbox{Ass.}~\ref{moment})}{\le}\|x_k-x^* \|^2+\tfrac{3}{4}\|x_k-y_{k-1}\|^2 \notag+ 2\gamma F(x^*)^T(x_{k-1}-x^*) - 2\alpha \gamma \mbox{dist}\left(x_k,X^*\right)  \\
\notag&  +\tfrac{16\gamma^2({\nu_1^2(\|y_k\|^2 + \|y_{k-1}\|^2)+ 2\nu_2^2)}}{N_k}-\left(1-16\gamma^2L^2\right)\|x_k-y_k\|^2\\  
\notag &\le\|x_k-x^* \|^2+\tfrac{3}{4}\|x_k-y_{k-1}\|^2 \notag+ 2\gamma F(x^*)^T(x_{k-1}-x^*) - 2\alpha \gamma \mbox{dist}\left(x_k,X^*\right)  \\
\notag &  +\tfrac{16\gamma^2\left({\nu_1^2(3\|y_k\|^2 + 2\|y_{k-1}-y_k\|^2)+ 2\nu_2^2}\right)}{N_k}-\left(1-16\gamma^2L^2\right)\|x_k-y_k\|^2\\  
\notag &\le\|x_k-x^* \|^2+\tfrac{3}{4}\|x_k-y_{k-1}\|^2 \notag+ 2\gamma F(x^*)^T(x_{k-1}-x^*) - 2\alpha \gamma \mbox{dist}\left(x_k,X^*\right)  \\
\notag &  +\tfrac{{16}\gamma^2\left({\nu_1^2(6\|y_k-x_k\|^2 + 6\|x_k\|^2 + 4\|y_{k-1}-x_k\|^2 + 4\|y_k-x_k\|^2)+ 2\nu_2^2}\right)}{N_k}-\left(1-16\gamma^2L^2\right)\|x_k-y_k\|^2\\  
\notag &\le\|x_k-x^* \|^2+\tfrac{3}{4}\|x_k-y_{k-1}\|^2 \notag+ 2\gamma F(x^*)^T(x_{k-1}-x^*) - 2\alpha \gamma \mbox{dist}\left(x_k,X^*\right)  \\
\notag &  +\tfrac{{16}\gamma^2\left({\nu_1^2(10\|y_k-x_k\|^2 + 12\|x_k-x^*\|^2 + 12\|x^*\|^2 + 4\|y_{k-1}-x_k\|^2 )+ 2\nu_2^2}\right)}{N_k}-\left(1-16\gamma^2L^2\right)\|x_k-y_k\|^2\\ \notag  
\notag &\le \left(1 + {\tfrac{192 \gamma^2 \nu_1^2}{N_k}}\right)\left(\|x_k-x^* \|^2+\tfrac{3}{4}\|x_k-y_{k-1}\|^2 + 2\gamma F(x^*)^T(x_{k-1}-x^*)\right) - 2\alpha \gamma \mbox{dist}\left(x_k,X^*\right)  \\
\notag &  +\tfrac{{{192}\gamma^2\|x^*\|^2+ {32}\gamma^2\nu_2^2}}{N_k}-\left(1-16\gamma^2L^2 - \tfrac{160\gamma^2 \nu_1^2}{N_k}\right)\|x_k-y_k\|^2\\  
& { \ = \ } \ \left(1 + {\tfrac{192 \gamma^2 \nu_1^2}{N_k}}\right)v_k - \delta_k+\psi_k, \label{eqn0}
\end{align}
where $v_k$, $\delta_k$,  and $\psi_k$ are random variables defined as 
\begin{align*}
 v_k & \triangleq \|x_k-x^* \|^2+\tfrac{3}{4}\|x_k-y_{k-1}\|^2+ 2\gamma F(x^*)^T(x_k-x^*),\\
\delta_k & \triangleq\left(1-16\gamma^2L^2- \tfrac{160\gamma^2 \nu_1^2}{N_k}\right)\|x_k-y_k\|^2+ 2\alpha
	\gamma \mbox{dist}\left(x_k,X^*\right),
\ \mathrm{ and }  \  \psi_k  \triangleq{\tfrac{{{192}\gamma^2\|x^*\|^2+ {32}\gamma^2\nu_2^2}}{N_k}}.
\end{align*}
{Since $x^* \in X^*$, $v_k \geq 0$ for every $k$ while $\psi_k \geq 0$ for every $k$ follows immediately}. In addition, by assumption, $\gamma \leq \tfrac{1}{8\tilde L}$ where   
{$$\left(16\gamma^2L^2 + \tfrac{160\gamma^2 \nu_1^2}{N_k}\right) \vvs{ \ = \ } \vvs{16\gamma^2} {\tilde L}^2\leq \tfrac{1}{4}, \mbox{ where } \tilde L^2 \triangleq \left(L^2 + \tfrac{10 \nu_1^2}{N_0}\right). $$ 
Consequently, by the choice of $\gamma$ and by noting that $\mbox{dist}(x_k,X^*) \geq 0$ for all $k$, it follows that $\psi_k \geq 0$ for all $k$. Furthermore,  $\sum_k \psi_k < \infty$ since $\sum_k
\frac{1}{N_k} < \infty$. } We may now invoke Lemma \ref{robbins} to claim
that $v_k \to \bar v \geq 0$ and $\sum_k \delta_k < \infty$ in an
	a.s. sense, implying the following holds a.s. {when $\gamma \leq \tfrac{1}{8\tilde L}$}:
\begin{align*} \infty & > \sum_k
\left(\left(1-16\gamma^2L^2- \tfrac{160\gamma^2 \nu_1^2}{N_k}\right)\|x_k-y_k\|^2 \right. \left.+2\alpha \gamma \mbox{dist}\left(x_k,X^*\right)\right) \\
	& \ge \sum_k \left(\left(1-\tfrac{1}{4}\right)\|x_k-y_k\|^2
+ 2\alpha \gamma \mbox{dist}\left(x_k,X^*\right)\right) 
 =  \sum_k \left(\tfrac{3}{4}\|x_k-y_k\|^2
+ 2\alpha \gamma \mbox{dist}\left(x_k,X^*\right)\right).
\end{align*}
Consequently, we have that 
$\infty > \sum_k \left(\tfrac{3}{4}\|x_k-y_k\|^2
+ 2\alpha \gamma \mbox{dist}\left(y_k,X^*\right)\right),$ implying that 
 $x_k-y_k \xrightarrow[a.s.]{k \to \infty}0$ and $\mbox{dist}(y_k,X^*)
\xrightarrow[a.s.]{k \to \infty} 0$. Since $\{x_k\}$ and $\{y_k\}$ have
the same set of limit points \avs{with probability one}, we may conclude that
$\mbox{dist}(x_k,X^*)\xrightarrow[a.s.]{k \to \infty} 0$ and $\{x_k\}$ is
bounded \avs{a.s.}. {It follows that \avs{with probability one}, $\{x_k\}$} has a convergent
subsequence;   we denote \avs{this subsequence by $\Kscr$ and its limit point} by \avs{$x_1^*(\omega)$. From hereafter, we suppress $\omega$ for ease of exposition}. {Since $x_k-y_k \xrightarrow[a.s.]{k \to \infty}0$
or $x_k-x_{k-1} \xrightarrow[a.s.]{k \to \infty}0$, we have
$x_k-y_{k-1} =(x_k-x_{k-1})-(x_{k-1}-x_{k-2}) \xrightarrow[a.s.]{k \to
\infty}0$. Since $\mbox{dist}(x_k,X^*) \xrightarrow[a.s.]{k \to \infty}0$, {every limit point of $\{x_k\}$ lies in $X^*$ in an a.s. sense and for any convergent subsequence $\Kscr$,} $F(x^*)^T(x_k-x^*)
\xrightarrow[a.s.]{k \in \Kscr, \ k \to \infty} F(x^*)^T(\avs{x_1}^*-x^*)
\ = \  0$ by Lemma~\ref{mn}, where $\avs{x_1}^* \in X^*$ \avs{in an a.s. sense}. Therefore, $\{\|x_k
    -y_{k-1}\|^2 + F(x^*)^T(x_k-x^*)\}$  \avs{converges in an a.s. sense}. Since $\{\|x_k-x^*\|^2 + \|x_k -y_{k-1}\|^2 +
    F(x^*)^T(x_k-x^*)\}$ \avs{converges a.s.} (Lemma~\ref{robbins}), \vvs{it follows that } $\{\|x_k-x^*\|^2\}$ \avs{converges a.s.} because $\{\|x_k-y_{k-1}\|^2+F(x^*)^T(x_k-x^*)\}$ \avs{converges a.s.}.
        Since $\{\|x_k-x^*\|^2\}$ \avs{converges a.s.} for any $x^* \in
    X^*$, it also \avs{converges a.s.} for $x^* = x_1^*$ and
    $\|x_k-x_1^*\| \xrightarrow[a.s.]{k \in \Kscr, \ k \to \infty} 0$; \avs{i.e. some subsequence of $\{\|x_k-x^*_1\|^2\}$ converges to zero}. \avs{Since $\{\|x_k-x_1^*\|\}$ is convergent a.s.},  \avs{we may conclude}
that the
entire sequence} $\{x_k\}$ \avs{converges to $x^*_1 \in X^*$}  in an a.s. sense.  
\end{proof}

\vvs{Next we derive rate statements for the averaged sequence in the merely monotone regimes without imposing a weak sharpness requirement. However, we do require a compactness requirement on $X$, a more common restriction when conducting rate analysis.} Unlike in stochastic convex
optimization where the function value represents a metric to ascertain
progress of the algorithm, a similar metric is not immediately available
for variational inequality problems. Instead, the progress of the scheme can
be ascertained by using the gap function, defined next (cf.~\cite{facchinei2007finite}).
\begin{definition}[{\bf Gap function}]
Given a nonempty closed set $X\subseteq \mathbb{R}^n$ and a mapping
$F:\mathbb{R}^n\to \mathbb{R}^n$, then the gap function at $x$ is
denoted by $G(x)$ and is defined as follows for any $x \in X$.
$$
G(x) \triangleq \sup_{y\in X}F(y)^T(x-y). 
$$
\end{definition}
The gap function is nonnegative for all $ x \in X$ and is zero if and only if
$x$ is a solution of \ic{VI}.  We establish the convergence rate for ({\bf v-SPRG})
by using the gap function. Importantly, we attain a rate of $\mathcal{O}(1/K)$
in terms of the expected gap and derive the oracle complexity.

\begin{theorem}\label{gapSPRG}
\em
\noindent Consider the ({\bf v-SPRG}) scheme and let $\{\bar{x}_K\}$ be defined as $\bar{x}_K=\sum_{k=0}^{K-1}x_{k}/K$, 
 where
$0<\gamma\leq\tfrac{1}{8\tilde L}$ where $\tilde L^2
\triangleq \left(L^2 + \tfrac{10 \nu_1^2}{N_0}\right)$ and {$\{N_k\}$ is a
non-decreasing sequence} satisfying $\sum_{k=1}^{\infty} {1\over N_k} < M$.   Let Assumptions~\ref{lip-mon}, \ref{bd5}, and \ref{moment} hold. {In addition, for any $u, v \in X$, suppose that there exists a $D_X > 0$ such that $\|u-v\|^2 \leq D_X^2$.}\\
(a) Then we have $\mathbb{E}[G(\bar{x}_K)] \leq \mathcal{O}\left(\frac{1}{K}\right)$ for any $K$. \\
(b) Suppose $N_k\triangleq\lfloor k^a\rfloor$ for $a>1$. Then the oracle complexity to ensure that $\mathbb{E}[G(\bar{x}_K)] \leq \epsilon$ satisfies 
$\sum_{k=1}^{K}N_k \leq \mathcal{O}\left(\frac{1}{\epsilon^\ssc{a+1}}\right)$.
\end{theorem}
\begin{proof}
(a) \vvs{From~\eqref{eq10}, we obtain} 
\begin{align}
2\gamma F(y)^T(x_k-y)
\notag&\le(\|x_k-y
\|^2+\tfrac{3}{4}\|x_k-y_{k-1}\|^2+ 2\gamma F(y)^T(x_{k-1}-y))\\
	& \notag - 
(\|x_{k+1}-y \|^2+\tfrac{3}{4}\|x_{k+1}-y_k\|^2+2\gamma F(x^*)^T(x_k-y)) 
 +8\gamma^2\|\bar{w}_k-\bar{w}_{k-1}\|^2 \\
& -\left(1-16\gamma^2L^2\right)\|x_k-y_k\|^2  -2\gamma \bar{w}_k^T(y_k-x_k) - 2\gamma_k \bar{w}_k^T(x_k-y).\label{eq10-r}
\end{align}
We now define an auxiliary sequence $\{u_k\}$ such that 
$$ u_{k+1}:= \Pi_X (u_k - \gamma \bar{w}_{k}), $$
where $u_0 \in X$. \ic{We have $\|u_{k+1}-y\|^2=\|\Pi_X (u_k - \gamma \bar{w}_{k})-y\|^2 \le \|u_k - \gamma \bar{w}_{k}-y\|^2=\|u_k-y\|^2-2\gamma\bar{w}_k(y-u_k)+\gamma^2\|\bar{w}_k\|^2$}. Then we may then express the last term on the right in \eqref{eq10-r} as follows.
\begin{align}\notag
 2\gamma \bar{w}_{k}^T(\vvs{y}-x_{k}) & = 2\gamma \bar{w}_{k}^T(\vvs{y}-u_k) + 2\gamma \bar{w}_{k}^T(u_k-x_{k}) \\
\label{eq-uk2}
	& \leq \|u_k-y\|^2 - \|u_{k+1}-y\|^2 + \gamma^2 \|\bar{w}_{k}\|^2 + 2\gamma \bar{w}_{k}^T(u_k-x_{k}). 
\end{align}
Summing over $k$ {and invoking \eqref{eq-uk2}}, we obtain the following bound:
{\begin{align*}
\notag 2\gamma\sum_{k=0}^{K-1} F(y)^T(x_{k}-y) &\le\|x_0-y\|^2+\tfrac{3}{4}\|x_{0}-y_{-1}\|^2+2\gamma F(y)^T(x_0-y)\\
	& + 8\gamma^2\sum_{k=0}^{K-1} \|\bar{w}_k-\bar{w}_{k-1}\|^2 
-2\gamma \sum_{k=0}^{K-1} \bar{w}_{k}^T(y_{k}-x_k) - 2\gamma \sum_{k=0}^{K-1}\bar{w}_k^T(x_k-y) \\
\implies \tfrac{2\gamma}{K} \sum_{k=0}^{K-1} F(y)^T(x_{k}-y)& \le\tfrac{1}{K}(\|x_0-y\|^2+\tfrac{3}{4}\|x_{0}-y_{-1}\|^2+2\gamma F(y)^T(x_0-y))\\
	& +\tfrac{8\gamma^2\sum_{k=0}^{K-1} \|\bar{w}_k-\bar{w}_{k-1}\|^2}{K} 
+\tfrac{\sum_{k=0}^{K-1} 2\gamma \bar{w}_{k}^T(\vvs{x_k}-y_{k})}{K} 
+\tfrac{\sum_{k=0}^{K-1} 2\gamma \bar{w}_{k}^T(\vvs{y}-x_{k})}{K} \\
\mbox{ or }  F(y)^T(\bar{x}_{K}-y)&\le\tfrac{1}{2\gamma K}(\|x_0-y\|^2+\tfrac{3}{4}\|x_{0}-y_{-1}\|^2+2\gamma F(y)^T(x_0-y)) \\
& +\tfrac{8\gamma^2\sum_{k=0}^{K-1} \|\bar{w}_k-\bar{w}_{k-1}\|^2}{2\gamma K} 
+ \tfrac{\sum_{k=0}^{K-1} 2\gamma\bar{w}_{k}^T(\vvs{y}-x_{k})}{2\gamma K} 
+ \tfrac{\sum_{k=0}^{K-1} 2\gamma\bar{w}_{k}^T(\vvs{x_k}-y_{k})}{2\gamma K} \\
&  \le\tfrac{1}{2\gamma K}(\|x_0-y\|^2+\tfrac{3}{4}\|x_{0}-y_{-1}\|^2+2\gamma F(y)^T(x_0-y))\\
	& +\tfrac{\gamma^2\sum_{k=0}^{K-1} (8\|\bar{w}_k-\bar{w}_{k-1}\|^2)}{2\gamma K}\\ 
& + \tfrac{\|u_0-y\|^2 + \sum_{k=0}^{K-1} (\gamma^2 \|\bar{w}_{k}\|^2+ 2\gamma \bar{w}_{k}^T(u_k-x_{k}))}{2\gamma K} 
+ \tfrac{\sum_{k=0}^{K-1} 2\gamma\bar{w}_{k}^T(\vvs{x_k}-y_{k})}{2\gamma K} \\
&  \le\tfrac{1}{2\gamma K}(\|x_0-y\|^2+\|u_0-y\|^2 + \tfrac{3}{4}\|x_{0}-y_{-1}\|^2+2\gamma F(y)^T(x_0-y))\\
	& +\tfrac{\gamma^2\sum_{k=0}^{K-1} (8\|\bar{w}_k-\bar{w}_{k-1}\|^2+ \|\bar{w}_k\|^2)}{2\gamma K}
 + \tfrac{ \sum_{k=0}^{K-1} 2\gamma \bar{w}_{k}^T(u_k-y_{k})}{2\gamma K} \\
 &  \le\tfrac{C_X^2}{2\gamma K}+\tfrac{\gamma^2\sum_{k=0}^{K-1} (8\|\bar{w}_k-\bar{w}_{k-1}\|^2+\|\bar{w}_{k}\|^2)}{2\gamma K}
 + \tfrac{\sum_{k=0}^{K-1} 2\gamma \bar{w}_{k}^T(u_k-y_{k})}{2\gamma K}. 
\end{align*}}
By taking supremum over $y \in X$, we obtain the following inequality:
{\begin{align*} G(\bar{x}_{K}) \triangleq 
\sup_{y \in X} F(y)^T(\bar{x}_{K}-y) & \le\tfrac{2C_X^2}{2\gamma K}+\tfrac{\gamma^2\sum_{k=0}^{K-1} (8\|\bar{w}_k-\bar{w}_{k-1}\|^2+\|\bar{w}_{k}\|^2)}{2\gamma K}
  + \tfrac{\sum_{k=0}^{K-1} 2\gamma \bar{w}_{k}^T(u_k-y_{k}))}{2\gamma K}, 
\end{align*}
where 
\begin{align*}  \|x_0-y\|^2+\ic{\|u_0-y\|^2}+\tfrac{3}{4}\|x_{0}-y_{-1}\|^2+2\gamma F(y)^T(x_0-y) & \leq \tfrac{\ic{11}}{4} D_X^2 + 8\gamma^2 \|F(y)-F(x^*)\|^2\\
	&  + 8\gamma^2 \|F(x^*)\|^2 + 2\|x_0 - y\|^2 \\
	& \leq \tfrac{\ic{19}}{4} D_X^2 + \underbrace{8\gamma^2 L^2 D_X^2}_{\mtiny {\le \vvs{8  \gamma^2 \tilde{L}^2 D_X^2 \leq} 1/8D_X^2}} D_X^2 + 8\gamma^2 C^2 \\
	& \leq \ic{5} D_X^2 + 8\gamma^2 C^2 \triangleq C_X^2.  
\end{align*}  
}
Taking expectations on both sides, leads to the following inequality. 
\begin{align} \notag
\mathbb{E}[G(\bar{x}_{K})]& \leq  {\tfrac{C_X^2}{2\gamma K}}+\tfrac{\gamma^2\sum_{k=0}^{K-1} 8\mathbb{E}[\|\bar{w}_k-\bar{w}_{k-1}\|^2]+\mathbb{E}[\|\bar{w}_{k}\|^2]}{2\gamma K} 
+ \tfrac{\sum_{k=0}^{K-1} 2\gamma\mathbb{E}[\bar{w}_{k}^T({u_k}-y_{k})]}{2\gamma K} \\
	\notag& \leq \tfrac{C_X^2+{\gamma^2}\sum_{k=0}^{K-1} \tfrac{{\nu_1^2(17\|x_k\|^2+16\|x_{k-1}\|^2)+33\nu_2^2}}{N_k}}{2\gamma K} 
	\leq \tfrac{{C_X^2}+{\gamma^2}\sum_{k=0}^{K-1} \tfrac{{\nu_1^2(66D_X^2+66\|x^*\|^2)+33\nu_2^2}}{N_k}}{2\gamma K} \\
	&\leq {\tfrac{{C_X^2}+{\gamma^2\vvs{M}}( {{\nu_1^2(66D_X^2+66\|x^*\|^2)+33\nu_2^2}})}{2\gamma K}} {=\tfrac{\widehat{C}}{K}},
\end{align}
{by {defining} $\widehat{C}\triangleq \left(C_X^2+\gamma^2 \vvs{M}({\nu_1^2(66D_X^2+66\|x^*\|^2)+33\nu_2^2})\right)/2\gamma$}. It follows that $\mathbb{E}[G(\bar{x}_{K})] \leq \mathcal{O}(1/K).$ \\
(b) {It follows from (a) that $K=\lfloor(\tfrac{\widehat{C}}{\epsilon})\rfloor$. We have
\begin{align*}
\sum_{k=0}^{K-1}N_k&\le\sum_{k=0}^{\lfloor(\widehat{C}/\epsilon)\rfloor-1} (k+1)^a = \sum_{t=1}^{\lfloor(\widehat{C}/\epsilon)\rfloor}t^a\le\int_{1}^{(\widehat{C}/\epsilon)+1}x^a dx\le\tfrac{((\widehat{C}/\epsilon)+1)^{a+1}}{a+1}\le\left(\tfrac{{\widetilde{C}}}{\epsilon^{a+1}}\right).
\end{align*}}
\end{proof}

\subsection{Stochastic Subgradient Extragradient Schemes} \label{sec:3.2}

We begin by proving the a.s. convergence of the iterates produced by  {\bf
(v-SSE)}. Unlike {\bf (v-SPRG)}, \vvs{to show a.s. convergence}, this scheme does not require an assumption of
weak sharpness but mere monotonicity suffices.

\begin{proposition}[{\bf a.s. convergence of (v-SSE)}]\label{egcgr}
\em {Consider a sequence generated by {\bf (v-SSE)}. Let Assumptions
\ref{lip-mon} and \ref{moment} hold. Suppose
\ssc{$0<\gamma_k=\gamma\leq\tfrac{1}{\sqrt{2}\tilde L}$ where $\tilde L^2
\triangleq \left(L^2 + \tfrac{4 \nu_1^2}{N_0}\right)$} and \vvs{$\{N_k\}$ is a
non-decreasing sequence} \vvs{satisfying} $\sum_{k=1}^\infty \tfrac{1}{N_k} < M$.  Then for any
$x_0 \in X$, the sequence generated by ({\bf v-SSE}) converges to a point in $X^*$ in an a.s. sense. } 
\end{proposition}

\begin{proof}
By Lemma \ref{project}(ii) we have for any $x^*$,
\begin{align}
\notag\|x_{k+1}-x^*\|^2&\le\|x_k-\gamma_k(F(x_{k+\frac{1}{2}})+\bar{w}_{k+\frac{1}{2}})-x^*\|^2-\|x_k-\gamma_k(F(x_{k+\frac{1}{2}})+\bar{w}_{k+\frac{1}{2}})-x_{k+1}\|^2\\
&=\|x_k-x^*\|^2-\|x_k-x_{k+1}\|^2+2\gamma_k(F(x_{k+\frac{1}{2}})+\bar{w}_{k+\frac{1}{2}})^T(x^*-x_{k+1}).\label{eq22}
\end{align}
It is clear that
\begin{align}
\notag F(x_{k+\frac{1}{2}})^T(x_{k+1}-x^*)&=F(x_{k+\frac{1}{2}})^T(x_{k+1}-x_{k+\frac{1}{2}}+x_{k+\frac{1}{2}}-x^*) \\
&= F(x_{k+\frac{1}{2}})^T(x_{k+1}-x_{k+\frac{1}{2}})+F(x_{k+\frac{1}{2}})^T(x_{k+\frac{1}{2}}-x^*). \label{eq23}
\end{align}
Substituting \eqref{eq23} in \eqref{eq22}, we obtain
\begin{align}
\notag\|x_{k+1}-x^*\|^2&\le\|x_k-x^*\|^2-\|x_k-x_{k+1}\|^2+2\gamma_kF(x_{k+\frac{1}{2}})^T(x_{k+\frac{1}{2}}-x_{k+1})\\\notag&-2\gamma_kF(x_{k+\frac{1}{2}})^T(x_{k+\frac{1}{2}}-x^*)+2\gamma_k\bar{w}_{k+\frac{1}{2}}^T(x^*-x_{k+1})\\
\notag&=\|x_k-x^*\|^2-\|x_k-x_{k+\frac{1}{2}}+x_{k+\frac{1}{2}}-x_{k+1}\|^2+2\gamma_kF(x_{k+\frac{1}{2}})^T(x_{k+\frac{1}{2}}-x_{k+1})\\
\notag&-2\gamma_kF(x_{k+\frac{1}{2}})^T(x_{k+\frac{1}{2}}-x^*)+2\gamma_k\bar{w}_{k+\frac{1}{2}}^T(x^*-x_{k+1})\\
\notag&=\|x_k-x^*\|^2-\|x_k-x_{k+\frac{1}{2}}\|^2-\|x_{k+\frac{1}{2}}-x_{k+1}\|^2-2(x_k-x_{k+\frac{1}{2}})^T(x_{k+\frac{1}{2}}-x_{k+1})\\
\notag&+2\gamma_kF(x_{k+\frac{1}{2}})^T(x_{k+\frac{1}{2}}-x_{k+1})-2\gamma_kF(x_{k+\frac{1}{2}})^T(x_{k+\frac{1}{2}}-x^*)+2\gamma_k\bar{w}_{k+\frac{1}{2}}^T(x^*-x_{k+1})\\
\notag&=\|x_k-x^*\|^2-\|x_k-x_{k+\frac{1}{2}}\|^2-\|x_{k+\frac{1}{2}}-x_{k+1}\|^2-2\gamma_kF(x_{k+\frac{1}{2}})^T(x_{k+\frac{1}{2}}-x^*)\\
&+2(x_{k+1}-x_{k+\frac{1}{2}})^T(x_k-\gamma_kF(x_{k+\frac{1}{2}})-x_{k+\frac{1}{2}})+2\gamma_k\bar{w}_{k+\frac{1}{2}}^T(x^*-x_{k+1}). \label{eq24}
\end{align}
By definition of $C_k$, we have
\begin{align}
(x_{k+1}-x_{k+\frac{1}{2}})^T(x_k-\gamma_k(F(x_k)+\bar{w}_k)-x_{k+\frac{1}{2}})\le0. \label{eq25}
\end{align}
Substituting \eqref{eq25} in \eqref{eq24}, we deduce that
\begin{align}
\notag&\|x_{k+1}-x^*\|^2\le\|x_k-x^*\|^2-\|x_k-x_{k+\frac{1}{2}}\|^2-\|x_{k+\frac{1}{2}}-x_{k+1}\|^2-2\gamma_kF(x_{k+\frac{1}{2}})^T(x_{k+\frac{1}{2}}-x^*) \\
\notag&+2\gamma_k(x_{k+1}-x_{k+\frac{1}{2}})^T(F(x_k)-F(x_{k+\frac{1}{2}}))+2\gamma_k \bar{w}_k^T(x_{k+1}-x_{k+\frac{1}{2}})+2\gamma_k\bar{w}_{k+\frac{1}{2}}^T(x^*-x_{k+1})\\
\notag&\le\|x_k-x^*\|^2-\|x_k-x_{k+\frac{1}{2}}\|^2-\|x_{k+\frac{1}{2}}-x_{k+1}\|^2-2\gamma_kF(x_{k+\frac{1}{2}})^T(x_{k+\frac{1}{2}}-x^*) \\
\notag&+2\gamma_k\|x_{k+1}-x_{k+\frac{1}{2}}\|\|F(x_k)-F(x_{k+\frac{1}{2}})\|+2\gamma_k(\bar{w}_k-\bar{w}_{k+\frac{1}{2}})^T(x_{k+1}-x_{k+\frac{1}{2}})+2\gamma_k\bar{w}_{k+\frac{1}{2}}^T(x^*-x_{k+\frac{1}{2}})\\
\notag&\le\|x_k-x^*\|^2-\|x_k-x_{k+\frac{1}{2}}\|^2-\|x_{k+\frac{1}{2}}-x_{k+1}\|^2+\tfrac{1}{2}\|x_{k+1}-x_{k+\frac{1}{2}}\|^2+2\gamma_k^2L^2\|x_k-x_{k+\frac{1}{2}}\|^2 \\
\notag&+2\gamma_k^2\|\bar{w}_k-\bar{w}_{k+\frac{1}{2}}\|^2+\tfrac{1}{2}\|x_{k+1}-x_{k+\frac{1}{2}}\|^2-2\gamma_kF(x_{k+\frac{1}{2}})^T(x_{k+\frac{1}{2}}-x^*)+2\gamma_k\bar{w}_{k+\frac{1}{2}}^T(x^*-x_{k+\frac{1}{2}})\\
\notag&=\|x_k-x^*\|^2-(1-2\gamma_k^2L^2)\|x_k-x_{k+\frac{1}{2}}\|^2+2\gamma_k^2\|\bar{w}_k-\bar{w}_{k+\frac{1}{2}}\|^2 \\
\notag&-2\gamma_kF(x_{k+\frac{1}{2}})^T(x_{k+\frac{1}{2}}-x^*)+2\gamma_k\bar{w}_{k+\frac{1}{2}}^T(x^*-x_{k+\frac{1}{2}}) \\
\notag&=\|x_k-x^*\|^2-(1-2\gamma^2L^2)\|x_k-x_{k+\frac{1}{2}}\|^2+2\gamma^2\|\bar{w}_k-\bar{w}_{k+\frac{1}{2}}\|^2 \\
&-2\gamma F(x_{k+\frac{1}{2}})^T(x_{k+\frac{1}{2}}-x^*)+2\gamma \bar{w}_{k+\frac{1}{2}}^T(x^*-x_{k+\frac{1}{2}}),
\label{eq251}
\end{align} 
by noticing that $\gamma_k=\gamma$.  Define $r_\gamma(x) \triangleq \|x-\Pi_X(x-\gamma F(x))\|$ as a residual function. We have
\begin{align}
\notag r^2_\gamma(x_k)&=\|x_k-\Pi_X(x_k-\gamma F(x_k))\|^2 \\
\notag&=\|x_k-x_{k+\frac{1}{2}}+\Pi_X(x_k-\gamma {F}(x_k)-\gamma \bar{w}_k)-\Pi_X(x_k-\gamma F(x_k))\| \\
\notag&\le2\|x_k-x_{k+\frac{1}{2}}\|^2+2\gamma^2\|\bar{w}_k\|^2.
\end{align}
It follows that
\begin{align}
-\tfrac{1}{2}\|x_k-x_{k+\frac{1}{2}}\|^2\le-\tfrac{1}{4}r^2_\gamma(x_k)+\tfrac{1}{2}\gamma^2\|\bar{w}_k\|^2.\label{eqn1}
\end{align}
Using \eqref{eqn1} in \eqref{eq251}, we obtain
\begin{align}
\notag&\|x_{k+1}-x^*\|^2\le\|x_k-x^*\|^2-\left(\tfrac{1}{2}-2\gamma^2L^2\right)\|x_k-x_{k+\frac{1}{2}}\|^2+2\gamma^2\|\bar{w}_k-\bar{w}_{k+\frac{1}{2}}\|^2 \\
\notag&-2\gamma F(x_{k+\frac{1}{2}})^T(x_{k+\frac{1}{2}}-x^*)+2\gamma \bar{w}_{k+\frac{1}{2}}^T(x^*-x_{k+\frac{1}{2}})-\tfrac{1}{2}\|x_k-x_{k+\frac{1}{2}}\|^2 \\
\notag&\le\|x_k-x^*\|^2-\left(\tfrac{1}{2}-2\gamma^2L^2\right)\|x_k-x_{k+\frac{1}{2}}\|^2+2\gamma^2\|\bar{w}_k-\bar{w}_{k+\tfrac{1}{2}}\|^2 \\
\notag&-2\gamma F(x_{k+\frac{1}{2}})^T(x_{k+\frac{1}{2}}-x^*)+2\gamma \bar{w}_{k+\frac{1}{2}}^T(x^*-x_{k+\frac{1}{2}})-\tfrac{1}{4}r^2_\gamma(x_k)+\tfrac{1}{2}\gamma^2\|\bar{w}_k\|^2 \\
\notag&\le\|x_k-x^*\|^2-\left(\tfrac{1}{2}-2\gamma^2L^2\right)\|x_k-x_{k+\frac{1}{2}}\|^2+\tfrac{9}{2}\gamma^2\|\bar{w}_k\|^2+4\gamma^2\|\bar{w}_{k+\frac{1}{2}}\|^2 \\
\notag&-2\gamma F(x_{k+\frac{1}{2}})^T(x_{k+\frac{1}{2}}-x^*)+2\gamma \bar{w}_{k+\frac{1}{2}}^T(x^*-x_{k+\frac{1}{2}})-\tfrac{1}{4}r^2_\gamma(x_k).
\end{align}
Taking expectations conditioned on $\mathcal{F}_{k}$, we 
obtain the following bound:
\begin{align}
\notag\mathbb{E}[\|x_{k+1}-x^*\|^2\mid\mathcal{F}_{k}] &\le\|x_k-x^*\|^2-(1-2\gamma^2L^2){\mathbb{E}[}\|x_k-x_{k+\frac{1}{2}}\|^2{\mid\mathcal{F}_{k}]}\\\notag&+\mathbb{E}[\mathbb{E}[4\gamma^2\|\bar{w}_{k+\frac{1}{2}}\|^2\mid\mathcal{F}_{k+\frac{1}{2}}]\mid\mathcal{F}_{k}]+\mathbb{E}\left[\tfrac{9}{2}\gamma^2\|\bar{w}_{k}\|^2\mid\mathcal{F}_{k}\right]\\\notag&-\mathbb{E}[\mathbb{E}[2\gamma \bar{w}_{k+\frac{1}{2}}^T(x_{k+\frac{1}{2}}-x^*)\mid\mathcal{F}_{k+\frac{1}{2}}]\mid\mathcal{F}_{k}]-\tfrac{1}{4}r^2_\gamma(x_k) \\
\notag & \leq \|x_k-x^*\|^2-(1-2\gamma^2L^2)\mathbb{E}[\|x_k-x_{k+\frac{1}{2}}\|^2{\mid\mathcal{F}_{k}]}\\\notag&+{\tfrac{4\gamma^2(\nu_1^2 {\mathbb{E}[}\|x_{k+\frac{1}{2}}\|^2{\mid\mathcal{F}_{k}]} + \nu^2_2)}{N_k}+\tfrac{\frac{9}{2}\gamma^2(\nu_1^2 \|x_k\|^2 + \nu^2_2)}{N_k}} 
-\tfrac{1}{4}r^2_\gamma(x_k) \\
\notag & \leq \|x_k-x^*\|^2-(1-2\gamma^2L^2){\mathbb{E}[}\|x_k-x_{k+\frac{1}{2}}\|^2{\mid\mathcal{F}_{k}]}\\\notag&+{\tfrac{4\gamma^2(2\nu_1^2 {\mathbb{E}[}\|x_k-x_{k+\frac{1}{2}}\|^2{\mid\mathcal{F}_{k}]} + 2\nu_1^2 \|x_k\|^2 + \nu^2_2)}{N_k}+\tfrac{\frac{9}{2}\gamma^2(\nu_1^2 \|x_k\|^2 + \nu^2_2)}{N_k}} 
-\tfrac{1}{4}r^2_\gamma(x_k) \\
\notag & \leq \|x_k-x^*\|^2-(1-2\gamma^2L^2- \tfrac{8\gamma^2 \nu_1^2}{N_k}){\mathbb{E}[}\|x_k-x_{k+\frac{1}{2}}\|^2{\mid\mathcal{F}_{k}]}\\&+{\tfrac{\tfrac{25}{2}\gamma^2(2\nu_1^2 \|x_k-x^*\|^2 + 2\nu_1^2 \|x^*\|^2)}{N_k}+{\tfrac{17\gamma^2\nu_2^2}{2N_k}}} 
-\tfrac{r^2_\gamma(x_k)}{4} \label{xk2} \\
\notag&\le \left(1+ \ssc{\tfrac{25\gamma^2 \nu_1^2}{N_k}}\right)\|x_k-x^*\|^2+\ssc{\tfrac{25\gamma^2\nu_1^2\|x^*\|^2}{N_k}+{\tfrac{17\gamma^2\nu_2^2}{2N_k}}} -\tfrac{r^2_\gamma(x_k)}{4}, 
\end{align}
{where the penultimate inequality follows from noting that $1-2\gamma^2L^2-
\tfrac{8\gamma^2 \nu_1^2}{N_k} \leq 1- 2\gamma^2(L^2+ \tfrac{4\gamma^2
\nu_1^2}{N_0})) \leq 1$, if $\gamma < \tfrac{1}{\sqrt 2 \tilde L}$ and $\tilde
L \triangleq L^2+ \tfrac{4\gamma^2 \nu_1^2}{N_0}.$}  We may now apply
Lemma~\ref{robbins} which allows us to claim that $\{\|x_k-x^*\|\}$ is
convergent \vvs{for any $x^* \in X^*$} and $\sum_{k}r_\gamma(x_k)^2<\infty$ in an a.s.  sense. Therefore,
in an a.s.
	sense, we have 
$$ \lim_{k \to \infty} r_\gamma(x_k)^2 = 0. $$
{{Since $\{\|x_k-x^*\|^2\}$ is a convergent sequence in an a.s. sense,
$\{x_k\}$ is bounded a.s. and has a convergent subsequence.} Consider any
convergent subsequence of $\{x_k\}$ {with index set} denoted by ${\cal K}$
\vvs{and suppose} its limit point is \vvs{denoted by} $\bar{x}(\omega)$. \avs{The dependence of $\bar{x}$ on $\omega$ is suppressed for ease of exposition.}  We have that
$\avs{0 = } \lim_{k \in {\cal K}} r_{\gamma}(x_k) = r_{\gamma}({\bar x})$ \avs{in an a.s. sense} since
$r_{\gamma}(\cdot)$ is a continuous map. It follows that $\bar x$ is a solution
to $\mbox{VI}(X,F)$ \avs{in an a.s. sense}. 
\vvs{Since $\{\|x_k-x^*\|^2\}$ is convergent \avs{a.s.} for any
    $x^* \in X^*$, it follows}  that $\{\|x_k-\avs{\bar{x}}\|^2\}$ is convergent  \avs{a.s.} \avs{since $\bar{x} \in X^*$ in an a.s. sense}. \avs{Since a subsequence of $\{\|x_k-\bar{x}\|^2\}$, denoted by $\Kscr$, converges to zero a.s., the entire sequence $\{\|x_k-\bar{x}\|^2\}$ converges to zero.} Therefore, the entire sequence $\{x_k\}$ converges \avs{a.s.} to a point in $X^*$.} \end{proof}

\vvs{We now proceed to derive a rate statement in terms of the gap function by imposing an extra compactness requirement.} 

\begin{proposition}\label{gapCGR}
\em
\noindent Consider the ({\bf v-SSE}) scheme and let $\{\bar{x}_K\}$ be defined as $\bar{x}_K=\tfrac{\sum_{k=0}^{K-1}x_{k+\frac{1}{2}}}{K}$, 
 where
{$0<\gamma\leq\tfrac{1}{\sqrt{2}\tilde L}$ where $\tilde L^2
\triangleq \left(L^2 + \tfrac{4 \nu_1^2}{N_0}\right)$} \vvs{and} {$\{N_k\}$ is a
non-decreasing sequence} \vvs{such that} $\sum_{k=1}^{\infty} {1\over N_k} < M$.
 Let Assumptions~\ref{lip-mon}, \ref{bd5} and \ref{moment}
hold. {In addition, suppose there exists a $D_X > 0$ such that $\|u-v\|^2 \leq D_X^2$ for any $u, v \in X$.} \\
(a) Then we have $\mathbb{E}[G(\bar{x}_K)] \leq \mathcal{O}\left(\frac{1}{K}\right)$ for any $K$. \\
(b) Suppose $N_k\triangleq\lfloor k^a\rfloor$, for $a>1$. Then the oracle complexity to ensure that $\mathbb{E}[G(\bar{x}_K)] \leq \epsilon$ satisfies 
$\sum_{k=1}^{K}N_k \leq \mathcal{O}\left(\frac{1}{\epsilon^{a+1}}\right)$.
\end{proposition}
\begin{proof}
(a) {From~\eqref{eq251}, we obtain} 
\begin{align}\notag
 \quad 2\gamma F(y)^T(x_{k+\frac{1}{2}}-y)&\le\|x_k-y\|^2-\|x_{k+1}-y\|^2-(1-2\gamma^2L^2)\|x_k-x_{k+\frac{1}{2}}\|^2+2\gamma^2\|\bar{w}_k-\bar{w}_{k+\frac{1}{2}}\|^2 \\
&+2\gamma \bar{w}_{k+\frac{1}{2}}^T({y}-x_{k+\frac{1}{2}}). \label{eq28a}
\end{align}
{We now define an auxiliary sequence $\{u_k\}$ such that 
$$ u_{k+1}:= \Pi_X (u_k - \gamma \bar{w}_{k+\frac{1}{2}}), $$
where $u_0 \in X$. With a similar analysis in \eqref{eq-uk2}, we may then express the last term on the right in \eqref{eq28a} as follows.
\begin{align}\notag
 2\gamma \bar{w}_{k+\frac{1}{2}}^T({y}-x_{k+\frac{1}{2}}) & = 2\gamma \bar{w}_{k+\frac{1}{2}}^T({y}-u_k) + 2\gamma \bar{w}_{k+\frac{1}{2}}^T(u_k-x_{k+\frac{1}{2}}) \\
\label{eq-uk}
	& \leq \|u_k-y\|^2 - \|u_{k+1}-y\|^2 + \gamma^2 \|\bar{w}_{k+\frac{1}{2}}\|^2 + 2\gamma \bar{w}_{k+\frac{1}{2}}^T(u_k-x_{k+\frac{1}{2}}). 
\end{align}
} 
Summing over $k$ {and invoking \eqref{eq-uk}}, we obtain the following bound:
{\begin{align*}
\notag \quad \sum_{k=0}^{K-1}2\gamma F(y)^T(x_{k+\frac{1}{2}}-y)& \le\|x_0-y\|^2+2\gamma^2\sum_{k=0}^{K-1} \|\bar{w}_k-\bar{w}_{k+\frac{1}{2}}\|^2 
+2\gamma \sum_{k=0}^{K-1} \bar{w}_{k+\frac{1}{2}}^T({y}-x_{k+\frac{1}{2}}) \\
\implies \tfrac{2\gamma}{K} \sum_{k=0}^{K-1} F(y)^T(x_{k+\frac{1}{2}}-y)& \le\tfrac{1}{K}\|x_0-y\|^2+\tfrac{2\gamma^2\sum_{k=0}^{K-1} \|\bar{w}_k-\bar{w}_{k+\frac{1}{2}}\|^2}{K} 
+\tfrac{\sum_{k=0}^{K-1} 2\gamma \bar{w}_{k+\frac{1}{2}}^T({y}-x_{k+\frac{1}{2}})}{K} \\
\mbox{ or }  F(y)^T(\bar{x}_{K}-y)&\le\tfrac{1}{2\gamma K}\|x_0-y\|^2+\tfrac{2\gamma^2\sum_{k=0}^{K-1} \|\bar{w}_k-\bar{w}_{k+\frac{1}{2}}\|^2}{2\gamma K} 
+ \tfrac{\sum_{k=0}^{K-1} 2\gamma\bar{w}_{k+\frac{1}{2}}^T({y}-x_{k+\frac{1}{2}})}{2\gamma K} \\
&  \le\tfrac{1}{2\gamma K}\|x_0-y\|^2+\tfrac{\gamma\sum_{k=0}^{K-1} 2\|\bar{w}_k-\bar{w}_{k+\frac{1}{2}}\|^2}{2K}\\ 
& + \tfrac{\|u_0-y\|^2 + \sum_{k=0}^{K-1} (\gamma^2 \|\bar{w}_{k+\frac{1}{2}}\|^2+ 2\gamma \bar{w}_{k+\frac{1}{2}}^T(u_k-x_{k+\frac{1}{2}}))}{2\gamma K} \\
 &  \le\tfrac{D_X^2}{\gamma K}+\tfrac{\gamma\sum_{k=0}^{K-1} (2\|\bar{w}_k-\bar{w}_{k+\frac{1}{2}}\|^2+\|\bar{w}_{k+\frac{1}{2}}\|^2)}{2K}
 + \tfrac{\sum_{k=0}^{K-1} \bar{w}_{k+\frac{1}{2}}^T(u_k-x_{k+\frac{1}{2}})}{K}. 
\end{align*}}
By taking supremum over $y \in X$, we obtain the following inequality:
{\begin{align*} G(\bar{x}_{K}) \triangleq 
\sup_{y \in X} F(y)^T(\bar{x}_{K}-y) & \le\tfrac{D_X^2}{\gamma K}+\tfrac{\gamma\sum_{k=0}^{K-1} (2\|\bar{w}_k-\bar{w}_{k+\frac{1}{2}}\|^2+\|\bar{w}_{k+\frac{1}{2}}\|^2)}{2K}
  + \tfrac{\sum_{k=0}^{K-1}  \bar{w}_{k+\frac{1}{2}}^T(u_k-x_{k+\frac{1}{2}})}{K}. 
\end{align*}}
Taking expectations on both sides, leads to the following inequality. 
\begin{align} \notag
\mathbb{E}[G(\bar{x}_{K})]& \leq  {\tfrac{D_X^2}{\gamma K}}+\tfrac{\gamma\sum_{k=0}^{K-1} 2\mathbb{E}[\|\bar{w}_k-\bar{w}_{k+\frac{1}{2}}\|^2]+2\mathbb{E}[\|\bar{w}_{k+\frac{1}{2}}\|^2]}{2K} 
+ \tfrac{\sum_{k=0}^{K-1} \mathbb{E}[\bar{w}_{k+\frac{1}{2}}^T({u_k}-x_{k+\frac{1}{2}})]}{K} \\
	\notag& \leq \tfrac{{2D_X^2}+{\gamma^2}\sum_{k=0}^{K-1} \tfrac{{\nu_1^2(4\|x_k\|^2+6\|x_{k+\frac{1}{2}}\|^2)+10\nu_2^2}}{N_k}}{2\gamma K} 
	\leq \tfrac{{2D_X^2}+{\gamma^2}\sum_{k=0}^{K-1} \tfrac{{\nu_1^2(20D_X^2+20\|x^*\|^2)+10\nu_2^2}}{N_k}}{2\gamma K}\\
	& \leq \tfrac{{2D_X^2}+{\gamma^2 \vvs{M}} ({{\nu_1^2(20D_X^2+20\|x^*\|^2)+10\nu_2^2}})}{2\gamma K}
{=\tfrac{\widehat{C}}{K}},
\end{align}
{by \vvs{defining} $\widehat{C} \triangleq (2D_X^2+\gamma^2\vvs{M} (\nu_1^2(20D_X^2+20\|x^*\|^2)+10\nu_2^2)/2\gamma$}. It follows that $\mathbb{E}[G(\bar{x}_{K})] \leq \mathcal{O}(1/K).$ \\
(b) We \vvs{may prove this result in a fashion similar to that used in} Proposition~\ref{gapSPRG}\vvs{(b)}.
\end{proof}

\vvs{\begin{remark} \em Several aspects of the prior results require further emphasis.

\noindent (a)   We observe that the rate guarantees of $\mathcal{O}(1/K)$ in
terms of the gap function for variance-reduced schemes for monotone stochastic
variational inequality problems matches those obtained by Iusem et
al.~\cite{iusem2017extragradient} but with a lower per-iteration complexity.
This rate was also achieved by more recently by Jalilzadeh and
Shanbhag~\cite{jalilzadeh19proximal}. Notably, the latter scheme leverages an
inexact proximal framework. In fact, proximal-point techniques have been useful in conducting a unified analysis for both gradient and extragradient schemes for deterministic strongly convex-concave saddle point problems as seen by the recent work by Mokhtari et al.~\cite{mokhtari19unified} that proves a linear rate of convergence in terms of solution error.\\

\noindent (b) While much of the techniques in the literature rely on uniform
bounds on the conditional second moments, we allow for state-dependence akin to
that adopted in~\cite{iusem2017extragradient} and do not rely on compactness
for proving a.s. convergence statements. Naturally, rate statements do impose a
compactness assumption. In addition, weak sharpness is only required for
proving a.s. convergence for ({\bf v-SPRG}) but does not find application
elsewhere. \\

\noindent (c) To the best of our knowledge, this is the first available rate for projected reflected gradient methods in stochastic and merely monotone regimes. The prior rate of linear convergence was provided in deterministic and strongly monotone settings~\cite{malitsky2015projected}. We also remain unaware of rate statements for SSE schemes and this appears to be the first rate statements for such schemes.

\end{remark}}

\section{Incorporating Random Projections in  ({\bf SPRG}) and ({\bf SSE})}
\vs{In this section, we assume that even a single projection onto the feasible
set $X$ is challenging. We assume that $X$ is given by an intersection of a
collection of closed and convex sets $\{X_i\}_{i \in I}$ where $I$ is a finite
set and consider a variants of ({\bf SPRG}) and ({\bf SSE}) where the projection onto $X$ is replaced by a projection onto a randomly selected set $X_i$.} In Section~\ref{sec:4.1}, we review our main assumptions and any
supporting results and proceed to derive asymptotic and rate guarantees in Sections~\ref{sec:4.2} and ~\ref{sec:4.3} for the random projection variants of ({\bf SPRG}) and ({\bf SSE}), respectively. 
\subsection{Assumptions and Supporting Results}\label{sec:4.1}

To establish the convergence, we need the following additional assumptions on the set $X=\bigcap_{i\in I}X_i$ and random projection $\Pi_{l_k}$.
The following assumption is known as linear regularity and is discussed in \cite{wang2015incremental}. It indicates that this condition is a mild restriction in practice.
\begin{assumption}\label{ita}
There exists a positive scalar $\eta$ such that for any $x\in \mathbb{R}^n$
$$\|x-\Pi_X(x)\|^2 \le \eta\max_{i\in I}\|x-\Pi_{X_i}(x)\|^2,$$
where $I$ is a finite set of indices, $I=\{1,\dots,m\}$.
\end{assumption}
The following assumption requires that each constraint is sampled with at least some probability and the random samples are nearly independent, which refers to \cite{wang2015incremental}.
\begin{assumption}\label{rl}
The random variables $l_k, k=0, 1, \dots,$ are such that
$\inf_{k\ge0}P(l_k=X_i\mid\mathcal{F}_k)\ge\frac{\rho_i}{m}$ with probability $1$ for $i=1, \dots,m,$ where $\rho_i\in(0,1]$ is a scalar for $i = 1, \hdots, m$.
\end{assumption}
The following lemma is essential to our proofs and it leverages basic properties of projection.
\begin{lemma}\label{bd0} \em
Let $X$ be a closed convex subset of $\mathbb{R}^n$. We have
$$\|y-\Pi_X(y)\|^2\le2\|x-\Pi_X(x)\|^2+8\|x-y\|^2, \quad\forall x,y\in\mathbb{R}^n.$$
\end{lemma}

\begin{proof}
Since $y-\Pi_X(y)=(x-\Pi_X(x))-(x-y)+(\Pi_X(x)-\Pi_X(y))$, we have
\begin{align*}
\|y-\Pi_X(y)\| &\le \|x-\Pi_X(x)\|+\|x-y\|+\|\Pi_X(x)-\Pi_X(y)\| \le \|x-\Pi_X(x)\|+2\|x-y\|.
\end{align*}
Thus,
\begin{align*}
\|y-\Pi_X(y)\|^2 &\le 2\|x-\Pi_X(x)\|^2+8\|x-y\|^2,
\end{align*}
where the last inequality leverages $\|a+b\|^2\le2\|a\|^2+2\|b\|^2$. 
\end{proof}
The following lemma provides an inequality useful in deriving lower bound \vvs{for} $\|x_{k+1}-x^*\|^2$.
\begin{lemma}\label{bd} \em
\vvs{Suppose Assumptions \ref{lip-mon} -- \ref{weak-sharp} hold.} Then, we have
$$F(x)^T(x-x^*)\ge \alpha \mbox{dist}\left(\Pi_X(x),X^*\right)-C\mbox{dist}(x,X), \quad \forall x\in\mathbb{R}^n. $$
\end{lemma}

\begin{proof}
We have
\begin{align}
F(x)^T(x-x^*)=(F(x)-F(x^*))^T(x-x^*)+F(x^*)^T(\Pi_X(x)-x^*)+F(x^*)^T(x-\Pi_X(x)). \label{bd01}
\end{align}
From the monotonicity assumption on $F$, we have
\begin{align}
(F(x)-F(x^*))^T(x-x^*) \ge 0. \label{bd02}
\end{align}
Since $x^*$ is a solution, it follows that from the weak sharpness property,
\begin{align}
F(x^*)^T(\Pi_X(x)-x^*) \ge \alpha\mbox{dist}\left(\Pi_X(x),X^*\right). \label{bd03}
\end{align}
Finally, \vvs{by recalling that}  $F(x^*)^T(\Pi_X(x)-x) \le \|F(x^*)\|\|x-\Pi_X(x)\|$ and $\|F(x^*)\| \le C$ (by Assumption \ref{bd5}), \vvs{we have that}
\begin{align}
F(x^*)^T(x-\Pi_X(x)) \ge -\|F(x^*)\|\|x-\Pi_X(x)\| \ge -C\mbox{dist}(x,X). \label{bd04}
\end{align}
By substituting \eqref{bd02} -- \eqref{bd04} in \eqref{bd01}, the result follows.
\end{proof}
Next, we provide a simple bound on $F(x)$. 
\begin{lemma}\label{lu} \em
Suppose Assumptions \ref{lip-mon} -- \ref{bd5} hold. Then for any $x \in \Real^n$,
$\|F(x)\|^2 \le 2L^2\|x-x^*\|^2+2C^2.$
\end{lemma}

\vvs{Finally, we derive a lower bound on $\mathbb{E}[\|x_k - \Pi_{l_k}(x_k)\|^2 \mid \Fscr_k]$.}

\begin{lemma}\label{bd20} \em
Suppose Assumptions~\ref{ita} and \ref{rl} hold. Then for any $l_k \in I$ and any $x \in \Real^n$,
$$\mathbb{E}[\|x-\Pi_{l_k}(x)\|^2\mid\mathcal{F}_k] \ge \tfrac{\rho}{m\eta}\mbox{dist}^2(x,X),\quad  k\ge 0,$$
with probability 1, where $\rho\triangleq\min_{i\in I}\{\rho_i\}$.
\end{lemma}

\begin{proof}
Following from Assumption \ref{rl}, we have
\begin{align*}
\mathbb{E}[\|x-\Pi_{l_k}(x)\|^2\mid\mathcal{F}_k]& =\sum_{i=1}^mP(l_k=i\mid\mathcal{F}_k)\|x-\Pi_i(x)\|^2 \ge \tfrac{\rho}{m}\|x-\Pi_j(x)\|^2, \quad \forall j=1,\dots,m \\
\implies \mathbb{E}[\|x-\Pi_{l_k}(x)\|^2\mid\mathcal{F}_k] &  \ge \tfrac{\rho}{m}\max_{j}\|x-\Pi_j(x)\|^2 \overset{\tiny (\mbox{Ass.}~\ref{ita})}{\ge} \tfrac{\rho}{m\eta}\mbox{dist}^2(x,X).
\end{align*}
\end{proof}

\subsection{{\bf SPRG} with random projections}\label{sec:4.2}
We begin with an a.s. convergence claim for ({\bf r-SPRG}). 

\begin{theorem}\label{rpml}
\em
Let Assumptions \ssc{\ref{lip-mon} -- \ref{rl}} hold. {Suppose the steplength sequence $\{\gamma_k\}$ satisfies $\sum_{k=0}^{\infty} \gamma_k = \infty$ and $\sum_{k=0}^{\infty} \gamma^2_k < \infty$.} Then any sequence generated by {\bf (r-SPRG)}, where the projections are randomly generated, converges to a solution $x^*\in X^*$ in an a.s. sense.
\end{theorem}

\begin{proof}
Define $y_k=2x_k-x_{k-1}$ for all $k\ge1$ {and $w_k=F(y_k,\omega_k)-F(y_k)$}.
By Lemma \ref{project}(ii) and by noting that
$x_{k+1}=\Pi_{\ssc{l_k}}(x_k-\gamma_kF(2x_k-x_{k-1}))$ and
$F(\ssc{y_k},\omega_k) = F(\ssc{y_k}) + w_k$, we have the following inequality:
\begin{align}
\notag \|x_{k+1}-x^* \|^2 &\le \|x_k-\gamma_kF(y_k,\omega_k)-x^*\|^2-\|x_k-\gamma_kF(y_k,\omega_k)-x_{k+1}\|^2 \\ 
\notag &= \|x_k-x^* \|^2-\|x_{k+1}-x_k \|^2-2\gamma_k(F(y_k)+w_k)^T(x_{k+1}-x^*) \\
&= \|x_k-x^* \|^2-\|x_{k+1}-x_k \|^2-2\gamma_kF(y_k)^T(x_{k+1}-x^*)-2\gamma_kw_k^T(x_{k+1}-x^*). \label{ml-eq1}
\end{align} 
{Recall that $\|y_k-x_{k+1}\|^2$ can be expressed as follows.}
{\begin{align*}
\|y_k-x_{k+1}\|^2  
& = \|(2x_k-x_{k-1})-x_{k+1}\|^2  \\
& = \|(x_k - (x_{k-1}-x_k)-x_{k+1})\|^2 \\
& = 2(x_k-x_{k-1})^T(x_k-x_{k+1})+\|x_{k-1}-x_k\|^2 +\|x_k- x_{k+1}\|^2 \\
& = 2(x_k-x_{k-1})^T(2x_k-2x_{k+1})+\|x_{k-1}-x_{k+1}\|^2 \\
& = 2(x_k-x_{k-1})^T(x_k+x_{k-1}-2x_{k+1})+2\|x_k-x_{k-1}\|^2+ \|x_{k-1}-x_{k+1}\|^2 \\ 
& = 2\|x_k-x_{k+1}\|^2-2\|x_{k-1}-x_{k+1}\|^2+2\|x_k-x_{k-1}\|^2+ \|x_{k-1}-x_{k+1}\|^2 \\
& = 2\|x_k-x_{k+1}\|^2-\|x_{k-1}-x_{k+1}\|^2+2\|x_k-x_{k-1}\|^2. 
\end{align*}}
Consequently, we  have that
\begin{align}
\tfrac{1}{4}\|x_k-x_{k+1}\|^2=\tfrac{1}{8}\|y_k-x_{k+1}\|^2+\tfrac{1}{8}\|x_{k-1}-x_{k+1}\|^2-\tfrac{1}{4}\|x_k-x_{k-1}\|^2. \label{ml-eq2}
\end{align}
Using \eqref{ml-eq2} in \eqref{ml-eq1}, we obtain
\begin{align}
\notag &\|x_{k+1}-x^* \|^2 \le \|x_k-x^* \|^2-\tfrac{3}{4}\|x_{k+1}-x_k \|^2-\tfrac{1}{8}\|y_k-x_{k+1}\|^2-\tfrac{1}{8}\|x_{k-1}-x_{k+1}\|^2 \\
\notag&+\tfrac{1}{4}\|x_k-x_{k-1}\|^2-2\gamma_kF(y_k)^T(x_{k+1}-x^*)-2\gamma_kw_k^T(x_{k+1}-x^*) \\
\notag&=\|x_k-x^* \|^2-\tfrac{3}{4}\|x_{k+1}-x_k \|^2-\tfrac{1}{8}\|y_k-x_{k+1}\|^2-\tfrac{1}{8}\|x_{k-1}-x_{k+1}\|^2 \\
&+\tfrac{1}{4}\|x_k-x_{k-1}\|^2-2\gamma_kF(y_k)^T(y_k-x^*)-2\gamma_kF(y_k)^T(x_{k+1}-y_k)-2\gamma_kw_k^T(x_{k+1}-x^*)\label{ml-eq2i} \\
\notag&\le\|x_k-x^* \|^2-\tfrac{3}{4}\|x_{k+1}-x_k \|^2-\tfrac{1}{8}\|y_k-x_{k+1}\|^2-\tfrac{1}{8}\|x_{k-1}-x_{k+1}\|^2+\tfrac{1}{4}\|x_k-x_{k-1}\|^2 \\
&-2\gamma_k\alpha\mbox{dist}\left(\Pi_X(y_k),X^*\right)+2\gamma_kC\mbox{dist}(y_k,X)-2\gamma_kF(y_k)^T(x_{k+1}-y_k)-2\gamma_kw_k^T(x_{k+1}-x^*) \label{ml-eq3},
\end{align}
where the last inequality follows from Lemma \ref{bd}. Since 
\begin{align*}
-2\gamma_kF(y_k)^T(x_{k+1}-y_k)\le16\gamma_k^2\|F(y_k)\|^2+\tfrac{1}{16}\|x_{k+1}-y_k\|^2,
\end{align*}
{ inequality \eqref{ml-eq3} can be rewritten as follows:}
\begin{align}
\notag &\|x_{k+1}-x^* \|^2\le\|x_k-x^* \|^2-\tfrac{3}{4}\|x_{k+1}-x_k \|^2-\tfrac{1}{16}\|y_k-x_{k+1}\|^2-\tfrac{1}{8}\|x_{k-1}-x_{k+1}\|^2+\tfrac{1}{4}\|x_k-x_{k-1}\|^2 \\
\notag&-2\gamma_k\alpha\mbox{dist}\left(\Pi_X(y_k),X^*\right)+2\gamma_kC\mbox{dist}(y_k,X)+16\gamma_k^2\|F(y_k)\|^2+16\gamma_k^2\|w_k\|^2-2\gamma_kw_k^T(y_k-x^*) \\
\notag&\le\|x_k-x^* \|^2-\tfrac{3}{4}\|x_{k+1}-x_k \|^2-\tfrac{1}{16}\|y_k-x_{k+1}\|^2-\tfrac{1}{8}\|x_{k-1}-x_{k+1}\|^2+\tfrac{1}{4}\|x_k-x_{k-1}\|^2 \\
\notag&-2\gamma_k\alpha\mbox{dist}\left(\Pi_X(y_k),X^*\right)+2\gamma_kC\mbox{dist}(y_k,X)+32\gamma_k^2L^2\|y_k-x^*\|^2 \\
&+32\gamma_k^2C^2+16\gamma_k^2\|w_k\|^2-2\gamma_kw_k^T(y_k-x^*). \label{ml-eq3i}
\end{align}
Since 
\begin{align*}
-2\gamma_k\alpha\mbox{dist}\left(\Pi_X(y_k),X^*\right) &\le -2\gamma_k\alpha\mbox{dist}\left(x_k,X^*\right)+2\gamma_k\alpha\|x_k-\Pi_X(y_k)\| \\
&\le -2\gamma_k\alpha\mbox{dist}\left(x_k,X^*\right)+2\gamma_k\alpha\|x_k-y_k\|+2\gamma_k\alpha\|y_k-\Pi_X(y_k)\| \\
&= -2\gamma_k\alpha\mbox{dist}\left(x_k,X^*\right)+2\gamma_k\alpha\|x_k-y_k\|+2\gamma_k\alpha \mbox{dist}(y_k,X),
\end{align*}
we have
\begin{align}
& \notag \|x_{k+1}-x^* \|^2\le\|x_k-x^* \|^2-2\gamma_k\alpha\mbox{dist}\left(x_k,X^*\right)-\tfrac{3}{4}\|x_{k+1}-x_k \|^2-\tfrac{1}{16}\|y_k-x_{k+1}\|^2 \\
\notag&-\tfrac{1}{8}\|x_{k-1}-x_{k+1}\|^2+\tfrac{1}{4}\|x_k-x_{k-1}\|^2+2\gamma_k\alpha\|x_k-y_k\|+2\gamma_k(C+\alpha)\mbox{dist}(y_k,X) \\
&+64\gamma_k^2L^2\|x_k-x^*\|^2+64\gamma_k^2L^2\|x_k-x_{k-1}\|^2+32\gamma_k^2C^2+16\gamma_k^2\|w_k\|^2-2\gamma_kw_k^T(y_k-x^*). \label{ml-eq5}
\end{align}
By Lemma \ref{bd20},
\begin{align}
\mathbb{E}[\|y_k-x_{k+1}\|^2\mid\mathcal{F}_{k}] &\ge \mathbb{E}[\|y_k-\Pi_{l_k}y_k\|^2\mid\mathcal{F}_{k}] \ge \tfrac{\rho}{m\eta}\mbox{dist}^2(y_k,X). \label{ml-eq6}
\end{align}
Taking expectations conditioned on $\mathcal{F}_{k}$ and using \eqref{ml-eq6} in \eqref{ml-eq5}, we have
\begin{align}
\notag&\mathbb{E}[\|x_{k+1}-x^*\|^2+\tfrac{3}{4}\|x_{k+1}-x_k\|^2\mid\mathcal{F}_{k}] \le \|x_k-x^* \|^2-2\gamma_k\alpha\mbox{dist}\left(x_k,X^*\right)-\tfrac{1}{16}\mathbb{E}[\|y_k-x_{k+1}\|^2\mid\mathcal{F}_{k}] \\
\notag&-\tfrac{1}{8}\mathbb{E}[\|x_{k-1}-x_{k+1}\|^2\mid\mathcal{F}_{k}]+\tfrac{1}{4}\|x_k-x_{k-1}\|^2+2\gamma_k\alpha\|x_k-y_k\|+2\gamma_k(C+\alpha)\mbox{dist}(y_k,X) \\
\notag&+64\gamma_k^2L^2\|x_k-x^*\|^2+64\gamma_k^2L^2\|x_k-x_{k-1}\|^2+32\gamma_k^2C^2+16\gamma_k^2\mathbb{E}[\|w_k\|^2\mid\mathcal{F}_{k}] \\
\notag&\le \|x_k-x^* \|^2-2\gamma_k\alpha\mbox{dist}\left(x_k,X^*\right)-\tfrac{1}{16}\tfrac{\rho}{m\eta}\mbox{dist}^2(y_k,X)+\tfrac{1}{4}\|x_k-x_{k-1}\|^2+2\gamma_k\alpha\|x_k-y_k\| \\
\notag&+2\gamma_k(C+\alpha)\mbox{dist}(y_k,X)+64\gamma_k^2L^2\|x_k-x^*\|^2+64\gamma_k^2L^2\|x_k-x_{k-1}\|^2+32\gamma_k^2C^2\\\notag&+{16\gamma_k^2(\nu_1^2\|y_k\|^2+\nu_2^2)} \\
\notag&\le \|x_k-x^* \|^2-2\gamma_k\alpha\mbox{dist}\left(x_k,X^*\right)-\tfrac{1}{16}\tfrac{\rho}{m\eta}\mbox{dist}^2(y_k,X)+\tfrac{1}{4}\|x_k-x_{k-1}\|^2+2\gamma_k\alpha\|x_k-y_k\| \\
\notag&+2\gamma_k(C+\alpha)\mbox{dist}(y_k,X)+64\gamma_k^2L^2\|x_k-x^*\|^2+64\gamma_k^2L^2\|x_k-x_{k-1}\|^2+32\gamma_k^2C^2\\\notag&+{16\gamma_k^2(2\nu_1^2\|y_k-x_k\|^2+2\nu_1^2\|x_k\|^2+\nu_2^2)} \\
\notag&\le \|x_k-x^* \|^2-2\gamma_k\alpha\mbox{dist}\left(x_k,X^*\right)-\tfrac{1}{16}\tfrac{\rho}{m\eta}\mbox{dist}^2(y_k,X)+\tfrac{1}{4}\|x_k-x_{k-1}\|^2+2\gamma_k\alpha\|x_k-y_k\| \\
\notag&+2\gamma_k(C+\alpha)\mbox{dist}(y_k,X)+64\gamma_k^2L^2\|x_k-x^*\|^2+64\gamma_k^2L^2\|x_k-x_{k-1}\|^2+32\gamma_k^2C^2\\\notag&+{16\gamma_k^2(2\nu_1^2\|y_k-x_k\|^2+4\nu_1^2\|x_k-x^*\|^2+4\nu_1^2\|x^*\|^2+\nu_2^2)} \\
\notag&= \|x_k-x^* \|^2+\tfrac{3}{4}\|x_k-x_{k-1}\|^2-2\gamma_k\alpha\mbox{dist}\left(x_k,X^*\right)-\tfrac{1}{2}\|x_k-x_{k-1}\|^2+2\gamma_k\alpha\|x_k-x_{k-1}\| \\
\notag&+2\gamma_k(C+\alpha)\mbox{dist}(y_k,X)-\tfrac{1}{16}\tfrac{\rho}{m\eta}\mbox{dist}^2(y_k,X)+64\gamma_k^2(L^2+{\nu_1^2})\|x_k-x^*\|^2\\&+{32\gamma_k^2(2L^2+\nu_1^2)}\|x_k-x_{k-1}\|^2+32\gamma_k^2C^2+16\gamma_k^2{(4\nu_1^2\|x^*\|^2+\nu_2^2)}. \label{ml-eq61}
\end{align}
{Noting that $-\frac{1}{2}\|x_k-x_{k-1}\|^2+2\gamma_k\alpha\|x_k-x_{k-1}\|=-\frac{1}{2} \vs{(\|x_k-x_{k-1}\|}-2\gamma_k\alpha)^2+2\gamma_k^2\alpha^2$ and inserting it in \eqref{ml-eq61}}, we have
\begin{align}
\notag&\mathbb{E}[\|x_{k+1}-x^*\|^2+\tfrac{3}{4}\|x_{k+1}-x_k\|^2\mid\mathcal{F}_{k}]\le\|x_k-x^* \|^2+\tfrac{3}{4}\|x_k-x_{k-1}\|^2-2\gamma_k\alpha\mbox{dist}\left(x_k,X^*\right)\\\notag&-\tfrac{1}{2}{(\|x_k-x_{k-1}\|-2\gamma_k\alpha)^2}+2\gamma_k^2\alpha^2-\tfrac{\rho}{16m\eta}\left(\mbox{dist}(y_k,X)-\tfrac{16m\eta\gamma_k(C+\alpha)}{\rho}\right)^2+\tfrac{16m\eta (C+\alpha)^2}{\rho}\gamma_k^2 \\
\notag&+64\gamma_k^2(L^2+{\nu_1^2})\|x_k-x^*\|^2+{32\gamma_k^2(2L^2+\nu_1^2)}\|x_k-x_{k-1}\|^2+32\gamma_k^2C^2+16\gamma_k^2{(4\nu_1^2\|x^*\|^2+\nu_2^2)} \\
\notag&\le(1+86\gamma_k^2L^2{+64\gamma_k^2\nu_1^2})\left(\|x_k-x^* \|^2+\tfrac{3}{4}\|x_k-x_{k-1}\|^2\right)\\
  \notag& -\underbrace{\left(\tfrac{1}{2}{(\|x_k-x_{k-1}\|-2\gamma_k\alpha)^2}+2\gamma_k\alpha\mbox{dist}\left(x_k,X^*\right)+\tfrac{\rho}{16m\eta}\left(\mbox{dist}(y_k,X)-\tfrac{16m\eta\gamma_k(C+\alpha)}{\rho}\right)^2\right)}_{\beta_k} \\
&+\underbrace{\left(2\gamma_k^2\alpha^2+\tfrac{16m\eta (C+\alpha)^2}{\rho}\gamma_k^2+32\gamma_k^2C^2+16\gamma_k^2{(4\nu_1^2\|x^*\|^2+\nu_2^2)}\right)}_{\eta_k}. \label{ml-eq7}
\end{align}
In effect, we obtain the following recursion:
\begin{align*}
\mathbb{E}[v_{k+1} \mid \mathcal{F}_k] \leq (1+u_k) v_k - \beta_k+ \eta_k, \quad \avs{\mbox{a.s.}} 
\end{align*}
where $v_k \triangleq \left(\|x_k-x^* \|^2+\tfrac{3}{4}\|x_k-x_{k-1}\|^2\right)$
and $u_k = 86 \gamma_k^2L^2{+64\gamma_k^2\nu_1^2}$. Since $\sum{\gamma_k^2}<\infty$, it follows that
$u_k$ and $\eta_k$ are summable. We may then invoke Lemma \ref{robbins} and
it follows that \vs{with probability \avs{one}, the random sequence} $\{\|x_k-x^*
\|^2+\tfrac{3}{4}\|x_k-x_{k-1}\|^2\}$ is convergent \avs{a.s.} and
$\sum\{\frac{1}{2}\|x_k-x_{k-1}-2\gamma_k\alpha\|^2+2\gamma_k\alpha\mbox{dist}\left(x_k,X^*\right)\}<\infty$
\vs{with probability \avs{one}}. We have that $\sum_{k}
\frac{1}{2}\|x_k-x_{k-1}-2\gamma_k\alpha\|^2 < \infty$ \avs{a.s.} implying that
$\|x_k-x_{k-1}-2\gamma_k\alpha\| \to 0$ in a.s. sense. It follows that $\|y_k -
x_k - 2\gamma_k \alpha\| \to 0$ almost surely. Since $\gamma_k \to 0$, it follows that
{$x_k - x_{k-1} \to 0$ in an
a.s. sense}. Thus $\{\|x_k-x^*\|\}$ is convergent \us{in an a.s. sense.}
\vvs{Consequently, $\{\|x_k-x^*\|\}$ is bounded \avs{a.s.} and has a convergent
    subsequence  \avs{a.s.}. For any convergent subsequence denoted by $\Kscr$, we have that
$\{x_k\}_{k \in \Kscr} \to \hat{x}(\omega)$. We proceed by contradiction and
assume that  $\hat{x}(\omega) \notin X^*$ with finite probability}. Therefore,
$\mbox{dist}(x_k,X^*)\to h(\omega)>0$ where $\omega \in V$ and $P(V)>0$
$\implies \sum_{\avs{k}}\gamma_k\mbox{dist}(x_k,X^*) \avs{ \ = \ } \infty$ \avs{with finite probability. But this implies that} $\sum_{\omega\in
V}\gamma_k\mbox{dist}(x_k,X^*)\nless\infty$ a.s. $\Longrightarrow$
contradiction. Thus every limit point of $\{x_k\}$ lies in $X^*$ a.s. . Consider
any such limit point $x_1^*\in X^*$. \avs{Then} we have that $\{\|x_k-x_1^*\|\}$ is
convergent to \vvs{zero} \avs{in an a.s. sense} Since a subsequence  \avs{of} $\{\|x_k-x_1^*\|\}$ converges to \vvs{zero a.s.} and the entire sequence is
convergent, the entire sequence converges to $x_1^*$ a.s. . 

\end{proof}

\us{Unlike in ({\bf v-SPRG}), feasibility of the iterates $\{x_k\}$ cannot be maintained in ({\bf r-SPRG}).  This} feasibility error arises because the random projection algorithms cannot
guarantee feasibility of $\{x_k\}$. First we conduct almost-sure convergence
analysis on the metric $\{\mbox{dist}(x_k,X)\}$ for both randomly generated
algorithms and then derive the rate of convergence. To establish the rate of
convergence, we need the following lemma \vvs{from} \cite{wang2015incremental}.
\begin{lemma}\label{wang} \em
Suppose $\beta\in(0,1)$ and $R \ge 0$.
Let $\{\delta_k\}$ and $\{\alpha_k\}$ be nonnegative sequences such that
\begin{align*}
\delta_{k+1} \le (1-\beta)\delta_k+R\alpha^2_k, \quad\forall k \ge 0.
\end{align*}
 If there exists $\bar{k}\ge0$ such that $\alpha^2_{k+1}\ge(1-\frac{\beta}{2})\alpha^2_k$ for all $k\ge\bar{k}$, we have
\begin{align*}
\delta_k \le \tfrac{2R}{\beta}\alpha^2_k+\delta_0(1-\beta)^k+\left(R\sum_{t=0}^{\bar{k}}\alpha^2_t\right)(1-\beta)^{k-\bar{k}}.
\end{align*}
\end{lemma}
We also prove the following result that relates geometric rates to polynomial rates.
\begin{lemma}\label{lem:geom_poly} \em
Consider $\beta \in (0,1)$, $d > 0$, $t \in \mathbb{Z}_+$, and $t \geq 1$. Then there exists a scalar $\bar{a} > 0$ such that
$$ d\beta^k \leq \frac{\bar a}{k^t}, \qquad \forall k \geq 0 $$
where $\bar a$ is defined as
\begin{align*}
	\bar a \triangleq 
 \left(\tfrac{te}{\ln(1/\beta)}\right)^t.
\end{align*} 
\end{lemma}
\begin{proof}
By definition of $\bar{a}$, we have that for all $k \geq 0$, 
\begin{align} d \beta^k \leq \tfrac{\bar{a}}{k^t} \mbox{ or } dk^t \beta^k \leq \bar{a}. \label{star} \end{align}
But \eqref{star} holds if 
$$ \bar{a} = \max_{z \geq 0} dz^t \beta^z. $$
If $h(z) \triangleq z^t \beta^z$, then any interior maximizer of $\max_{z \geq 0} \, h(z)$ satisfies 
$$ h'(z)= tz^{t-1} \beta^z + z^t \beta^z \ln(\beta) = 0 \implies t + z \ln(\beta) = 0 \mbox{ or } z^* = -\tfrac{t}{\ln(\beta)} =\tfrac{t}{\ln(1/\beta)}. $$ It is relatively simple to show that $h''(z^*) < 0$.
Consequently, $$\max_{z \geq 0} z^t \beta^z = h(z^*) =\left(\tfrac{t}{\ln(1/\beta)}\right)^t \beta^{-\tfrac{t}{\ln(\beta)}} = \left(\tfrac{t}{\ln(1/\beta)}\right)^t \left(\left(\tfrac{1}{\beta}\right)^{\tfrac{1}{\ln(\beta)}}\right)^t = \left(\tfrac{te}{\ln(1/\beta)}\right)^t. $$
Note that we utilize the relation that $(1/x)^{1/\ln(x)} = e$ in the last equality.  
As a result, we have 
$ d k^t \beta^k \leq d \left(\tfrac{te}{\ln(1/\beta)}\right)^t = d\bar{a}^t.$
\end{proof}
\begin{theorem}\label{rpml-fe} \em
Let Assumptions \ssc{\ref{lip-mon} -- \ref{bd5}, \ref{moment} -- \ref{rl}} hold. 
Suppose $\{x_k\}$ is generated by ({\bf r-SPRG}), where the projections are randomly generated. \ssc{In addition, suppose there exists a $D_X > 0$ such that $\|u-v\|^2 \leq D_X^2$ for any $u, v \in X$.} Then the following hold. \\
(a) {If $\sum_{k} \gamma_k = \infty$ and $\sum_k \gamma_k^2< \infty$}, then $\mbox{dist}(x_k,X) \xrightarrow[a.s.]{k \to \infty}0$. \\
(b) {Suppose $\gamma_k = 1/k^{t/2}$ where $t \geq 1$.} Then $\mathbb{E}[\mbox{dist}(x_k,X)] \leq \mathcal{O}\left(\frac{1}{k^{t/2}}\right)$ for any $k \geq \bar k$, where
\vvs{$$ \bar k \triangleq 
\ceil[\bigg]{\left(\tfrac{1}{1- \left(1-\beta/2\right)^{1/t}} -1\right)}.$$}
(c)  {Suppose $\gamma_k = 1/k^{1/2}$.} Suppose $\bar x_{K,\bar k} \triangleq \tfrac{\sum_{k=\bar k+\lfloor K/2\rfloor }^{\bar{k}+K} \gamma_k x_k}{
\sum_{k=\bar k+\lfloor K/2\rfloor}^{\bar k+K} \gamma_k}$. Then 
 $\mathbb{E}[\mbox{dist}(\bar{x}_{K, \bar k},X)] \leq \mathcal{O}\left(\tfrac{1}{\sqrt{K}}\right)$.

\end{theorem}

\begin{proof}
(a) Let $z_k=x_k-\gamma_kF(2x_k-x_{k-1},\omega_k)$.
From Lemma~\ref{project}, we have
\begin{align}
\notag\mbox{dist}^2(x_{k+1},X) &\le \|x_{k+1}-\Pi_X(z_k)\|^2 = \|\Pi_{l_k}(z_k)-\Pi_X(z_k)\|^2 \\ &\le \|z_k-\Pi_X(z_k)\|^2-\|\Pi_{l_k}(z_k)-z_k\|^2. \label{fe1}
\end{align}
Choose $\theta \ge \max\left\{1,\tfrac{3\rho}{64m\eta}\right\}$. By $\|a+b\|^2 \le \left(1+\tfrac{4\theta m\eta}{\rho}\right)\|a\|^2+\left(1+\tfrac{\rho}{4\theta m\eta}\right)\|b\|^2$, we obtain
\begin{align}
\notag \|z_k-\Pi_X(z_k)\|^2 &\le \|z_k-\Pi_X(x_k)\|^2 = \|z_k-x_k+x_k-\Pi_X(x_k)\|^2 \\
&\le \left(1+\tfrac{4\theta m\eta}{\rho}\right)\|z_k-x_k\|^2+\left(1+\tfrac{\rho}{4\theta m\eta}\right)\|x_k-\Pi_X(x_k)\|^2. \label{fe2}
\end{align}
Combining \eqref{fe1} and \eqref{fe2}, we obtain that
\begin{align}
\mbox{dist}^2(x_{k+1},X) \le \left(1+\tfrac{4\theta m\eta}{\rho}\right)\|z_k-x_k\|^2+\left(1+\tfrac{\rho}{4\theta m\eta}\right)\mbox{dist}^2(x_k,X)-\|\Pi_{l_k}(z_k)-z_k\|^2. \label{fe3}
\end{align}
From Lemmas \ref{bd0} and \ref{bd20}, we have
\begin{align}
& \quad \notag \mathbb{E}[\|z_k-\Pi_{l_k}(z_k)\|^2\mid\mathcal{F}_k] \ge \tfrac{\rho}{m\eta}\mbox{dist}^2(z_k,X) \ge \tfrac{\rho}{\theta m\eta}\left(\tfrac{1}{2}\mbox{dist}^2(x_k,X)-4\|z_k-x_k\|^2\right) \\
&\ge \tfrac{\rho}{2\theta m\eta}\mbox{dist}^2(x_k,X)-\tfrac{4\rho}{\theta m\eta}\|z_k-x_k\|^2 \ge \tfrac{\rho}{2\theta m\eta}\mbox{dist}^2(x_k,X)-4\|z_k-x_k\|^2, \label{fe4}
\end{align}
where the last inequality follows from $\tfrac{\rho}{\theta m\eta}\le1$. Substituting \eqref{fe4} into \eqref{fe3}, and taking conditional expectations, it follows that
\begin{align}
 \mathbb{E}[&\mbox{dist}^2(x_{k+1},X)\mid\mathcal{F}_k] \le \left(1-\tfrac{\rho}{4\theta m\eta}\right)\mbox{dist}^2(x_k,X)+\left(5+\tfrac{4\theta m\eta}{\rho}\right)\|z_k-x_k\|^2 \label{st2} \\
\notag&= \left(1-\tfrac{\rho}{4\theta m\eta}\right)\mbox{dist}^2(x_k,X)+\left(5+\tfrac{4\theta m\eta}{\rho}\right)\|\gamma_k\mathbb{E}[F(2x_k-x_{k-1},\omega_k)\mid\mathcal{F}_k]\|^2 \\
\notag&\le \left(1-\tfrac{\rho}{4\theta m\eta}\right)\mbox{dist}^2(x_k,X)+\left(5+\tfrac{4\theta m\eta}{\rho}\right)\gamma_k^2(2\|F(2x_k-x_{k-1})\|^2+2\mathbb{E}[\|w_k\|^2\mid\mathcal{F}_k]) \\
\notag&\le \left(1-\tfrac{\rho}{4\theta m\eta}\right)\mbox{dist}^2(x_k,X)+\left(5+\tfrac{4\theta m\eta}{\rho}\right)\gamma_k^2(4L^2\|2x_k-x_{k-1}-x^*\|^2\\\notag&+4\|F(x^*)\|^2+2\nu_1^2\|y_k\|^2+2\nu_2^2) \\
\notag&\le \left(1-\tfrac{\rho}{4\theta m\eta}\right)\mbox{dist}^2(x_k,X)+\left(5+\tfrac{4\theta m\eta}{\rho}\right)\gamma_k^2(4L^2\|2x_k-x_{k-1}-x^*\|^2\\\notag&+4\|F(x^*)\|^2+4\nu_1^2\|2x_k-x_{k-1}-x^*\|^2+4\nu_1^2\|x^*\|^2+2\nu_2^2) \\
\notag&\le \left(1-\tfrac{\rho}{4\theta m\eta}\right)\mbox{dist}^2(x_k,X)+\left(5+\tfrac{4\theta m\eta}{\rho}\right)\gamma_k^2(32(L^2+\nu_1^2)\|x_k-x^*\|^2\\&+8(L^2+\nu_1^2)\|x_{k-1}-x^*\|^2+4\|F(x^*)\|^2+4\nu_1^2\|x^*\|^2+2\nu_2^2), \label{cir1}
\end{align}
where $x^*\in X^*$. We now use the inequality 
\begin{align}
\|x_k-x^*\|^2\le2\|x_k-\Pi_X(x_k)\|^2+2\|\Pi_X(x_k)-x^*\|^2\le2\mbox{dist}^2(x_k,X)+2D_X^2. \label{dist1}
\end{align}
Therefore, we have that from \eqref{cir1},
\begin{align}
\notag\mathbb{E}[&\mbox{dist}^2(x_{k+1},X)\mid\mathcal{F}_k]\le \left(1-\tfrac{\rho}{4\theta m\eta}\right)\mbox{dist}^2(x_k,X)\\\notag&+\left(5+\tfrac{4\theta m\eta}{\rho}\right)({80(L^2+\nu_1^2)D_X^2+4\nu_1^2\|x^*\|^2+4C^2+2\nu_2^2})\gamma_k^2 \\
\notag&+64\left(5+\tfrac{4\theta m\eta}{\rho}\right)\gamma_k^2(L^2+\nu_1^2)\mbox{dist}^2(x_k,X)+16\left(5+\tfrac{4\theta m\eta}{\rho}\right)\gamma_k^2(L^2+\nu_1^2)\mbox{dist}^2(x_{k-1},X).
\end{align}
Suppose $\gamma_0^2\le\tfrac{\rho}{8\theta m\eta}/64\left(5+\tfrac{4\theta m\eta}{\rho}\right)(L^2+\nu_1^2) $. Since $\{\gamma_k\}$ is a diminishing sequence, it holds that $64\left(5+\tfrac{4\theta m\eta}{\rho}\right)\gamma_k^2(L^2+\nu_1^2)\le \tfrac{\rho}{8\theta m\eta}$. Then we have
\begin{align*}
\mathbb{E}[&\mbox{dist}^2(x_{k+1},X)\mid\mathcal{F}_k] \le \left(1-\tfrac{\rho}{8\theta m\eta}\right)\mbox{dist}^2(x_k,X)+\tfrac{\rho}{32\theta m\eta}\mbox{dist}^2(x_{k-1},X) \\&+\left(5+\tfrac{4\theta m\eta}{\rho}\right)({80(L^2+\nu_1^2)D_X^2+4\nu_1^2\|x^*\|^2+4C^2+2\nu_2^2})\gamma_k^2.
\end{align*}
Since $\tfrac{\rho}{32\theta m\eta}\le\tfrac{2}{3}$, we have $\tfrac{\rho}{32\theta m\eta}\le\left(1-\tfrac{\rho}{32\theta m\eta}\right)\tfrac{3\rho}{32\theta m\eta}$.
It follows that
\begin{align}
\notag\mathbb{E}[&\mbox{dist}^2(x_{k+1},X)+\tfrac{3\rho}{32\theta m\eta}\mbox{dist}^2(x_{k},X)\mid\mathcal{F}_k] \le \left(1-\tfrac{\rho}{32\theta m\eta}\right)\left(\mbox{dist}^2(x_k,X)+\tfrac{3\rho}{32\theta m\eta}\mbox{dist}^2(x_{k-1},X)\right) \\&+\left(5+\tfrac{4\theta m\eta}{\rho}\right)({80(L^2+\nu_1^2)D_X^2+4\nu_1^2\|x^*\|^2+4C^2+2\nu_2^2})\gamma_k^2. \label{fe314}
\end{align}
We may now invoke Lemma \ref{robbins2} \vvs{and the summability of $\gamma^2_k$} to claim $\mbox{dist}^2(x_k,X)\xrightarrow[a.s.]{k \to \infty} 0.$ \\
(b) \vvs{We begin by noting that
$\gamma^2_{k+1}\ge\left(1-\frac{\rho}{32\theta m\eta}\right)\gamma^2_k$ when $k
\geq \vvs{\bar k}$, where $\bar{k}$ is obtained} as follows when $\gamma_k = 1/k^{t/2}$.
\begin{align*}
	\tfrac{1}{({k}+1)^t} & \geq  \left(1-\tfrac{\beta}{2}\right) \tfrac{1}{{k}^t} 
\mbox{ or } \tfrac{{k}}{{k}+1} \geq \left(1-\tfrac{\beta}{2}\right)^{1/t} \\
 \tfrac{1}{{k}+1} & \leq 1- \left(1-\tfrac{\beta}{2}\right)^{1/t} \mbox{ or } 
		k \geq \bar k \triangleq  
\ceil[\bigg]{\left(\tfrac{1}{1- \left(1-\beta/2\right)^{1/t}} -1\right)}.
\end{align*}
Taking unconditional expectations on \eqref{fe314}, recalling $\gamma^2_{k+1}\ge\left(1-\frac{\rho}{32\theta m\eta}\right)\gamma^2_k$ when $k
\geq \vvs{\bar k}$, and  by leveraging Lemma \ref{wang}, we have
\begin{align}
\notag \mathbb{E}[\mbox{dist}^2(x_k,X)] & \le \underbrace{\left(\tfrac{320\theta m\eta}{\rho}+\tfrac{256\theta^2m^2\eta^2}{\rho^2}\right)({80(L^2+\nu_1^2)D_X^2+4\nu_1^2\|x^*\|^2+4C^2+2\nu_2^2})}_{\triangleq a}\gamma_k^2\\\notag&+\underbrace{\left(\mbox{dist}^2(x_0,X)+\tfrac{3\rho}{32\theta m\eta}\mbox{dist}^2(x_{-1},X)\right)}_{\triangleq b}\left(1-\tfrac{\rho}{32\theta m\eta}\right)^k\\&+\underbrace{\left(\left(5+\tfrac{4\theta m\eta}{\rho}\right)({80(L^2+\nu_1^2)D_X^2+4\nu_1^2\|x^*\|^2+4C^2+2\nu_2^2})\sum_{t=0}^{\bar{k}}\gamma^2_t\right)}_{\triangleq c}\biggl(\underbrace{1-\tfrac{\rho}{32\theta m\eta}}_{\triangleq\beta}\biggr)^{k-\bar{k}}\notag  \\
		& \leq a\gamma_k^2 + b \beta^k + c \beta^{k-\bar{k}} ,\label{eq:bd_feas1}
\end{align}

Consequently, for $k \geq \bar{k}$, \eqref{eq:bd_feas1} reduces to 
\begin{align*}
	\mathbb{E}[\mbox{dist}^2(x_k,X)] \leq \tfrac{a}{k^t} +  d\beta^k, \mbox{ where } d \triangleq  (b+ c\beta^{-\bar{k}}) .
\end{align*}
\vvs{From Lemma~\ref{lem:geom_poly},  we have that for $k \geq \bar{k}$, 
\begin{align}
	\mathbb{E}[\mbox{dist}^2(x_k,X)] \leq \tfrac{a+\bar{a}}{k^t}, \mbox{ where }	\bar a \triangleq 
 \left(\tfrac{te}{\ln(1/\beta)}\right)^t \label{st1}
\end{align}
By Jensen's inequality, we have that 
\begin{align*}
	\mathbb{E}[\mbox{dist}(x_k,X)] \leq \sqrt{\mathbb{E}[\mbox{dist}^2(x_k,X)]} \leq \tfrac{\sqrt{a+\bar{a}}}{k^{t/2}}.
\end{align*}
}
\vvs{(c) Suppose $\bar x_{K,\bar k} \triangleq \tfrac{\sum_{k=\bar k+\lfloor K/2 \rfloor }^{K+\bar k} \gamma_k x_k}{
\sum_{k=\bar k+\lfloor K/2\rfloor }^{K+\bar k} \gamma_k}$ denotes the window-based weighted average from $\bar{k}+\lfloor K/2\rfloor$ to $\bar{k}+K$. Then we have the following.  
\begin{align}
\mathbb{E}[\mbox{dist}(\bar{x}_{K,\bar k},X)]&\le \tfrac{\sum_{k=\bar k+\lfloor K/2\rfloor }^{K+\bar k}\vvs{\gamma_k}\mathbb{E}[\mbox{dist}(x_k,X)]}{\sum_{k=\bar k+\lfloor K/2\rfloor }^{K+\bar k} \gamma_k}\le\tfrac{\sum_{k=\bar k+\lfloor K/2\rfloor }^{K+\bar k}\tfrac{a+\bar a}{\vvs{k}}}{\sum_{k=\bar k+\lceil K/2\rceil }^{K+\bar k} \tfrac{1}{\sqrt{k}}} \leq  
\tfrac{(a+\bar a)(\ln(2)+1)}{\tfrac{\sqrt{K}}{2\sqrt{\bar{k}+2}}} \leq \mathcal{O}\left(\tfrac{1}{\sqrt{K}}\right).
\label{fen-1}
\end{align}
}
\end{proof}

We now provide a rate \vvs{of convergence for the iterates in terms of the gap
function expressed at a projection of the averaged sequence. Note the
difference between this result and that in the previous Section where
$\bar{x}_{K,\bar k}$ is feasible; here, the lack of feasibility requires
utilizing the average of the projection $\Pi_X(\bar{y}_{k})$ instead of the
standard weighted average. Recall that the previous result derives a rate
statement for the infeasibility. In addition, we also derive an and oracle
complexity statement for ensuring that the condition $$
\mathbb{E}[G(\bar{y}_{K,\bar k}))] \leq
\mathcal{O}\left(\tfrac{1}{\sqrt{K}}\right), \mbox{ where } \bar{y}_{K,\bar k} =\tfrac{\sum_{k=\lfloor K/2\rfloor+\bar{k}}^{K+\bar{k}}\gamma_k\Pi_X(y_k)}{\sum_{k=\lfloor K/2\rfloor+\bar{k}}^{K+\bar{k}}\gamma_k}. $$}
\begin{proposition} \label{raterpml} \em
Let Assumptions \ssc{\ref{lip-mon} -- \ref{bd5} and \ref{moment} -- \ref{rl}}
hold. Let $\gamma_k=\tfrac{1}{\sqrt{k}}$. {In addition, for any $u, v \in
X$, suppose that there exists a $D_X > 0$ such that $\|u-v\|^2 \leq D_X^2$.}
Then the following holds for any sequence generated by ({\bf r-SPRG}) in an
expected value sense, where $\bar{y}_{K,\bar k}=\tfrac{\sum_{k=\lfloor
K/2\rfloor+\bar{k}}^{K+\bar{k}}\gamma_k\Pi_X(y_k)}{\sum_{k=\lfloor
K/2\rfloor+\bar{k}}^{K+\bar{k}}\gamma_k}$. \\ (a)
$\mathbb{E}[G(\bar{y}_{K,\bar k})]\leq \mathcal{O}\left(\frac{1}{\sqrt{K}}\right)$; \\
(b) The oracle complexity  to compute an $\bar{y}_{K,\bar k}$ such that $\mathbb{E}[G(\bar{y}_{K,\bar k})]\le\epsilon$ is bounded \ic{by
$\mathcal{O}\left(\frac{1}{\epsilon^2}\right)$.}
\end{proposition}

\begin{proof}

(a) Invoking \ic{the analysis of} Lemma~\ref{bd}, \ic{without invoking \eqref{bd03}}, we obtain \begin{align}F(x)^T(x-x^*)\ge F(x^*)^T(\Pi_X(x)-x^*)-C\mbox{dist}(x,X), \quad \forall x\in\mathbb{R}^n. \label{gap2} \end{align}
Using this property in \eqref{ml-eq2i} and rewriting it with a similar manner as \eqref{ml-eq3i}, we have
\begin{align}
\notag\|x_{k+1}&-x^*\|^2\le\|x_k-x^* \|^2-\tfrac{3}{4}\|x_{k+1}-x_k \|^2-\tfrac{1}{16}\|y_k-x_{k+1}\|^2-\tfrac{1}{8}\|x_{k-1}-x_{k+1}\|^2\\\notag&+\tfrac{1}{4}\|x_k-x_{k-1}\|^2-2\gamma_kF(x^*)^T(\Pi_X(y_k)-x^*)+2\gamma_kC\mbox{dist}(y_k,X)+32\gamma_k^2L^2\|y_k-x^*\|^2 \\
\notag&+32\gamma_k^2C^2+16\gamma_k^2\|w_k\|^2-2\gamma_kw_k^T(y_k-x^*).
\end{align}
It follows that
\begin{align}
\notag&\|x_{k+1}-x^*\|^2+\tfrac{3}{4}\|x_{k+1}-x_k\|^2 \le \|x_k-x^* \|^2-2\gamma_kF(x^*)^T(\Pi_X(y_k)-x^*)\\\notag&-\tfrac{1}{16}\|y_k-x_{k+1}\|^2-\tfrac{1}{8}\|x_{k-1}-x_{k+1}\|^2+\tfrac{1}{4}\|x_k-x_{k-1}\|^2+2\gamma_kC\mbox{dist}(y_k,X) \\
\notag&+64\gamma_k^2L^2\|x_k-x^*\|^2+64\gamma_k^2L^2\|x_k-x_{k-1}\|^2+32\gamma_k^2C^2+16\gamma_k^2\|w_k\|^2-2\gamma_kw_k^T(y_k-x^*) \\
\notag&\le \|x_k-x^* \|^2+\tfrac{3}{4}\|x_k-x_{k-1}\|^2-2\gamma_kF(x^*)^T(\Pi_X(y_k)-x^*)-\tfrac{1}{2}\|x_k-x_{k-1}\|^2 \\
\notag&+2\gamma_kC\mbox{dist}(y_k,X)-\tfrac{1}{16}\|y_k-x_{k+1}\|^2+64\gamma_k^2L^2\|x_k-x^*\|^2+64\gamma_k^2L^2\|x_k-x_{k-1}\|^2\\\notag&+32\gamma_k^2C^2+16\gamma_k^2\|w_k\|^2-2\gamma_kw_k^T(y_k-x^*).
\end{align}
We have the following inequality by replacing $x^*$ with $y$:
\begin{align}
\notag&2\gamma_kF(y)^T(\Pi_X(y_k)-y)\le\|x_k-y \|^2+\tfrac{3}{4}\|x_k-x_{k-1}\|^2-(\|x_{k+1}-y\|^2+\tfrac{3}{4}\|x_{k+1}-x_k\|^2) \\\notag&+2\gamma_kC\mbox{dist}(y_k,X)-\tfrac{1}{16}\|y_k-x_{k+1}\|^2+64\gamma_k^2L^2\|x_k-y\|^2+64\gamma_k^2L^2\|x_k-x_{k-1}\|^2\\&+32\gamma_k^2C^2+16\gamma_k^2\|w_k\|^2-2\gamma_kw_k^T(y_k-y). \label{gap4}
\end{align}
We now define an auxiliary sequence $\{u_k\}$ such that 
$$ u_{k+1}:= \Pi_X (u_k - \gamma_k w_{k}), $$
where $u_0 \in X$. We may then express the last term on the right in \eqref{gap4} as follows.
\begin{align}\notag
 2\gamma_k w_{k}^T({y}-y_{k}) & = 2\gamma_k w_{k}^T(\vvs{y}-u_k) + 2\gamma_k w_{k}^T(u_k-y_{k}) \\
\label{rl-uk2}
	& \leq \|u_k-y\|^2 - \|u_{k+1}-y\|^2 + \gamma_k^2 \|w_{k}\|^2 + 2\gamma_k w_{k}^T(u_k-y_{k}). 
\end{align}
Summing over $k$ {and invoking \eqref{rl-uk2}}, we obtain the following bound:
\begin{align}
\notag&\sum_{k=\lfloor K/2\rfloor+\bar{k}}^{K+\bar{k}}2\gamma_kF(y)^T(\Pi_X(y_k)-y)\le\|x_0-y \|^2+\tfrac{3}{4}\|x_0-x_{-1}\|^2\\\notag&+\sum_{k=\lfloor K/2\rfloor+\bar{k}}^{K+\bar{k}}(2\gamma_kC\mbox{dist}(y_k,X)-\tfrac{1}{16}\|y_k-x_{k+1}\|^2) \\\notag
&+\sum_{k=\lfloor K/2\rfloor+\bar{k}}^{K+\bar{k}}(64L^2\|x_k-y\|^2+64L^2\|x_k-x_{k-1}\|^2+32C^2)\gamma_k^2+\sum_{k=\lfloor K/2\rfloor+\bar{k}}^{K+\bar{k}}17\gamma_k^2\|w_k\|^2\\&+\|u_0-y\|^2+\sum_{k=\lfloor K/2\rfloor+\bar{k}}^{K+\bar{k}}2\gamma_k w_{k}^T(u_k-y_{k}). \label{rpl-3}
\end{align}
Dividing both sides of \eqref{rpl-3} by $\sum_{k=\lfloor K/2\rfloor+\bar{k}}^{K+\bar{k}}\gamma_k$, we have
\begin{align}
\notag2 &F(y)^T(\bar{y}_{K,\bar k}-y)\le\tfrac{\|x_0-y \|^2+\tfrac{3}{4}\|x_0-x_{-1}\|^2+\|u_0-y\|^2}{\sum_{k=\lfloor K/2\rfloor+\bar{k}}^{K+\bar{k}}\gamma_k}+\tfrac{\sum_{k=\lfloor K/2\rfloor+\bar{k}}^{K+\bar{k}}\left(2\gamma_k C{\scriptsize \mbox{dist}}(y_k,X)-\tfrac{1}{16}\|y_k-x_{k+1}\|^2\right)}{\sum_{k=\lfloor K/2\rfloor+\bar{k}}^{K+\bar{k}}\gamma_k} \\
\notag&+\tfrac{\sum_{k=\lfloor K/2\rfloor+\bar{k}}^{K+\bar{k}}(64L^2\|x_k-y\|^2+64L^2\|x_k-x_{k-1}\|^2+32C^2)\gamma_k^2}{\sum_{k=\lfloor K/2\rfloor+\bar{k}}^{K+\bar{k}}\gamma_k}+\tfrac{\sum_{k=\lfloor K/2\rfloor+\bar{k}}^{K+\bar{k}}17\gamma_k^2\|w_k\|^2+\sum_{k=\lfloor K/2\rfloor+\bar{k}}^{K+\bar{k}}2\gamma_k w_{k}^T(u_k-y_{k})}{\sum_{k=\lfloor K/2\rfloor+\bar{k}}^{K+\bar{k}}\gamma_k} \\
\notag&\le\tfrac{B_1}{\sum_{k=\lfloor K/2\rfloor+\bar{k}}^{K+\bar{k}}\gamma_k}+\tfrac{\sum_{k=\lfloor K/2\rfloor+\bar{k}}^{K+\bar{k}}\left(2\gamma_k C{\scriptsize \mbox{dist}}(y_k,X)-\tfrac{1}{16}\|y_k-x_{k+1}\|^2\right)}{\sum_{k=\lfloor K/2\rfloor+\bar{k}}^{K+\bar{k}}\gamma_k} \\
\notag&+\tfrac{\sum_{k=\lfloor K/2\rfloor+\bar{k}}^{K+\bar{k}}(64L^2\|x_k-y\|^2+64L^2\|x_k-x_{k-1}\|^2+32C^2)\gamma_k^2}{\sum_{k=\lfloor K/2\rfloor+\bar{k}}^{K+\bar{k}}\gamma_k}+\tfrac{\sum_{k=\lfloor K/2\rfloor+\bar{k}}^{K+\bar{k}}17\gamma_k^2\|w_k\|^2+\sum_{k=\lfloor K/2\rfloor+\bar{k}}^{K+\bar{k}}2\gamma_k w_{k}^T(u_k-y_{k})}{\sum_{k=\lfloor K/2\rfloor+\bar{k}}^{K+\bar{k}}\gamma_k} \\
\notag&\le\tfrac{B_1}{\sum_{k=\lfloor K/2\rfloor+\bar{k}}^{K+\bar{k}}\gamma_k}+\tfrac{\sum_{k=\lfloor K/2\rfloor+\bar{k}}^{K+\bar{k}}\left(2\gamma_k C{\scriptsize \mbox{dist}}(y_k,X)-\tfrac{1}{16}\|y_k-x_{k+1}\|^2\right)}{\sum_{k=\lfloor K/2\rfloor+\bar{k}}^{K+\bar{k}}\gamma_k} \\
\notag&+\tfrac{\sum_{k=\lfloor K/2\rfloor+\bar{k}}^{K+\bar{k}}(128L^2({\scriptsize \mbox{dist}}^2(x_k,X)+D_X^2)+64L^2\|x_k-x_{k-1}\|^2+32C^2)\gamma_k^2}{\sum_{k=\lfloor K/2\rfloor+\bar{k}}^{K+\bar{k}}\gamma_k}\\\notag&+\tfrac{\sum_{k=\lfloor K/2\rfloor+\bar{k}}^{K+\bar{k}}17\gamma_k^2\|w_k\|^2+\sum_{k=\lfloor K/2\rfloor+\bar{k}}^{K+\bar{k}}2\gamma_k w_{k}^T(u_k-y_{k})}{\sum_{k=\lfloor K/2\rfloor+\bar{k}}^{K+\bar{k}}\gamma_k},
\end{align}
where $\|x_0-y \|^2+\tfrac{3}{4}\|x_0-x_{-1}\|^2+\|u_0-y\|^2\le2\|x_0-x^* \|^2+2\|x^*-y \|^2+\tfrac{3}{4}\|x_0-x_{-1}\|^2+2\|u_0-x^* \|^2+2\|x^*-y \|^2\le2\|x_0-x^* \|^2+\tfrac{3}{4}\|x_0-x_{-1}\|^2+2\|u_0-x^* \|^2+2D_X^2\triangleq B_1$. By taking supremum over $y \in X$, we obtain the following inequality:
\begin{align} \notag G(&\bar{y}_{K,\bar k}) \triangleq 
\sup_{y \in X} F(y)^T(\bar{y}_{K,\bar k}-y) \le\tfrac{B_1}{2\sum_{k=\lfloor K/2\rfloor+\bar{k}}^{K+\bar{k}}\gamma_k}+\tfrac{\sum_{k=\lfloor K/2\rfloor+\bar{k}}^{K+\bar{k}}\left(2\gamma_k C{\scriptsize \mbox{dist}}(y_k,X)-\tfrac{1}{16}\|y_k-x_{k+1}\|^2\right)}{2\sum_{k=\lfloor K/2\rfloor+\bar{k}}^{K+\bar{k}}\gamma_k}  \\
\notag&+\tfrac{\sum_{k=\lfloor K/2\rfloor+\bar{k}}^{K+\bar{k}}(128L^2({\scriptsize \mbox{dist}}^2(x_k,X)+D_X^2)+64L^2\|x_k-x_{k-1}\|^2+32C^2)\gamma_k^2}{2\sum_{k=\lfloor K/2\rfloor+\bar{k}}^{K+\bar{k}}\gamma_k}\\&+\tfrac{\sum_{k=\lfloor K/2\rfloor+\bar{k}}^{K+\bar{k}}17\gamma_k^2\|w_k\|^2+\sum_{k=\lfloor K/2\rfloor+\bar{k}}^{K+\bar{k}}2\gamma_k w_{k}^T(u_k-y_{k})}{2\sum_{k=\lfloor K/2\rfloor+\bar{k}}^{K+\bar{k}}\gamma_k}. \label{rpl-5}
\end{align}
Taking unconditional expectation and using \eqref{ml-eq6} in \eqref{rpl-5}, we have
\begin{align}
\notag \mathbb{E}[G(\bar{y}_{K,\bar k})] &\le\tfrac{B_1}{2\sum_{k=\lfloor K/2\rfloor+\bar{k}}^{K+\bar{k}}\gamma_k}+\tfrac{\sum_{k=\lfloor K/2\rfloor+\bar{k}}^{K+\bar{k}}\left(2\gamma_k C\mathbb{E}[{\scriptsize \mbox{dist}}(y_k,X)]-\tfrac{\rho}{16m\eta}\mathbb{E}[{\scriptsize \mbox{dist}}^2(y_k,X)]\right)}{2\sum_{k=\lfloor K/2\rfloor+\bar{k}}^{K+\bar{k}}\gamma_k} \\
\notag&+\tfrac{\sum_{k=\lfloor K/2\rfloor+\bar{k}}^{K+\bar{k}}(128L^2(\mathbb{E}[{\scriptsize \mbox{dist}}^2(x_k,X)]+D_X^2)+64L^2\mathbb{E}[\|x_k-x_{k-1}\|^2]+32C^2)\gamma_k^2}{2\sum_{k=\lfloor K/2\rfloor+\bar{k}}^{K+\bar{k}}\gamma_k}\\\notag&+\tfrac{\sum_{k=\lfloor K/2\rfloor+\bar{k}}^{K+\bar{k}}\left(17\gamma_k^2(32\nu_1^2\mathbb{E}[{\scriptsize \mbox{dist}}^2(x_k,X)]+8\nu_1^2\mathbb{E}[{\scriptsize \mbox{dist}}^2(x_{k-1},X)]+40\nu_1^2D_X^2+2\nu_1^2\|x^*\|^2+\nu_2^2)\right)}{2\sum_{k=\lfloor K/2\rfloor+\bar{k}}^{K+\bar{k}}\gamma_k}\\
\notag&\le\tfrac{B_1}{2\sum_{k=\lfloor K/2\rfloor+\bar{k}}^{K+\bar{k}}\gamma_k}+\tfrac{\sum_{k=\lfloor K/2\rfloor+\bar{k}}^{K+\bar{k}}\left(\tfrac{16m\eta C^2}{\rho}\gamma_k^2\right)}{2\sum_{k=\lfloor K/2\rfloor+\bar{k}}^{K+\bar{k}}\gamma_k} \\
\notag&+\tfrac{\sum_{k=\lfloor K/2\rfloor+\bar{k}}^{K+\bar{k}}(128L^2(\mathbb{E}[{\scriptsize \mbox{dist}}^2(x_k,X)]+D_X^2)+64L^2\mathbb{E}[\|x_k-x_{k-1}\|^2]+32C^2)\gamma_k^2}{2\sum_{k=\lfloor K/2\rfloor+\bar{k}}^{K+\bar{k}}\gamma_k}\\\notag&+\tfrac{\sum_{k=\lfloor K/2\rfloor+\bar{k}}^{K+\bar{k}}\left(17\gamma_k^2(32\nu_1^2\mathbb{E}[{\scriptsize \mbox{dist}}^2(x_k,X)]+8\nu_1^2\mathbb{E}[{\scriptsize \mbox{dist}}^2(x_{k-1},X)]+40\nu_1^2D_X^2+2\nu_1^2\|x^*\|^2+\nu_2^2)\right)}{2\sum_{k=\lfloor K/2\rfloor+\bar{k}}^{K+\bar{k}}\gamma_k} \\
\notag&\overset{\tiny\mbox{From }\eqref{st1}}{\le}\tfrac{B_1}{2\sum_{k=\lfloor K/2\rfloor+\bar{k}}^{K+\bar{k}}\gamma_k}+\tfrac{\tfrac{16m\eta C^2}{\rho}\sum_{k=\lfloor K/2\rfloor+\bar{k}}^{K+\bar{k}}\gamma_k^2}{2\sum_{k=\lfloor K/2\rfloor+\bar{k}}^{K+\bar{k}}\gamma_k}+\tfrac{\sum_{k=\lfloor K/2\rfloor+\bar{k}}^{K+\bar{k}}(512L^2((a+\bar{a})/\bar{k})+320D_X^2+32C^2)\gamma_k^2}{2\sum_{k=\lfloor K/2\rfloor+\bar{k}}^{K+\bar{k}}\gamma_k}\\&+\tfrac{\left(17(40\nu_1^2(\bar{b}/\bar{k})+40\nu_1^2D_X^2+2\nu_1^2\|x^*\|^2+\nu_2^2)\right)\sum_{k=\lfloor K/2\rfloor+\bar{k}}^{K+\bar{k}}\gamma_k^2}{2\sum_{k=\lfloor K/2\rfloor+\bar{k}}^{K+\bar{k}}\gamma_k}. \label{bd11}
\end{align}
We now leverage  the following lower bound on the denominator for $K
\geq 1$:
\begin{align}
& \quad\sum_{k=\lfloor K/2\rfloor+\bar{k}}^{K+\bar{k}}k^{-\frac{1}{2}}\ge\int_{K/2+\bar{k}}^{K+\bar{k}}(x+1)^{-\frac{1}{2}} \ dx=2\sqrt{K+\bar{k}+1}-2\sqrt{K/2+\bar{k}+1} 
\ge\tfrac{\sqrt{K}}{(2\sqrt{\bar{k}+2})}.\label{low-bd2}
\end{align}
Similarly an upper bound may be constructed:
\begin{align}
\vvs{\sum_{k=\lfloor K/2\rfloor+\bar{k}}^{K+\bar{k}}\gamma_k^2} = 
\sum_{k=\lfloor K/2\rfloor+\bar{k}}^{K+\bar{k}}k^{-1}\le\int_{K/2+\bar{k}}^{K+\bar{k}}x^{-1} \ dx+\tfrac{1}{\lfloor K/2\rfloor+\bar{k}}\le\log{2}+1.\label{upp-bd}
\end{align}
By substituting \eqref{low-bd2} and \eqref{upp-bd} in \eqref{bd11},
	we obtain that the following holds:
\begin{align}
\notag \mathbb{E}[G(\bar{y}_{K,\bar k})] &\le \mathcal{O}\left(\tfrac{1}{\sqrt{K}}\right).
\end{align}
(b) From (a), we know that $K_\epsilon=\mathcal{O}(1/\epsilon^2)$ to ensure that $\mathbb{E}[G(\bar{y}_{K,\bar k})]\le\mathcal{O}\left(\tfrac{1}{\sqrt{K}}\right)\le\epsilon$. It follows that
\begin{align*}
\sum_{k=1}^{K_\epsilon}1={K_\epsilon}=\mathcal{O}\left(\tfrac{1}{\epsilon^2}\right).
\end{align*}

\end{proof}

\subsection{{\bf SSE} with random projections}\label{sec:4.3}
We now proceed to provide an analogous set of statements for the SSE scheme with random projections.
\begin{proposition}\label{rpcgr} \em
Let Assumptions {\ref{lip-mon} -- \ref{rl}} hold. {Suppose the steplength
sequence $\{\gamma_k\}$ satisfies $\sum_{k=0}^{\infty} \gamma_k = \infty$, 
$\sum_{k=0}^{\infty} \gamma^2_k < \infty$}, and $\gamma_k\le
\frac{1}{2{\sqrt{L^2+2\nu_1^2}}}$. Then any sequence generated by ({\bf
r-SSE}) converges to a solution $x^*\in X$ in an a.s. sense.  
\end{proposition}
\begin{proof}
By Lemma \ref{project}(ii), we have
\begin{align}
\notag\|x_{k+1}-x^*\|^2&\le\|x_k-\gamma_k(F(x_{k+\frac{1}{2}})+w_{k+\frac{1}{2}})-x^*\|^2-\|x_k-\gamma_k(F(x_{k+\frac{1}{2}})+w_{k+\frac{1}{2}})-x_{k+1}\|^2\\
&=\|x_k-x^*\|^2-\|x_k-x_{k+1}\|^2+2\gamma_k(F(x_{k+\frac{1}{2}})+w_{k+\frac{1}{2}})^T(x^*-x_{k+1}).\label{cgr-eq1}
\end{align}
It is clear that
\begin{align}
F(x_{k+\frac{1}{2}})^T(x_{k+1}-x^*)=F(x_{k+\frac{1}{2}})^T(x_{k+1}-x_{k+\frac{1}{2}})+F(x_{k+\frac{1}{2}})^T(x_{k+\frac{1}{2}}-x^*). \label{cgr-eq2}
\end{align}
Using \eqref{cgr-eq2} in \eqref{cgr-eq1}, we obtain
\begin{align}
\notag&\|x_{k+1}-x^*\|^2=\|x_k-x^*\|^2-\|x_k-x_{k+1}\|^2+2\gamma_kF(x_{k+\frac{1}{2}})^T(x_{k+\frac{1}{2}}-x_{k+1})+2\gamma_kw_{k+\frac{1}{2}}^T(x^*-x_{k+1}) \\
\notag&-2\gamma_kF(x_{k+\frac{1}{2}})^T(x_{k+\frac{1}{2}}-x^*) \\
\notag&=\|x_k-x^*\|^2-\|x_k-x_{k+\frac{1}{2}}+x_{k+\frac{1}{2}}-x_{k+1}\|^2+2\gamma_kF(x_{k+\frac{1}{2}})^T(x_{k+\frac{1}{2}}-x_{k+1})\\
\notag&+2\gamma_kw_{k+\frac{1}{2}}^T(x^*-x_{k+1})-2\gamma_kF(x_{k+\frac{1}{2}})^T(x_{k+\frac{1}{2}}-x^*)\\
\notag&=\|x_k-x^*\|^2-\|x_k-x_{k+\frac{1}{2}}\|^2-\|x_{k+\frac{1}{2}}-x_{k+1}\|^2-2(x_k-x_{k+\frac{1}{2}})^T(x_{k+\frac{1}{2}}-x_{k+1}) \\
\notag&+2\gamma_kF(x_{k+\frac{1}{2}})^T(x_{k+\frac{1}{2}}-x_{k+1})+2\gamma_kw_{k+\frac{1}{2}}^T(x^*-x_{k+1})-2\gamma_kF(x_{k+\frac{1}{2}})^T(x_{k+\frac{1}{2}}-x^*)\\
\notag&=\|x_k-x^*\|^2-\|x_k-x_{k+\frac{1}{2}}\|^2-\|x_{k+\frac{1}{2}}-x_{k+1}\|^2+2(x_{k+1}-x_{k+\frac{1}{2}})^T(x_k-\gamma_kF(x_{k+\frac{1}{2}})-x_{k+\frac{1}{2}}) \\
&+2\gamma_kw_{k+\frac{1}{2}}^T(x^*-x_{k+1})-2\gamma_kF(x_{k+\frac{1}{2}})^T(x_{k+\frac{1}{2}}-x^*). \label{gap3}
\end{align}
Employing a similar approach as in Proposition \ref{egcgr}, we obtain that
\begin{align}
\notag\|x_{k+1}-x^*\|^2&\le\|x_k-x^*\|^2-(1-2\gamma_k^2L^2)\|x_k-x_{k+\frac{1}{2}}\|^2+2\gamma_k^2\|w_k-w_{k+\frac{1}{2}}\|^2 \\
\notag&+2\gamma_kw_{k+\frac{1}{2}}^T(x^*-x_{k+\frac{1}{2}})-2\gamma_kF(x_{k+\frac{1}{2}})^T(x_{k+\frac{1}{2}}-x^*) \\
\notag&\le\|x_k-x^*\|^2-(1-2\gamma_k^2L^2)\|x_k-x_{k+\frac{1}{2}}\|^2+2\gamma_k^2\|w_k-w_{k+\frac{1}{2}}\|^2 \\
&+2\gamma_kw_{k+\frac{1}{2}}^T(x^*-x_{k+\frac{1}{2}})-2\gamma_k\alpha\mbox{dist}\left(\Pi_X(x_{k+\frac{1}{2}}),X^*\right)+2\gamma_kC\mbox{dist}(x_{k+\frac{1}{2}},X). \label{cgr-eq4}
\end{align}
Invoking weak sharpness property, we have
\begin{align}
-2\gamma_k\alpha\mbox{dist}\left(\Pi_X(x_{k+\frac{1}{2}}),X^*\right) \le -2\gamma_k\alpha\mbox{dist}\left(x_k,X^*\right)+2\gamma_k\alpha\|x_k-x_{k+\frac{1}{2}}\|+2\gamma_k\alpha \mbox{dist}(x_{k+\frac{1}{2}},X) \label{weak1}
\end{align}
and
\begin{align}
\notag2\gamma_k(C+\alpha)\mbox{dist}(x_{k+\frac{1}{2}},X) &\le 2\gamma_k(C+\alpha)\mbox{dist}(x_k,X)+2\gamma_k(C+\alpha)\|x_k-x_{k+\frac{1}{2}}\| \\
&\le 2\gamma_k(C+\alpha)\mbox{dist}(x_k,X)+4\gamma_k^2(C+\alpha)^2+\tfrac{1}{4}\|x_k-x_{k+\frac{1}{2}}\|^2, \label{weak2}
\end{align}
Using \eqref{weak1} and \eqref{weak2} in \eqref{cgr-eq4}, we obtain
\begin{align}
\notag&\|x_{k+1}-x^*\|^2\le\|x_k-x^*\|^2-(1-2\gamma_k^2L^2)\|x_k-x_{k+\frac{1}{2}}\|^2+2\gamma_k^2\|w_k-w_{k+\frac{1}{2}}\|^2 \\
\notag&+2\gamma_kw_{k+\frac{1}{2}}^T(x^*-x_{k+\frac{1}{2}})-2\gamma_k\alpha\mbox{dist}\left(\Pi_X(x_{k+\frac{1}{2}}),X^*\right)+2\gamma_kC\mbox{dist}(x_{k+\frac{1}{2}},X) \\ 
\notag&\le \|x_k-x^*\|^2-(1-2\gamma_k^2L^2)\|x_k-x_{k+\frac{1}{2}}\|^2+2\gamma_k^2\|w_k-w_{k+\frac{1}{2}}\|^2 \\
\notag&+2\gamma_kw_{k+\frac{1}{2}}^T(x^*-x_{k+\frac{1}{2}})-2\gamma_k\alpha\mbox{dist}\left(x_k,X^*\right)+2\gamma_k\alpha\|x_k-x_{k+\frac{1}{2}}\|+2\gamma_k(C+\alpha)\mbox{dist}(x_{k+\frac{1}{2}},X) \\
\notag&\le \|x_k-x^*\|^2-(1-2\gamma_k^2L^2)\|x_k-x_{k+\frac{1}{2}}\|^2+2\gamma_k^2\|w_k-w_{k+\frac{1}{2}}\|^2 +2\gamma_kw_{k+\frac{1}{2}}^T(x^*-x_{k+\frac{1}{2}}) \\
\notag&-2\gamma_k\alpha\mbox{dist}\left(x_k,X^*\right)+2\gamma_k\alpha\|x_k-x_{k+\frac{1}{2}}\|+2\gamma_k(C+\alpha)\mbox{dist}(x_k,X)+4\gamma_k^2(C+\alpha)^2+\tfrac{1}{4}\|x_k-x_{k+\frac{1}{2}}\|^2 \\
\notag&\le \|x_k-x^*\|^2-2\gamma_k\alpha\mbox{dist}\left(x_k,X^*\right)-\left(\tfrac{5}{8}-2\gamma_k^2L^2\right)\|x_k-x_{k+\frac{1}{2}}\|^2-\tfrac{1}{8}{(\|x_k-x_{k+\frac{1}{2}}\|-8\gamma_k\alpha)}^2 \\
\notag&+8\gamma_k^2\alpha^2+4\gamma_k^2(C+\alpha)^2+2\gamma_k(C+\alpha)\mbox{dist}(x_k,X)+2\gamma_k^2\|w_{k+\frac{1}{2}}-w_k\|^2-2\gamma_kw_{k+\frac{1}{2}}^T(x_{k+\frac{1}{2}}-x^*).
\end{align}
Taking expectations conditioned on $\mathcal{F}_{k}$, we obtain
\begin{align}
\notag\mathbb{E}&[\|x_{k+1}-x^*\|^2\mid\mathcal{F}_{k}] \le \|x_k-x^*\|^2-2\gamma_k\alpha\mbox{dist}\left(x_k,X^*\right)-\left(\tfrac{5}{8}-2\gamma_k^2L^2\right)\mathbb{E}[\|x_k-x_{k+\frac{1}{2}}\|^2\mid\mathcal{F}_{k}] \\
\notag&+8\gamma_k^2\alpha^2+4\gamma_k^2(C+\alpha)^2+2\gamma_k(C+\alpha)\mbox{dist}(x_k,X)+{2\gamma_k^2(\nu_1^2\mathbb{E}[\|x_{k+\frac{1}{2}}\|^2\mid\mathcal{F}_{k}]+\nu_1^2\|x_k\|^2+2\nu_2^2)} \\
\notag&\le \|x_k-x^*\|^2-2\gamma_k\alpha\mbox{dist}\left(x_k,X^*\right)-\left(\tfrac{5}{8}-2\gamma_k^2L^2\right)\mathbb{E}[\|x_k-x_{k+\frac{1}{2}}\|^2\mid\mathcal{F}_{k}] \\
\notag&+8\gamma_k^2\alpha^2+4\gamma_k^2(C+\alpha)^2+2\gamma_k(C+\alpha)\mbox{dist}(x_k,X)\\\notag&+{2\gamma_k^2(2\nu_1^2\mathbb{E}[\|x_k-x_{k+\frac{1}{2}}\|^2\mid\mathcal{F}_{k}]+3\nu_1^2\|x_k\|^2+2\nu_2^2)}\\
\notag&\le \|x_k-x^*\|^2-2\gamma_k\alpha\mbox{dist}\left(x_k,X^*\right)-\left(\tfrac{5}{8}-2\gamma_k^2L^2\right)\mathbb{E}[\|x_k-x_{k+\frac{1}{2}}\|^2\mid\mathcal{F}_{k}] \\
\notag&+8\gamma_k^2\alpha^2+4\gamma_k^2(C+\alpha)^2+2\gamma_k(C+\alpha)\mbox{dist}(x_k,X)\\\notag&+{2\gamma_k^2(2\nu_1^2\mathbb{E}[\|x_k-x_{k+\frac{1}{2}}\|^2\mid\mathcal{F}_{k}]+6\nu_1^2\|x_k-x^*\|^2+6\nu_1^2\|x^*\|^2+2\nu_2^2)}\\
\notag&={(1+12\gamma_k^2\nu_1^2)} \|x_k-x^*\|^2-2\gamma_k\alpha\mbox{dist}\left(x_k,X^*\right)-\left(\tfrac{5}{8}-2\gamma_k^2(L^2{+2\nu_1^2})\right)\mathbb{E}[\|x_k-x_{k+\frac{1}{2}}\|^2\mid\mathcal{F}_{k}] \\
&+8\gamma_k^2\alpha^2+4\gamma_k^2(C+\alpha)^2+2\gamma_k(C+\alpha)\mbox{dist}(x_k,X)+{2\gamma_k^2(6\nu_1^2\|x^*\|^2+2\nu_2^2)}. \label{cgr-eq5}
\end{align}
According to Lemma \ref{bd20}, we have
\begin{align}
\notag\mathbb{E}[\|x_k-x_{k+\frac{1}{2}}\|^2\mid\mathcal{F}_{k}]&=\mathbb{E}[\|x_k-\Pi_{l_k}(x_k-\gamma_kF(x_k,\omega_k))\|^2\mid\mathcal{F}_{k}] \\
&\ge \mathbb{E}[\|x_k-\Pi_{l_k}(x_k)\|^2\mid\mathcal{F}_{k}] \ge \tfrac{\rho}{m\eta}\mbox{dist}^2(x_k,X). \label{cgr-eq6}
\end{align}
where the last inequality follows from Lemma \ref{bd0}.
Multiplying \eqref{cgr-eq6} by $\tfrac{1}{8}$ and using it in \eqref{cgr-eq5}, we have
\begin{align}
\notag\mathbb{E}&[\|x_{k+1}-x^*\|^2\mid\mathcal{F}_{k}] \le {(1+12\gamma_k^2\nu_1^2)}\|x_k-x^*\|^2-2\gamma_k\alpha\mbox{dist}\left(x_k,X^*\right)\\\notag&-\left(\tfrac{1}{2}-2\gamma_k^2(L^2{+2\nu_1^2})\right)\mathbb{E}[\|x_k-x_{k+\frac{1}{2}}\|^2\mid\mathcal{F}_{k}]+8\gamma_k^2\alpha^2+4\gamma_k^2(C+\alpha)^2\\\notag&+2\gamma_k(C+\alpha)\mbox{dist}(x_k,X)-\tfrac{\rho}{8m\eta}\mbox{dist}^2(x_k,X)+{2\gamma_k^2(6\nu_1^2\|x^*\|^2+2\nu_2^2)} \\
\notag&={(1+12\gamma_k^2\nu_1^2)}\|x_k-x^*\|^2-2\gamma_k\alpha\mbox{dist}\left(x_k,X^*\right)-\left(\tfrac{1}{2}-2\gamma_k^2(L^2{+2\nu_1^2})\right)\mathbb{E}[\|x_k-x_{k+\frac{1}{2}}\|^2\mid\mathcal{F}_{k}]\\\notag&+8\gamma_k^2\alpha^2+4\gamma_k^2(C+\alpha)^2-\tfrac{\rho}{8m\eta}\left(\mbox{dist}(x_k,X)-\tfrac{8m\eta\gamma_k(C+\alpha)}{\rho}\right)^2\\\notag&+\tfrac{8m\eta (C+\alpha)^2}{\rho}\gamma_k^2+{2\gamma_k^2(6\nu_1^2\|x^*\|^2+2\nu_2^2)}  \\
\notag&\le{(1+12\gamma_k^2\nu_1^2)}\|x_k-x^*\|^2-2\gamma_k\alpha\mbox{dist}\left(x_k,X^*\right)+8\gamma_k^2\alpha^2+4\gamma_k^2(C+\alpha)^2\\&+\tfrac{8m\eta (C+\alpha)^2}{\rho}\gamma_k^2+{2\gamma_k^2(6\nu_1^2\|x^*\|^2+2\nu_2^2)} \label{t-eq7}
\end{align}
\vs{Now we may invoke Lemma \ref{robbins}. It follows that $\{\|x_k-x^* \|^2\}$
is convergent in an a.s. sense and
$\sum_k 2\gamma_k\alpha\mbox{dist}\left(x_k,X^*\right)<\infty$ a.s. . We first 
show that $\mbox{dist}(x_k,X^*) \xrightarrow{k \to \infty} 0$ a.s. .  We proceed
by contradiction and assume that with finite probability,
$\mbox{dist}(x_k,X^*) \to h(\omega)>0$ for $\omega \in V$ where $\mathbb{P}(V) > 0$. Since $\sum_{k} \gamma_k = \infty$, it follows that $\sum_k \gamma_k \mbox{dist}(x_k,X^*)  = \infty$ with finite probability. But this contradicts 
$\sum{2\gamma_k\alpha\mbox{dist}\left(x_k,X^*\right)}<\infty$ a.s., implying that 
$\mbox{dist}\left(x_k,X^*\right) \to 0$ in an a.s. sense.} \vvs{In a similar fashion as in Proposition~\ref{rpml}, we may show that the entire sequence of $\{x_k\}$ is convergent \avs{to a random point in $X^*$}}.
\end{proof}

We continue with an analysis of the infeasibility sequence.
\begin{proposition}\label{rpcgr-fe} \rm
Let Assumptions \ssc{\ref{lip-mon} -- \ref{bd5}, \ref{moment} -- \ref{rl}} hold. 
Suppose $\{x_k\}$ is generated by ({\bf r-SSE}), where the projections are randomly generated. \ssc{In addition, suppose there exists a $D_X > 0$ such that $\|u-v\|^2 \leq D_X^2$ for any $u, v \in X$.} Then the following hold. \\
(a) \vvs{If $\sum_{k} \gamma_k = \infty$ and $\sum_k \gamma_k^2< \infty$}, then $\mbox{dist}(x_K,X) \xrightarrow[a.s.]{k \to \infty}0$. \\
(b) \vvs{Suppose $\gamma_k = 1/k^{t/2}$ where $t \geq 1$.} Then $\mathbb{E}[\mbox{dist}(x_k,X)] \leq \mathcal{O}\left(\frac{1}{k^{t/2}}\right)$ for any $k \geq \bar k$, where 
$$ \bar k \triangleq  \ceil[\bigg]{\left(\tfrac{1}{1- \left(1-\beta/2\right)^{1/t}} -1\right)}.$$ \\
(c)  \vvs{Suppose $\gamma_k = 1/k^{1/2}$} and $\bar x_{K,\bar k} \triangleq \tfrac{\sum_{k=\bar k+\lfloor K/2\rfloor }^{\bar k+K} \gamma_k x_k}{
\sum_{k=\bar k+\lfloor K/2\rfloor}^{K+\bar k} \gamma_k}$. Then 
 $\mathbb{E}[\mbox{dist}(\bar{x}_{K, \bar k},X)] \leq \mathcal{O}\left(\tfrac{1}{\sqrt{K}}\right)$.

\end{proposition}

\begin{proof}
(a) Let $z_k=x_k-\gamma_kF(x_{k+\frac{1}{2}},\omega_{k+\frac{1}{2}})$.
Choose $\theta \ge \max\left\{1,\tfrac{\rho}{m\eta}\right\}$. We have
\begin{align}
\notag \mbox{dist}^2(x_{k+1},X) &\le \|x_{k+1}-\Pi_X(x_{k+\frac{1}{2}})\|^2 = \|\Pi_{T_k}(z_k)-x_{k+\frac{1}{2}}+x_{k+\frac{1}{2}}-\Pi_X(x_{k+\frac{1}{2}})\|^2 \\ 
\notag &\le \left(1+\tfrac{16\theta m\eta}{\rho}\right)\|\Pi_{C_k}(z_k)-x_{k+\frac{1}{2}}\|^2+\left(1+\tfrac{\rho}{16\theta m\eta}\right)\|x_{k+\frac{1}{2}}-\Pi_X(x_{k+\frac{1}{2}})\|^2 \\
\notag &=\left(1+\tfrac{16\theta m\eta}{\rho}\right)\|\Pi_{C_k}(z_k)-\Pi_{l_k}(x_k)\|^2+\left(1+\tfrac{\rho}{16\theta m\eta}\right)\|x_{k+\frac{1}{2}}-\Pi_X(x_{k+\frac{1}{2}})\|^2 \\
\notag &=\left(1+\tfrac{16\theta m\eta}{\rho}\right)\|\Pi_{C_k}(z_k)-\Pi_{C_k}(x_k)\|^2+\left(1+\tfrac{\rho}{16\theta m\eta}\right)\|x_{k+\frac{1}{2}}-\Pi_X(x_{k+\frac{1}{2}})\|^2 \\
&\le \left(1+\tfrac{16\theta m\eta}{\rho}\right)\|z_k-x_k\|^2+\left(1+\tfrac{\rho}{16\theta m\eta}\right)\|x_{k+\frac{1}{2}}-\Pi_X(x_{k+\frac{1}{2}})\|^2, \label{fe10}
\end{align}
where we leverage $\|a+b\|^2 \le \left(1+\tfrac{16\theta m\eta}{\rho}\right)\|a\|^2+\left(1+\tfrac{\rho}{16\theta m\eta}\right)\|b\|^2$.
From \eqref{fe10}, we can deduce
\begin{align}
\notag \mathbb{E}[\mbox{dist}^2(x_{k+1},X)\mid\mathcal{F}_k] &\le \left(1+\tfrac{16\theta m\eta}{\rho}\right)\|\gamma_k\mathbb{E}[F(x_{k+\frac{1}{2}},\omega_{k+\frac{1}{2}})\mid\mathcal{F}_k]\|^2+\left(1+\tfrac{\rho}{16\theta m\eta}\right)\mathbb{E}[\mbox{dist}^2(x_{k+\frac{1}{2}},X)\mid\mathcal{F}_k] \\
\notag&\le \left(1+\tfrac{16\theta m\eta}{\rho}\right)\gamma_k^2(4(L^2+\nu_1^2)\|x_{k+\frac{1}{2}}-x^*\|^2+4\|F(x^*)\|^2+4\nu_1^2\|x^*\|^2+2\nu_2^2)\\\notag&+\left(1+\tfrac{\rho}{16\theta m\eta}\right)\mathbb{E}[\mbox{dist}^2(x_{k+\frac{1}{2}},X)\mid\mathcal{F}_k] \\
\notag&\le \left(1+\tfrac{16\theta m\eta}{\rho}\right)\gamma_k^2(8(L^2+\nu_1^2)D_X^2+4C^2+4\nu_1^2\|x^*\|^2+2\nu_2^2)\\\notag&+\left(1+\tfrac{\rho}{16\theta m\eta}+8\left(1+\tfrac{16\theta m\eta}{\rho}\right)(L^2+\nu_1^2)\gamma_k^2\right)\mathbb{E}[\mbox{dist}^2(x_{k+\frac{1}{2}},X)\mid\mathcal{F}_k].
\end{align}
Suppose $8\left(5+\tfrac{16\theta m\eta}{\rho}\right)(L^2+\nu_1^2)\gamma_k^2\le\tfrac{\rho}{16\theta m\eta}$. Then we have
\begin{align}
\notag\mathbb{E}[\mbox{dist}^2(x_{k+1},X)\mid\mathcal{F}_k] &\le\left(1+\tfrac{16\theta m\eta}{\rho}\right)\gamma_k^2(8(L^2+\nu_1^2)D_X^2+4C^2+4\nu_1^2\|x^*\|^2+2\nu_2^2)\\&+\left(1+\tfrac{\rho}{8\theta m\eta}\right)\mathbb{E}[\mbox{dist}^2(x_{k+\frac{1}{2}},X)\mid\mathcal{F}_k]. \label{fe101}
\end{align}
We can bound the second term using a similar way with \eqref{st2} as follows:
\begin{align}
\notag\mathbb{E}[\mbox{dist}^2(x_{k+\frac{1}{2}},X)&\mid\mathcal{F}_k] \le \left(1-\tfrac{\rho}{4\theta m\eta}\right)\mbox{dist}^2(x_k,X)+\left(5+\tfrac{4\theta m\eta}{\rho}\right)\|\gamma_k\mathbb{E}[F(x_k,\omega_k)\mid\mathcal{F}_k]\|^2 \\
\notag&\le \left(1-\tfrac{\rho}{4\theta m\eta}\right)\mbox{dist}^2(x_k,X)+\left(5+\tfrac{4\theta m\eta}{\rho}\right)\gamma_k^2(4(L^2+\nu_1^2)\|x_k-x^*\|^2\\&+4\|F(x^*)\|^2+4\nu_1^2\|x^*\|^2+2\nu_2^2)
. \label{fe11}
\end{align}
Using \eqref{dist1} in \eqref{fe11}, it follows that
\begin{align}
\notag\mathbb{E}[\mbox{dist}^2&(x_{k+\frac{1}{2}},X)\mid\mathcal{F}_k] \le\bigg(1-\tfrac{\rho}{4\theta m\eta}+\underbrace{8\left(5+\tfrac{4\theta m\eta}{\rho}\right)(L^2+\nu_1^2)\gamma_k^2}_{\mtiny {\le\tfrac{\rho}{16\theta m\eta}}}\bigg)\mbox{dist}^2(x_k,X)\\\notag&+\left(5+\tfrac{4\theta m\eta}{\rho}\right)\gamma_k^2(8(L^2+\nu_1^2)D_X^2+4C^2+4\nu_1^2\|x^*\|^2+2\nu_2^2) \\
&\le\left(1-\tfrac{\rho}{8\theta m\eta}\right)\mbox{dist}^2(x_k,X)+\left(5+\tfrac{4\theta m\eta}{\rho}\right)\gamma_k^2(8(L^2+\nu_1^2)D_X^2+4C^2+4\nu_1^2\|x^*\|^2+2\nu_2^2). \label{fe111}
\end{align}
Using \eqref{fe111} in \eqref{fe101}, we obtain
\begin{align}
\mathbb{E}[\mbox{dist}^2(x_{k+1},X)\mid\mathcal{F}_k] &\le \left(1-\tfrac{\rho^2}{64\theta^2m^2\eta^2}\right)\mbox{dist}^2(x_k,X)+\hat{D}\gamma_k^2, \label{fe315}
\end{align}
where $\hat{D}\triangleq \left(1+\tfrac{16\theta m\eta}{\rho}+\left(1+\tfrac{\rho}{8\theta m\eta}\right)\left(5+\tfrac{4\theta m\eta}{\rho}\right)\right)(8(L^2+\nu_1^2)D_X^2+4C^2+4\nu_1^2\|x^*\|^2+2\nu_2^2)$.
We may now invoke Lemma \ref{robbins2} \vvs{and the summability of $\gamma^2_k$} to claim $\mbox{dist}^2(x_k,X)\xrightarrow[a.s.]{k \to \infty} 0.$ \\
(b) Let $\beta\triangleq 1-\tfrac{\rho^2}{64\theta^2m^2\eta^2}$. The conclusion holds by using a similar fashion with Proposition \ref{rpml-fe}(b). \\
(c) We can derive the result using \eqref{fen-1}.
\end{proof}

We conclude this section with a rate of convergence of the gap function in terms of the projection of the averaged sequence for ({\bf r-SSE}) and the associated oracle complexity bound.

\begin{proposition}\label{raterpCGR} \em
Let Assumptions \ssc{\ref{lip-mon} -- \ref{bd5}, \ref{moment} -- \ref{rl}} hold. Let $\gamma_k=\tfrac{\gamma_0}{\sqrt{k}}$ and assume $\gamma_0\le\tfrac{1}{2\sqrt{L^2+5\nu_1^2}}$. \ssc{In addition, for any $u, v \in X$, suppose that there exists a $D_X > 0$ such that $\|u-v\|^2 \leq D_X^2$.} Then the following holds for any sequence generated by ({\bf r-SSE}) in an expected value sense, where $\bar{y}_{K,\bar k} \triangleq \tfrac{\sum_{k=\lfloor K/2\rfloor+\bar{k}}^{K+\bar{k}}\gamma_k\Pi_X(x_{k+\frac{1}{2}})}{\sum_{k=\lfloor K/2\rfloor+\bar{k}}^{K+\bar{k}}\gamma_k}$. \\
(a)
$\mathbb{E}[G(\bar{y}_{K,\bar k})]\leq \mathcal{O}\left(\frac{1}{\sqrt{K}}\right)$; \\
(b) The oracle complexity  to compute an $\bar{y}_{K,\bar k}$ such that $\mathbb{E}[\mbox{dist}(\bar{y}_{K,\bar k},X^*)]\le\epsilon$ is bounded \ic{by
$\mathcal{O}\left(\frac{1}{\epsilon^2}\right)$}.
\end{proposition}

\begin{proof}
(a) Using \eqref{gap2} in \eqref{gap3}, we obtain the following inequality which is similar with \eqref{cgr-eq4}
\begin{align}
\notag\|x_{k+1}-x^*\|^2&\le\|x_k-x^*\|^2-(1-2\gamma_k^2L^2)\|x_k-x_{k+\frac{1}{2}}\|^2+2\gamma_k^2\|w_k-w_{k+\frac{1}{2}}\|^2 \\
\notag&+2\gamma_kw_{k+\frac{1}{2}}^T(x^*-x_{k+\frac{1}{2}})-2\gamma_kF(x^*)^T(\Pi_X(x_{k+\frac{1}{2}})-x^*)+2\gamma_kC\mbox{dist}(x_{k+\frac{1}{2}},X).
\end{align}
Similarly with \eqref{gap4}, we have
\begin{align}
\notag2&\gamma_kF(y)^T(\Pi_X(x_{k+\frac{1}{2}})-y)\le\|x_k-y\|^2-\|x_{k+1}-y\|^2-(1-2\gamma_k^2L^2)\|x_k-x_{k+\frac{1}{2}}\|^2\\\notag&+2\gamma_k^2\|w_k-w_{k+\frac{1}{2}}\|^2+2\gamma_kw_{k+\frac{1}{2}}^T(y-x_{k+\frac{1}{2}})+2\gamma_kC\mbox{dist}(x_k,X)+4\gamma_k^2C^2+\tfrac{1}{4}\|x_k-x_{k+\frac{1}{2}}\|^2. \\
\notag&\le\|x_k-y\|^2-\|x_{k+1}-y\|^2-\tfrac{1}{4}\|x_k-x_{k+\frac{1}{2}}\|^2+2\gamma_k^2\|w_k-w_{k+\frac{1}{2}}\|^2\\\notag&+2\gamma_kw_{k+\frac{1}{2}}^T(y-x_{k+\frac{1}{2}})+2\gamma_kC\mbox{dist}(x_k,X)+4\gamma_k^2C^2-(\tfrac{1}{2}-2\gamma_k^2L^2)\|x_k-x_{k+\frac{1}{2}}\|^2,
\end{align}
We now define an auxiliary sequence $\{u_k\}$ such that 
$$ u_{k+1}:= \Pi_X (u_k - \gamma_k w_{k}), $$
where $u_0 \in X$. We may then express $2\gamma_kw_{k+\frac{1}{2}}^T(y-x_{k+\frac{1}{2}})$ as follows.
\begin{align}\notag
 2\gamma_k w_{k+\frac{1}{2}}^T(\vvs{y}-x_{k+\frac{1}{2}}) & = 2\gamma_k w_{k+\frac{1}{2}}^T(\vvs{y}-u_k) + 2\gamma_k w_{k+\frac{1}{2}}^T(u_k-x_{k+\frac{1}{2}}) \\
\label{rl-uk3}
	& \leq \|u_k-y\|^2 - \|u_{k+1}-y\|^2 + \gamma_k^2 \|w_{k+\frac{1}{2}}\|^2 + 2\gamma_k w_{k+\frac{1}{2}}^T(u_k-x_{k+\frac{1}{2}}). 
\end{align}
Summing over $k$ and invoking \eqref{rl-uk3}, we obtain the following bound:
\begin{align}
\notag&\sum_{k=\lfloor K/2\rfloor+\bar{k}}^{K+\bar{k}}2\gamma_kF(y)^T(\Pi_X(x_{k+\frac{1}{2}})-y)\le\|x_0-y \|^2+\sum_{k=\lfloor K/2\rfloor+\bar{k}}^{K+\bar{k}}(2\gamma_kC\mbox{dist}(x_k,X)-\tfrac{1}{4}\|x_k-x_{k+\frac{1}{2}}\|^2) \\\notag
&+\sum_{k=\lfloor K/2\rfloor+\bar{k}}^{K+\bar{k}}4C^2\gamma_k^2+\sum_{k=\lfloor K/2\rfloor+\bar{k}}^{K+\bar{k}}(4\gamma_k^2\|w_k\|^2+5\gamma_k^2\|w_{k+\frac{1}{2}}\|^2)+\|u_0-y\|^2\\\notag&+\sum_{k=\lfloor K/2\rfloor+\bar{k}}^{K+\bar{k}}2\gamma_k w_{k+\frac{1}{2}}^T(u_k-x_{k+\frac{1}{2}})-\sum_{k=\lfloor K/2\rfloor+\bar{k}}^{K+\bar{k}}(\tfrac{1}{2}-2\gamma_k^2L^2)\|x_k-x_{k+\frac{1}{2}}\|^2.
\end{align}
Dividing both sides by $\sum_{k=\lfloor K/2\rfloor+\bar{k}}^{K+\bar{k}}\gamma_k$, we have
\begin{align}
\notag2 &F(y)^T(\bar{y}_{K,\bar k}-y)\le\tfrac{B_2}{\sum_{k=\lfloor K/2\rfloor+\bar{k}}^{K+\bar{k}}\gamma_k}+\tfrac{\sum_{k=\lfloor K/2\rfloor+\bar{k}}^{K+\bar{k}}\left(2\gamma_k C{\scriptsize \mbox{dist}}(x_k,X)-\tfrac{1}{4}\|x_k-x_{k+\frac{1}{2}}\|^2\right)}{\sum_{k=\lfloor K/2\rfloor+\bar{k}}^{K+\bar{k}}\gamma_k} \\
\notag&+\tfrac{\sum_{k=\lfloor K/2\rfloor+\bar{k}}^{K+\bar{k}}4C^2\gamma_k^2}{\sum_{k=\lfloor K/2\rfloor+\bar{k}}^{K+\bar{k}}\gamma_k}+\tfrac{\sum_{k=\lfloor K/2\rfloor+\bar{k}}^{K+\bar{k}}(4\gamma_k^2\|w_k\|^2+5\gamma_k^2\|w_{k+\frac{1}{2}}\|^2)+\sum_{k=\lfloor K/2\rfloor+\bar{k}}^{K+\bar{k}}2\gamma_k w_{k+\frac{1}{2}}^T(u_k-x_{k+\frac{1}{2}})}{\sum_{k=\lfloor K/2\rfloor+\bar{k}}^{K+\bar{k}}\gamma_k} \\\notag &- \tfrac{\sum_{k=\lfloor K/2\rfloor+\bar{k}}^{K+\bar{k}}\left(\tfrac{1}{2}-2\gamma_k^2L^2\right)\|x_k-x_{k+\frac{1}{2}}\|^2}{\sum_{k=\lfloor K/2\rfloor+\bar{k}}^{K+\bar{k}}\gamma_k},
\end{align}
where $\|x_0-y \|^2+\|u_0-y\|^2\le2\|x_0-x^* \|^2+2\|x^*-y \|^2+2\|u_0-x^* \|^2+2\|x^*-y \|^2\le2\|x_0-x^* \|^2+2\|u_0-x^* \|^2+2D_X^2\triangleq B_2$.
By taking supremum over $y \in X$, we obtain the following inequality:
\begin{align} \notag G(&\bar{y}_{K,\bar k}) \triangleq 
\sup_{y \in X} F(y)^T(\bar{y}_{K,\bar k}-y) \le\tfrac{B_2}{2\sum_{k=\lfloor K/2\rfloor+\bar{k}}^{K+\bar{k}}\gamma_k}+\tfrac{\sum_{k=\lfloor K/2\rfloor+\bar{k}}^{K+\bar{k}}\left(2\gamma_k C{\scriptsize \mbox{dist}}(x_k,X)-\tfrac{1}{4}\|x_k-x_{k+\frac{1}{2}}\|^2\right)}{2\sum_{k=\lfloor K/2\rfloor+\bar{k}}^{K+\bar{k}}\gamma_k} \\
\notag&+\tfrac{\sum_{k=\lfloor K/2\rfloor+\bar{k}}^{K+\bar{k}}4C^2\gamma_k^2}{2\sum_{k=\lfloor K/2\rfloor+\bar{k}}^{K+\bar{k}}\gamma_k}+\tfrac{\sum_{k=\lfloor K/2\rfloor+\bar{k}}^{K+\bar{k}}(4\gamma_k^2\|w_k\|^2+5\gamma_k^2\|w_{k+\frac{1}{2}}\|^2)+\sum_{k=\lfloor K/2\rfloor+\bar{k}}^{K+\bar{k}}2\gamma_k w_{k+\frac{1}{2}}^T(u_k-x_{k+\frac{1}{2}})}{2\sum_{k=\lfloor K/2\rfloor+\bar{k}}^{K+\bar{k}}\gamma_k} \\ &- \tfrac{\sum_{k=\lfloor K/2\rfloor+\bar{k}}^{K+\bar{k}}\left(\tfrac{1}{2}-2\gamma_k^2L^2\right)\|x_k-x_{k+\frac{1}{2}}\|^2}{2\sum_{k=\lfloor K/2\rfloor+\bar{k}}^{K+\bar{k}}\gamma_k}. \label{rpc-5}
\end{align}
Taking unconditional expectation and using \eqref{cgr-eq6} in \eqref{rpc-5}, we have
\begin{align}
\notag \mathbb{E}[G(\bar{y}_{K,\bar k})] &\le\tfrac{B_2}{2\sum_{k=\lfloor K/2\rfloor+\bar{k}}^{K+\bar{k}}\gamma_k}+\tfrac{\sum_{k=\lfloor K/2\rfloor+\bar{k}}^{K+\bar{k}}\left(2\gamma_k C\mathbb{E}[{\scriptsize \mbox{dist}}(x_k,X)]-\tfrac{\rho}{4m\eta}\mathbb{E}[{\scriptsize \mbox{dist}}^2(x_k,X)]\right)}{2\sum_{k=\lfloor K/2\rfloor+\bar{k}}^{K+\bar{k}}\gamma_k} \\
\notag&+\tfrac{\sum_{k=\lfloor K/2\rfloor+\bar{k}}^{K+\bar{k}}4C^2\gamma_k^2}{2\sum_{k=\lfloor K/2\rfloor+\bar{k}}^{K+\bar{k}}\gamma_k}+\tfrac{\sum_{k=\lfloor K/2\rfloor+\bar{k}}^{K+\bar{k}}\left(\gamma_k^2(56\nu_1^2\mathbb{E}[{\scriptsize \mbox{dist}}^2(x_k,X)]+56\nu_1^2D_X^2+28\nu_1^2\|x^*\|^2+9\nu_2^2)\right)}{2\sum_{k=\lfloor K/2\rfloor+\bar{k}}^{K+\bar{k}}\gamma_k}\\
\notag&- \tfrac{\sum_{k=\lfloor K/2\rfloor+\bar{k}}^{K+\bar{k}}\left(\tfrac{1}{2}-2\gamma_k^2(L^2+5\nu_1^2)\right)\|x_k-x_{k+\frac{1}{2}}\|^2}{2\sum_{k=\lfloor K/2\rfloor+\bar{k}}^{K+\bar{k}}\gamma_k}\\
\notag&\le\tfrac{B_2}{2\sum_{k=\lfloor K/2\rfloor+\bar{k}}^{K+\bar{k}}\gamma_k}+\tfrac{\tfrac{4m\eta C^2}{\rho}\sum_{k=\lfloor K/2\rfloor+\bar{k}}^{K+\bar{k}}\gamma_k^2}{2\sum_{k=\lfloor K/2\rfloor+\bar{k}}^{K+\bar{k}}\gamma_k}+\tfrac{\sum_{k=\lfloor K/2\rfloor+\bar{k}}^{K+\bar{k}}4C^2\gamma_k^2}{2\sum_{k=\lfloor K/2\rfloor+\bar{k}}^{K+\bar{k}}\gamma_k}\\&+\tfrac{\left(56\nu_1^2(\bar{b}/\bar{k})+56\nu_1^2D_X^2+28\nu_1^2\|x^*\|^2+9\nu_2^2\right)\sum_{k=\lfloor K/2\rfloor+\bar{k}}^{K+\bar{k}}\gamma_k^2}{2\sum_{k=\lfloor K/2\rfloor+\bar{k}}^{K+\bar{k}}\gamma_k}. \label{rpc-6}
\end{align}
By substituting \eqref{low-bd2} and \eqref{upp-bd} in \eqref{rpc-6},
	we obtain that the following holds:
\begin{align}
\notag \mathbb{E}[G(\bar{y}_{K,\bar k})] &\le \mathcal{O}\left(\tfrac{1}{\sqrt{K}}\right).
\end{align}
\noindent (b) The result follows using the same avenue as Proposition~\ref{raterpml}(b).
\end{proof}

\begin{remark} \em Proving a.s. convergence of our randomized projection schemes
relies on imposing weak-sharpness assumptions as seen in
~\cite{iusem2018incremental}. However, gap statements do not impose such a
requirement. Since variance-reduction techniques cannot directly overcome the
impact of the infeasibility in the iterates, the resulting rate diminishes to
$\mathcal{O}(1/\sqrt{K})$, similar to that seen in the classical rate
statements for standard stochastic projection schemes for monotone stochastic
variational inequality
problems~\cite{juditsky2011solving,yousefian17smoothing}.

\end{remark}
\section{Numerical Results}
In this section, we apply the schemes on a stochastic Nash-Cournot equilibrium problem (Section~\ref{sec:5.1}) and the computation of the invariant distribution of a Markov chain (\ssc{Section~\ref{sec:5.2}}).
\subsection{A Stochastic Nash-Cournot Equilibrium Problem}\label{sec:5.1}
In this section, we present and compare the computational results of applying
the proposed schemes on a stochastic Nash-Cournot
equilibrium problem. This game is assumed that $\mathcal{I}$ firms compete over
a network of $\mathcal{J}$ nodes. Level of production and sales of firm $i \in
\mathcal{I}$ at node $j \in \mathcal{J}$ are denoted by $q_{ij}$ and $s_{ij}$,
respectively. Furthermore, we assume the cost of production at node $j$ is
$C_{ij}(q_{ij})$ and the price at node $j$ is denoted by $Q_j(\bar{s}_j,\xi) = a_j(\xi) - b_j \bar{s}_j$,
where $\bar{s}_j$ is the aggregate sales at node $j$ defined as $\bar{s}_j \triangleq \sum_{i=1}^{\mathcal{I}} s_{ij}$. For simplicity, we assume
the transportation costs are zero. Thus, each firm $i$ will solve a profit
maximization problem given by the following: \begin{align}
\notag \max \quad & \mathbb{E}\left[\sum_{j \in \mathcal{J}}(Q_j(\bar{s}_j,\xi)s_{ij}-C_{ij}(q_{ij}))\right] \\
\notag \st \quad & \sum_{j \in \mathcal{J}}q_{ij}=\sum_{j \in \mathcal{J}}s_{ij}, \quad 0 \leq q_{ij} \le \mathrm{cap}_{ij}, \quad s_{ij} \ge 0, \quad \forall j \in \mathcal{J}
\end{align}
This is an instance of a generalized Nash equilibrum problem (GNEP) with {\em
shared constraints} and a variational equilibrium (VE) (cf.~\cite{facchinei2007generalized}) of this (GNEP) given by
a solution to VI$(\mathcal{X}, F)$, where 
\begin{align*}   x & \triangleq \pmat{ s \\ q}, s \triangleq \pmat{s_{\bullet,1} \\
				   \vdots \\
				   s_{\bullet,\mathcal{J}}}, q \triangleq \pmat{q_{\bullet,1} \\
					\vdots \\
				   q_{\bullet,\mathcal{J}}}, \mbox{ and } z_{\bullet,j} \triangleq \pmat{ z_{1j} \\
			\vdots \\
		      x_{\mathcal{I}j}}, {\bf a} \triangleq \pmat{a_1 {\bf 1} \\
	\vdots \\ a_{\mathcal{J}} {\bf 1}},  \\
\notag {\bf B} & \triangleq \pmat{b_1 {\bf D} \\ & \ddots \\ & & b_{\mathcal{J}} {\bf D}}, 
  {\bf D} \triangleq (I + {\bf 1}{\bf 1}^T), c \triangleq \pmat{c_{\bullet,1} \\ \vdots \\ c_{\bullet, \mathcal{J}}}, c_{\bullet,j} \triangleq \pmat{c_{1j} \\ \vdots \\ c_{\mathcal{I}j}}, F(x)  \triangleq \pmat{ Bs - {\bf a} \\ c }, \\ \mbox{ and } 
\mathcal{X} & \triangleq \left\{ x = (s,q) \mid  
\sum_{j \in \mathcal{J}}q_{ij}=\sum_{j \in \mathcal{J}}s_{ij}, i = 1, \cdots, \mathcal{I}, \ 0 \leq q_{ij} \le \mathrm{ cap}_{ij}, \quad s_{ij} \ge 0, \  i \in \mathcal{I}, j \in \mathcal{J}\right\}. 
\notag
\end{align*}
\ssc{
Before proceeding, we prove that VI$(\mathcal{X},F)$ satisfies the required assumptions where $\mathcal{X} \subseteq \Real^n$ and $F: \Real^n \to \Real^n$:
\begin{enumerate}
\item[(i)] $F(x)$ is Lipschitz on $\Real^n$ by noting that for any $x_1, x_2 \in \mathbb{R}^n$, the following holds.   
$$ \|F(x_1) - F(x_2)\| = \|{\bf B} (s_1-s_2)\| \leq \|{\bf B} \|\|s_1-s_2\| \leq \|{\bf B}\|\|x_1-x_2\|. $$
\item[(ii)] $F(x)$ is monotone on $\Real^n$ by noting that for any $x_1, x_2 \in \mathbb{R}^n$, the following holds.   
$$ (F(x_1) - F(x_2))^T(x_1-x_2) = ({\bf B}(s_1-s_2))^T(s_1-s_2) > 0, $$
since ${\bf B} \succ 0$, a consequence of $b_j > 0$ for all $j$ and $(I+{\bf 1}{\bf 1}^T \succ 0$. Note that $F$ is merely monotone by noting that if $x_1 = ({\bf 0}, q_1)$ and $x_2 = ({\bf 0},q_2)$, we have 
$$ (F(x_1) - F(x_2))^T(x_1-x_2) = 0. $$
\item[(iii)] $\mathcal{X}$ is a compact set by recalling that $0 \leq q_{ij} \leq \mbox{cap}_{ij}$ for every $i,j$ and by observing that for any $i,j$.  
$$ s_{ij} \leq \sum_{j \in \mathcal{J}} s_{ij} =  
\sum_{j \in \mathcal{J}} q_{ij} \leq 
\sum_{j \in \mathcal{J}} \mbox{cap}_{ij}. $$
Furthermore, $\|F(x)\| \leq (\|A\|+\|c\|) \|s\| \leq (\|A\|+\|c\|) \|\mbox{cap}\|,$ for all $x$ and $\mbox{cap}$ denotes the vector of capacities over all nodes and firms. One assumption that is generally more challenging to verify is the weak-sharpness requirement. There may be approaches for claiming that such a condition holds by leveraging weak sharpness (cf.~\cite[Ch.~3]{facchinei2007finite}) and this remains a goal of future work. 
\end{enumerate}}

\noindent {\em Problem parameters.} We assume that there are
$\mathcal{I}=5$ firms and $\mathcal{J}=4$ nodes whille the capacity
$\mathrm{cap}_{ij}=300$, for all $i,j$. We assume that  $c_{ij}=1.5$ and
$d_{ij}$ is a positive constant for all $i,j$.  Furthermore,  for all $j$,
$b_j=0.05$ and $a_j(\xi) \sim U[49.5,50.5]$ where $U[a,b]$ denotes the uniform distribution on the interval $[a,b]$.\\  

\noindent {\em Algorithm parameters.} \vvs{We choose $\gamma =  0.1$ which satisfies the requirements of ({\bf v-SPRG}) and ({\bf v-SSE}) by noting that $L = 0.3$ and $\nu_1 = 0$. In addition, we choose $\gamma_k = \tfrac{\gamma}{\sqrt{k}}$ for ({\bf r-SPRG}) and ({\bf r-SSE}). Finally, in variance-reduced settings, we choose $N_k = \lfloor k^{1.1}\rfloor$ for $k \geq 0$. }


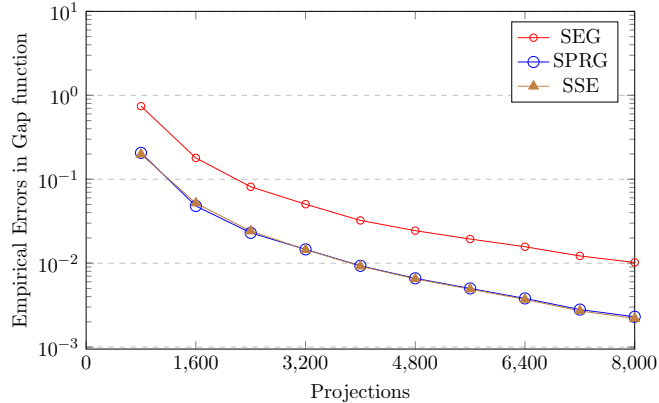
\begin{figure}[htp]
\centering
\begin{tikzpicture}[thick,scale=.7,every node/.style={scale=1}]
\begin{axis}[
    width=12cm,
    height=8cm,
    xlabel={Projections},
    ylabel={Empirical Errors in Gap function},
    ymode=log,
    xmin=0, xmax=8000,
    ymin=0, ymax=10,
    xtick={0,1600,3200,4800,6400,8000},
    ytick={0,0.0001,0.001,0.01,0.1,1,10,100,1000},
    legend pos=north east,
    ymajorgrids=true,
    grid style=dashed,
]
    
 \addplot[
    color=red,
    mark=o,
    ]
    coordinates {
    (800,0.7431)(1600,0.1794)(2400,0.0813)(3200,0.0504)(4000,0.0323)(4800,0.0244)(5600,0.0194)(6400,0.0157)(7200,0.0122)(8000,0.0102)
    };
   \addlegendentry{SEG}

   \addplot[
    color=blue,
    mark=o,
    mark size = 3pt
    ]
    coordinates {
    (800,0.2063)(1600,0.0480)(2400,0.0230)(3200,0.0146)(4000,0.0093)(4800,0.0066)(5600,0.0050)(6400,0.0038)(7200,0.0028)(8000,0.0023)
    };
    \addlegendentry{SPRG}
    
    \addplot[
    color=brown,
    mark=triangle*,
    mark size = 3pt
    ]
    coordinates {
    (800,0.1981)(1600,0.0519)(2400,0.0241)(3200,0.0145)(4000,0.0092)(4800,0.0065)(5600,0.0049)(6400,0.0037)(7200,0.0027)(8000,0.0022)
    };
    \addlegendentry{SSE}
    
\end{axis}
\end{tikzpicture}
\caption{Convergence based on projections under mere monotonicity} \label{pro}
\end{figure}

\noindent Recall that SEG requires two projections onto the set while the two \vvs{proposed}
schemes just require one. We compare their performance under the same
number of projections in Fig. \ref{pro}. 
Next we change the size and parameters of the original game to ascertain
parametric sensitivity. In Table \ref{time}, we consider 
a set of 16 problems where the settings, the empirical
errors, and elapsed time are shown in Table \ref{time}. Table~\ref{time} shows
the performance after 4000 iterations  and we observe that while SEG has almost the
same empirical error as the others but requires significantly more computational effort.
To examine the impact of variance reduction, we enlarge the random set for
random variable $a_j$ to $[40, 60]$. Fig. \ref{pro1} shows comparison of
variance reduction schemes with original ones under the same number of
iterations. Table \ref{time1} shows the results generated from different nodes
in the system. The number of iterations used is 4000. We note that all schemes
show relatively similar sensitivity to the changes introduced. \\

\noindent {\bf Key findings.} The key findings are that ({\bf v-SPRG}) and ({\bf
v-SSE}) produce comparable empirical errors to ({\bf v-SEG}) but do so in approximately $65\%$ of the time
utilized by ({\bf v-SEG}). Moreover, the presence of variance reduction allows
for significant improvement in empirical error in comparision with the single-sample counterparts (See Table \ref{time1}).

\begin{table}[h]
\tiny
\caption{Errors and elapsed time comparison of the three schemes with different parameters}
\vspace{-0.2in}
\begin{center}
    \begin{tabular}{  l | c | l | c | l | c | l }
    \hline
     & ({\bf v-SEG}) & Time & ({\bf v-SSE}) & Time & ({\bf v-SPRG}) & Time \\ \hline
    $\mathcal{I}=5,\mathcal{J}=4,c_{ij}=2,b_j=0.05$ & 9.1e-3 & 2.4e3s & 9.1e-3 & 1.6e3s & 9.2e-3 & 1.5e3s \\
    $\mathcal{I}=6,\mathcal{J}=4,c_{ij}=2,b_j=0.05$ & 1.0e-2 & 2.4e3s & 1.1e-2 & 1.6e3s & 1.1e-2 & 1.5e3s \\
    $\mathcal{I}=5,\mathcal{J}=5,c_{ij}=2,b_j=0.05$ & 1.2e-2 & 2.5e3s & 1.2e-2 & 1.8e3s & 1.2e-2 & 1.5e3s \\
    $\mathcal{I}=6,\mathcal{J}=5,c_{ij}=2,b_j=0.05$ & 1.2e-2 & 2.5e3s & 1.1e-2 & 1.9e3s & 1.3e-2 & 1.5e3s \\
    $\mathcal{I}=5,\mathcal{J}=4,c_{ij}=1,b_j=0.05$ & 9.1e-3 & 2.3e3s & 9.2e-3 & 1.7e3s & 9.3e-3 & 1.4e3s \\
    $\mathcal{I}=6,\mathcal{J}=4,c_{ij}=1,b_j=0.05$ & 1.1e-2 & 2.3e3s & 1.1e-2 & 1.8e3s & 1.1e-2 & 1.4e3s \\
    $\mathcal{I}=5,\mathcal{J}=5,c_{ij}=1,b_j=0.05$ & 1.2e-2 & 2.4e3s & 1.3e-2 & 1.8e3s & 1.3e-2 & 1.5e3s \\
    $\mathcal{I}=6,\mathcal{J}=5,c_{ij}=1,b_j=0.05$ & 1.2e-2 & 2.4e3s & 1.3e-2 & 1.9e3s & 1.3e-2 & 1.5e3s \\
    $\mathcal{I}=5,\mathcal{J}=4,c_{ij}=2,b_j=0.1$ & 1.1e-2 & 2.4e3s & 1.1e-2 & 1.6e3s & 1.2e-2 & 1.4e3s \\
    $\mathcal{I}=6,\mathcal{J}=4,c_{ij}=2,b_j=0.1$ & 1.1e-2 & 2.4e3s & 1.0e-2 & 1.6e3s & 1.1e-2 & 1.5e3s \\
    $\mathcal{I}=5,\mathcal{J}=5,c_{ij}=2,b_j=0.1$ & 1.2e-2 & 2.4e3s & 1.1e-2 & 1.7e3s & 1.2e-2 & 1.4e3s \\
    $\mathcal{I}=6,\mathcal{J}=5,c_{ij}=2,b_j=0.1$ & 1.1e-2 & 2.5e3s & 1.2e-2 & 1.8e3s & 1.3e-2 & 1.4e3s \\
    $\mathcal{I}=5,\mathcal{J}=4,c_{ij}=1,b_j=0.1$ & 1.0e-2 & 2.4e3s & 1.0e-2 & 1.7e3s & 1.1e-2 & 1.3e3s \\
    $\mathcal{I}=6,\mathcal{J}=4,c_{ij}=1,b_j=0.1$ & 1.1e-2 & 2.4e3s & 1.1e-2 & 1.6e3s & 1.1e-2 & 1.3e3s \\
    $\mathcal{I}=5,\mathcal{J}=5,c_{ij}=1,b_j=0.1$ & 1.2e-2 & 2.4e3s & 1.2e-2 & 1.8e3s & 1.1e-2 & 1.4e3s \\
    $\mathcal{I}=6,\mathcal{J}=5,c_{ij}=1,b_j=0.1$ & 1.1e-2 & 2.4e3s & 1.1e-2 & 1.7e3s & 1.2e-3 & 1.0e3s \\ \hline
    \end{tabular}
\end{center}
\label{time}
\end{table}

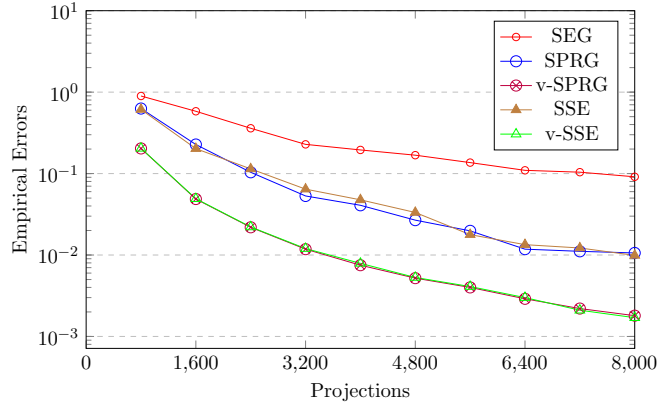
\begin{figure}[h]
\centering
\begin{tikzpicture}[thick,scale=.7,every node/.style={scale=1}]
\begin{axis}[
    width=12cm,
    height=8cm,
    xlabel={Projections},
    ylabel={Empirical Errors},
    ymode=log,
    xmin=0, xmax=8000,
    ymin=0, ymax=10,
    xtick={0,1600,3200,4800,6400,8000},
    ytick={0,0.0001,0.001,0.01,0.1,1,10,100,1000},
    legend pos=north east,
    ymajorgrids=true,
    grid style=dashed,
]
    
 \addplot[
    color=red,
    mark=o,
    ]
    coordinates {
    (800,0.8967)(1600,0.5821)(2400,0.3600)(3200,0.2284)(4000,0.1951)(4800,0.1682)(5600,0.1364)(6400,0.1096)(7200,0.1040)(8000,0.0913)
    };
   \addlegendentry{SEG}

   \addplot[
    color=blue,
    mark=o,
    mark size = 3pt
    ]
    coordinates {
    (800,0.6287)(1600,0.2274)(2400,0.1037)(3200,0.0531)(4000,0.0408)(4800,0.0268)(5600,0.0198)(6400,0.0118)(7200,0.0111)(8000,0.0106)
    };
    \addlegendentry{SPRG}
    
    \addplot[
    color=purple,
    mark=otimes,
    mark size = 3pt
    ]
    coordinates {
    (800,0.2035)(1600,0.0487)(2400,0.0219)(3200,0.0118)(4000,0.0075)(4800,0.0052)(5600,0.0040)(6400,0.0029)(7200,0.0022)(8000,0.0018)
    };
    \addlegendentry{v-SPRG}
    
    \addplot[
    color=brown,
    mark=triangle*,
    mark size = 3pt
    ]
    coordinates {
    (800,0.6127)(1600,0.2024)(2400,0.1141)(3200,0.0644)(4000,0.0476)(4800,0.0333)(5600,0.0178)(6400,0.0134)(7200,0.0122)(8000,0.0099)
    };
    \addlegendentry{SSE}
    
    \addplot[
    color=green,
    mark=triangle,
    mark size = 3pt
    ]
    coordinates {
    (800,0.2042)(1600,0.0491)(2400,0.0221)(3200,0.0120)(4000,0.0079)(4800,0.0053)(5600,0.0041)(6400,0.0030)(7200,0.0021)(8000,0.0017)
    };
    \addlegendentry{v-SSE}
    
\end{axis}
\end{tikzpicture}
\caption{Performance comparison between variance reduced schemes and original ones} \label{pro1}
\end{figure}

\begin{table}[htbp]
\tiny
\caption{Errors and elapsed time comparison of the schemes with different sizes under the same number of iterations}
\begin{center}
    \begin{tabular}{  c | l | l | l | l | l | l | l | l | l | l}
    \hline
     Network Size& SEG & Time & SSE & Time & ({\bf v-SSE}) & Time & SPRG & Time & ({\bf v-SPRG}) & Time \\ \hline
    20 & 1.0e-1 & 2.4e3s & 1.1e-1 & 1.7e3s & 7.5e-3 & 1.9e3s & 1.1e-1 & 1.5e3s & 7.4e-3 & 1.6e3s \\ 
    24 & 1.3e-1 & 2.4e3s & 1.4e-1 & 1.8e3s & 7.7e-3 & 2.0e3s & 1.3e-1 & 1.5e3s & 7.7e-3 & 1.7e3s \\ 
    28 & 1.8e-1 & 2.7e3s & 1.7e-1 & 1.9e3s & 7.9e-3 & 2.1e3s & 1.9e-1 & 1.6e3s & 8.0e-3 & 1.7e3s \\ 
    32 & 2.0e-1 & 2.8e3s & 1.9e-1 & 1.9e3s & 8.3e-3 & 2.2e3s & 2.0e-1 & 1.7e3s & 8.2e-3 & 1.8e3s \\ 
    36 & 2.5e-1 & 3.1e3s & 2.5e-1 & 2.2e3s & 8.7e-3 & 2.4e3s & 2.4e-1 & 2.0e3s & 8.8e-3 & 2.1e3s \\ 
    40 & 3.4e-1 & 3.2e3s & 3.5e-1 & 2.3e3s & 9.0e-3 & 2.5e3s & 3.5e-1 & 2.1e3s & 9.1e-3 & 2.2e3s \\ \hline
      \end{tabular}
\end{center}
\label{time1}
\end{table}

\subsection{Markov Invariant Distribution Approximation}\label{sec:5.2}

We now test the performance of the random projection schemes on an example
from~\cite{wang2015incremental} which requires  computing a low-dimensional
approximation to the invariant distribution of a Markov chain. We denote its
transition matrix by $P$ and its stationary distribution as $\pi$. The number
of states is assumed to be $1000$ and we want to approximate the states in a
low-dimensional subspace of $\mathbb{R}^{20}$ with a transformation matrix
$\Sigma$. Then we use a projection approach to approximate $\pi=P^T\pi$ as
$\Sigma x=\Pi_X(P^T\Sigma x)$, where $X \triangleq \{x \mid \Sigma x \ge 0,
e^T\Sigma x=1\}$. It has been proved
\cite{wang2015incremental,bertsekas2011temporal} that this projected equation is
equivalent to the variational inequality problem VI$(X, Sx)$.
where $S=\Sigma^T(I-P^T)\Sigma$. \vvs{Before proceeding, we verify that this problem satisfies the required assumptions.
\begin{enumerate}
\item[(i)] The mapping $Sx$ is a monotone map on $\Real^n$, since $(S+S^T)/2$ is a positive semidefinite matrix, a result that follows from $P$ being a transition matrix. In addition, $Sx$ is a Lipschitz continuous map with constant $\|S\|$.
\item[(ii)] The set $X$ is clearly compact and $\|F(x)\| \leq \|S\| \|x\|$, since $\|x\|$ is bounded on $X$. 
\end{enumerate}}

\noindent {\em Problem parameters.} \vvs{The transition matrix $P$ is randomly generated. We choose the columns of $\Sigma$ based on sinusoidal functions of various frequencies (see \cite{wang2015incremental} for details)}.\\  

\noindent {\em Algorithm parameters.} \vvs{We choose $\gamma_0 = 0.1$ which
satisfies the requirements of ({\bf v-SPRG}) and ({\bf v-SSE}) by noting that
$L = 1.0036$ (in our setting) and $\nu_1 = 0$. Finally, for ({\bf r-SPRG}) and
({\bf r-SSE}), we choose $\gamma_k$ = $\tfrac{\gamma}{\sqrt{k}}$, for
$k\ge0$. We choose $N_k = \lfloor k^{1.1}\rfloor$ in variance-reduced settings. }\\ 

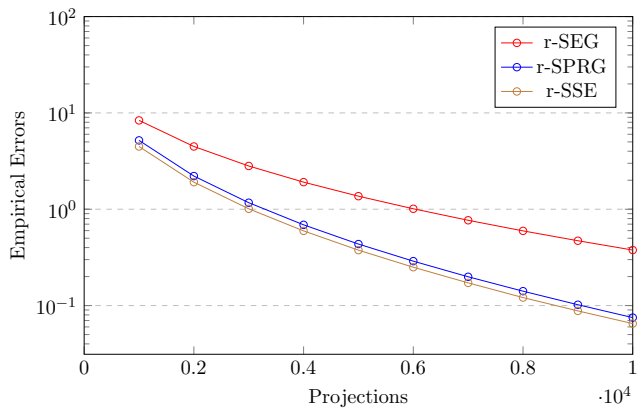
\begin{figure}[h]
\centering
\begin{tikzpicture}[thick,scale=.7,every node/.style={scale=1}]
\begin{axis}[
    width=12cm,
    height=8cm,
    xlabel={Projections},
    ylabel={Empirical Errors},
    ymode=log,
    xmin=0, xmax=10000,
    ymin=0, ymax=100,
    xtick={0,2000,4000,6000,8000,10000},
    ytick={0,0.0001,0.001,0.01,0.1,1,10,100},
    legend pos=north east,
    ymajorgrids=true,
    grid style=dashed,
]
    
 \addplot[
    color=red,
    mark=o,
    ]
    coordinates {
    (1000,8.362)(2000,4.482)(3000,2.812)(4000,1.911)(5000,1.366)(6000,1.011)(7000,0.768)(8000,0.596)(9000,0.471)(10000,0.377)
    };
   \addlegendentry{r-SEG}
   
   \addplot[
    color=blue,
    mark=o,
    ]
    coordinates {
    (1000,5.185)(2000,2.210)(3000,1.169)(4000,0.689)(5000,0.435)(6000,0.289)(7000,0.199)(8000,0.141)(9000,0.102)(10000,0.075)
    };
    \addlegendentry{r-SPRG}
   
    \addplot[
    color=brown,
    mark=o,
    ]
    coordinates {
    (1000,4.481)(2000,1.910)(3000,1.010)(4000,0.596)(5000,0.376)(6000,0.250)(7000,0.172)(8000,0.121)(9000,0.088)(10000,0.065)
    };
   \addlegendentry{r-SSE}
 
\end{axis}
\end{tikzpicture}
\caption{Empirical behavior for random projection variants}
\label{rpsite}
\end{figure}

Figure \ref{rpsite} illustrates the empirical behavior of all of the  schemes
considered.  We record the elapsed time and empirical errors of each scheme for
$10$ different transition matrices, as shown in Table \ref{rpstime} while the
comparison between the original variance-reduced schemes and their random
projection variants is shown in Table \ref{cpstime}.

\noindent {\bf Key insights.} In random projection variants, the projection onto each random constraint is cheap. Thus, the run-time benefits of ({\bf r-SSE}) are not obvious when compared with ({\bf r-SEG}) while ({\bf r-SPRG}) is still faster than others. This is because the second projection in ({\bf r-SSE}), while computable in closed form, is almost as expensive as a (cheap) projection. 
 \begin{table}[htbp]
\tiny
\caption{Comparison of schemes with differing transition matrices}
\begin{center}
    \begin{tabular}{  l | c | l | c | l | c | l  }
    \hline
     Matrix & ({\bf r-SEG}) & Time & ({\bf r-SSE}) & Time & ({\bf r-SPRG}) & Time \\ \hline
    No.1 & 7.7e-2 & 1.4e3s & 6.5e-2 &1.4e3s& 7.5e-2&0.7e3s \\ 
    No.2 & 4.0e-2 &1.3e3s& 3.9e-2 &1.4e3s& 4.0e-2&0.7e3s \\ 
    No.3 & 1.8e-2 &1.3e3s& 1.7e-2 &1.4e3s& 1.8e-2&0.7e3s \\ 
    No.4 & 5.2e-2 &1.4e3s& 4.9e-2 &1.4e3s& 5.1e-2&0.7e3s \\ 
    No.5 & 4.7e-2 &1.3e3s& 4.4e-2 &1.4e3s& 4.6e-2&0.7e3s \\ 
    No.6 & 5.9e-2 &1.3e3s& 5.5e-2 &1.4e3s& 5.8e-2&0.7e3s \\ 
    No.7 & 2.7e-2 &1.4e3s& 2.6e-2 &1.4e3s& 2.7e-2&0.7e3s \\
    No.8 & 5.8e-2 &1.3e3s& 5.3e-2 &1.4e3s& 5.7e-2&0.7e3s \\ 
    No.9 & 2.6e-2 &1.4e3s& 2.3e-2 &1.4e3s& 2.5e-2&0.7e3s \\ 
    No.10 & 3.3e-2 &1.4e3s& 3.1e-2 &1.4e3s& 3.2e-2&0.7e3s \\ \hline
    \end{tabular}
\end{center}
\label{rpstime}
\vspace{-0.2in}
\end{table}

\begin{table}[htb]
\tiny
\caption{Comparision  between original schemes and random projection variants}
\begin{center}
    \begin{tabular}{  l | c  c | c  c | c  c }
    \hline
      & ({\bf v-SEG}) & ({\bf r-SEG}) & ({\bf v-SSE}) & ({\bf r-SSE}) & ({\bf v-SRPG}) & ({\bf r-SRPG}) \\ \hline
    Error & 4.3e-3 & 7.7e-2 & 3.7e-3 & 6.5e-2 & 4.2e-3 & 7.5e-2\\ \hline
    Time & 2.8e4s & 1.4e3s & 1.6e4s & 1.4e3s & 1.5e4s & 0.7e3s\\  
     \hline
    \end{tabular}
\end{center}
\label{cpstime}
\vspace{-0.2in}
\end{table}

\section{Concluding remarks}
Extragradient schemes and their sampling-based counterparts represent a key
cornerstone of solving monotone deterministic and stochastic variational
inequality problems. Yet, the per-iteration complexity of such schemes is twice
as high as their single projection counterparts. We consider two avenues in
which the two projections are replaced by exactly one projection (a projected
reflected scheme) or a single projection onto the set and  another onto a
halfpace, the second of which is computable in closed form (a subgradient
extragradient scheme). In both instances, \vvs{under a variance-reduced
regime}, we derive a.s. convergence statements without imposing a compactness
requirement \vvs{and while allowing for state-dependent noise.}  Notably, the
sequences achieve a non-asymptotic rate of $\mathcal{O}(1/K)$ \vvs{in terms of
the expected gap function of an averaged sequence}, matching its deterministic
counterpart. Furthermore, when this set is given by the intersection of a large
number of convex sets, we develop a random projection variant for each scheme.
\vvs{Again, a.s. convergence guarantees are developed. Since the sequence of
iterates is no longer feasible, we proceed to develop rate guarantees for both
the expected infeasibility of iterates as well as the expected gap function of
a projected averaged sequence of iterates.} Empirical behavior of both schemes
show significant benefits in terms of per-iteration complexity compared to
extragradient counterparts.

\bibliographystyle{siam}
\scriptsize
\bibliography{extrasvi}


\end{document}